\documentclass[11pt,a4paper]{amsart}
\usepackage[utf8]{inputenc}
\usepackage{amsmath,amssymb,amsthm,amscd,amsfonts,mathrsfs,verbatim,stmaryrd,graphicx,mathtools}
\usepackage{float}
\usepackage{lmodern}
\usepackage{tikz}
\usepackage[normalem]{ulem}
\usetikzlibrary{cd,decorations.pathmorphing}
\usepackage{ytableau}
\usepackage{enumitem}
\usepackage{blkarray}
\usepackage{epigraph}

\usepackage[mono=false,osf]{libertine}
\usepackage{libertinust1math}
\usepackage[T1]{fontenc}
\usepackage[scaled=0.85]{beramono}
\renewcommand{\mathsf}[1]{\text{\normalfont\sffamily#1}}
\usepackage[cal=euler,scr=boondoxo,bb=libus]{mathalfa}

\usepackage{xcolor-material}
\usepackage[unicode, psdextra, colorlinks=true, linkcolor=GoogleBlue, citecolor=GoogleGreen, urlcolor=GoogleBlue, linktocpage]{hyperref}
\usepackage[noabbrev,capitalise]{cleveref}
\usepackage[margin=3cm]{geometry}
\usepackage{setspace}
\makeatletter
\newtheorem*{rep@theorem}{\rep@title}
\newcommand{\newreptheorem}[2]{%
\newenvironment{rep#1}[1]{%
 \def\rep@title{#2 \ref{##1}}%
 \begin{rep@theorem}}%
 {\end{rep@theorem}}}
\makeatother

\newtheorem{thm}{Theorem}[section]
\newtheorem{thmintro}{Theorem}

\newtheorem{cor}[thm]{Corollary}
\newreptheorem{cor}{Corollary}
\newtheorem{prop}[thm]{Proposition}
\newtheorem{lm}[thm]{Lemma}
\newtheorem{con}[thm]{Conjecture}

\theoremstyle{definition}
\newtheorem{de}[thm]{Definition}
\newtheorem{ex}[thm]{Example}
\theoremstyle{remark}
\newtheorem{rmk}[thm]{Remark}
\newtheorem{ansatz}{Ansatz}

\AddToHook{env/prop/begin}{\crefalias{thm}{prop}}
\AddToHook{env/lm/begin}{\crefalias{thm}{lm}}
\AddToHook{env/cor/begin}{\crefalias{thm}{cor}}
\AddToHook{env/con/begin}{\crefalias{thm}{con}}
\AddToHook{env/de/begin}{\crefalias{thm}{de}}
\AddToHook{env/rmk/begin}{\crefalias{thm}{rmk}}
\AddToHook{env/ex/begin}{\crefalias{thm}{ex}}

\Crefname{prop}{Proposition}{Propositions}
\Crefname{con}{Conjecture}{Conjectures}
\Crefname{ansatz}{Ansatz}{Ansätzen}

\DeclareFontFamily{U}{mathx}{}
\DeclareFontShape{U}{mathx}{m}{n}{<-> mathx10}{}
\DeclareSymbolFont{mathx}{U}{mathx}{m}{n}

\def\Z{\mathbb Z}
\def\Q{\mathbb Q}
\def\R{\mathbb R}
\def\C{\mathbb C}
\def\e{e}
\def\Irr{\mathrm{Irr}}
\def\Core{\mathcal{C}}
\def\Hitch{\mathcal{M}}
\def\Base{\mathbb{A}}
\def\Fix{\mathsf{Fix}}
\def\Hom{\mathrm{Hom}}
\def\RHom{\mathrm{RHom}}

\def\FF{\mathscr{F}}

\def\Coh{\mathrm{Coh}}
\def\Hilb{\mathrm{Hilb}}
\def\IHilb{\mathrm{IHilb}}
\def\Higgs{\mathrm{Higgs}}
\def\CHiggs{\mathcal{H}\mathrm{iggs}}
\def\LG{{\widecheck{G}}}

\def\rk{\operatorname{rk}}

\def\Spec{\operatorname{Spec}}

\def\Irr{\operatorname{Irr}}
\def\hk{hyperk\"{a}hler}

\def\AB{Atiyah--Bott}

\def\BB{Bia\l{}ynicki-Birula }

\def\T{\mathbb{T}}
\def\Part{\mathcal{P}}
\def\Fbot{\mathcal{N}}
\def\WG{\mathrm{WG}}
\def\BM{\mathrm{BM}}
\def\UF{B}
\def\loc{\mathrm{loc}}
\def\sp{\mathrm{sp}}
\def\Supp{\mathrm{Supp}}
\def\gr{\operatorname{gr}}
\def\Coh{\operatorname{Coh}}
\def\th{\theta}
\def\ulambda{{\underline{\lambda}}}
\def\KH{K\"{a}hler}
\DeclareFontFamily{U}{mathb}{\hyphenchar\font45}
\DeclareFontShape{U}{mathb}{m}{n}{
<-6> mathb5 <6-7> mathb6 <7-8> mathb7
<8-9> mathb8 <9-10> mathb9
<10-12> mathb10 <12-> mathb12
}{}
\DeclareSymbolFont{mathb}{U}{mathb}{m}{n}
\DeclareMathSymbol{\llcurly}{\mathrel}{mathb}{"CE}
\DeclareMathSymbol{\ggcurly}{\mathrel}{mathb}{"CF}

\newcommand{\simto}{\xrightarrow{\raisebox{-.8ex}[0ex][0ex]{$\sim$}}}

\newcommand{\qnom}{\genfrac{[}{]}{0pt}{}}

\DeclareMathOperator{\Sym}{Sym}
\DeclareMathOperator{\wt}{wt}
\DeclareMathOperator{\Span}{Span}

\title[Equivariant multiplicities and mirror symmetry]
{Equivariant multiplicities and mirror symmetry \\for Hilbert schemes}
\author{Alexandre Minets}
\address{A. Minets, Max Planck Institute for Mathematics, Vivatsgasse 7, 53111 Bonn, Germany}
\email{minets@mpim-bonn.mpg.de}
\curraddr{Mathematisches Institut, University of Bonn, Endenicher Allee 60, 53115 Bonn, Germany}
\email{aminets@math.uni-bonn.de}
\author{Filip Živanović} 
\address{F. T. Živanović, 
Simons Center for Geometry and Physics, 
Stony Brook, NY 11794-3636, U.S.A.}
\email{fzivanovic@scgp.stonybrook.edu} 

\begin{document}

\begin{abstract}
Following Hausel--Hitchin, we investigate core Lagrangians and upward flows in Hilbert schemes of points on elliptic surfaces. 
We compute the scheme-theoretic multiplicities of core Lagrangians, as well as the equivariant multiplicities of the very stable ones.
Furthermore, we extend the notion of equivariant multiplicity to wobbly components and compute it for Hilbert schemes of two points.
Inspired by Eisenstein series functor in Dolbeault Langlands correspondence, we propose that upward flows of very stable ideals are mirror dual to modified Procesi bundles, and justify this claim through numerical checks.
Finally, we make some conjectures about extending this picture to wobbly upward flows. 
\end{abstract}

\maketitle
\setcounter{secnumdepth}{3}
\setcounter{tocdepth}{1}

\epigraph{…er frage sich nach der Bedeutung des Satzes, demzufolge unsere Gebeine und Leiber von den Engeln dereinst übertragen werden in das Gesichtsfeld Ezechiels. Antworten habe er keine gefunden, aber es genügten ihm eigentlich auch schon die Fragen.}{\textit{Schwindel. Gefühle.\\W.G. Sebald}}

\tableofcontents 

\section{Introduction}
\subsection{Background}
Let $C$ be a smooth projective curve of genus $g>1$, and consider the moduli $\Hitch_n$ of stable Higgs bundles of rank $n$ on $C$.
Writing $\T = \mathbb{G}_m$, the space $\Hitch_n$ comes equipped with a $\T$-action by scaling the Higgs field, and $\T$-equivariant Hitchin fibration $\theta:\Hitch_n\to \Base$.
We also fix a Hitchin section $\kappa: \Base\to \Hitch_n$, and denote by $0\in \Base$ the origin.
As a consequence of \BB theory, connected components $F$ of $\Hitch_n^{\T}$ are in bijection with irreducible components $\Core_F$ of the global nilpotent cone $\Core = \theta^{-1}(0)$.

Let $\mathcal{E} = (E,\Phi)\in \Hitch_n^{\T}$ be a stable $\T$-fixed nilpotent Higgs bundle of rank $n$ and type $(1,\ldots,1)$.
In other words, $E = L_0\oplus\ldots \oplus L_{n-1}$ is a direct sum of line bundles, and the Higgs field $\Phi$ satisfies $\Phi(L_{i-1})\subset L_i\otimes K_C$.
Consider the number 
\[
\prod_{i=1}^{n-1} \binom{n}{n-i}^{\deg L_i-\deg L_{i-1}+(2g-2)}.
\]
Hausel--Hitchin proved~\cite{hausel2022very} that if the map of line bundles $\Phi^{n-1}:L_0\to L_{n-1}\otimes K_C$ has no multiple zeroes (in their language, $\mathcal{E}$ is \textit{very stable}), this number coincides with four quantities $\gamma_{\mathcal{E}}$, $l_{\mathcal{E}}$, $r_{\mathcal{E}}$, $m_{\mathcal{E}}$ associated to $(E,\Phi)$:
\begin{itemize}
    \item[($\gamma_{\mathcal{E}}$)] Number of points $(E',\Phi')$ in a general fiber of Hitchin map, satisfying $\lim\limits_{t\to 0} (E',t\Phi') = \mathcal{E}$;
    \item[($l_{\mathcal{E}}$)] Length of the sheaf $\theta_*(\mathcal{O}_{W^+_{\mathcal{E}}})|_0$, where $W^+_{\mathcal{E}}$ is the upward flow from $\mathcal{E}$;
    \item[($r_{\mathcal{E}}$)] Rank of the vector bundle $\Lambda_{\mathcal{E}} = \bigotimes_{i=1}^{n-1}\bigotimes_{j=1}^{d_i} \Lambda^{n-i}(\mathbb{E}_{c_{ij}})$\footnote{To simplify the exposition we omit the $i=0$ factor, which contributes an additional line bundle.}, where $\{c_{ij}\}$ is the set of zeroes of $\Phi|_{L_{i-1}}$, and $\mathbb{E}$ the tautological vector bundle on $\Hitch_n\times C$;
    \item[($m_{\mathcal{E}}$)] Scheme-theoretic multiplicity of the irreducible component $\Core_F$ of the global nilpotent cone corresponding to $\mathcal{E}$.
\end{itemize}
In particular, the equality $l_{\mathcal{E}} = r_{\mathcal{E}}$ is interpreted as a shadow of Donagi--Pantev's mirror symmetry for Higgs bundles~\cite{donagi2012langlands}, exchanging $\mathcal{O}_{W^+_{\mathcal{E}}}$ with $\Lambda_{\mathcal{E}}$.
For a very stable $\mathcal{E}$, the number $l_{\mathcal{E}}$ coincides with the value of an explicit rational function $\mu_{\mathcal{E}}(t)$, defined in terms of the tangent space of $\Hitch_n$ at $\mathcal{E}$, at $t=1$.
This function turns out to be a polynomial, which can be further interpreted as the image of either $\theta_*(\mathcal{O}_{W^+_{\mathcal{E}}})|_0$ or $\Lambda_{\mathcal{E}}|_{\kappa(0)}$ in $K^{\T}(\mathrm{pt})$.

Hausel--Hitchin do not explain the meaning of the function $\mu_{\mathcal{E}}(t)$ purely in terms of the irreducible component $\Core_F$.
They also define the notion of very stable Higgs bundles beyond type $(1,\ldots,1)$, but many $\T$-fixed components of $\mathcal{M}_n$ fail to have any very stable points.
This begs the following questions:
\begin{enumerate}[label={(\roman*)}, ref={(\roman*)}]
    \item What is the geometric meaning of $\mu_{\mathcal{E}}(t)$ in terms of the global nilpotent cone?\label{ques:m-meaning}
    \item How do the quantities above differ when $\mathcal{E}$ is not very stable?\label{ques:difference}
    \item What is the analogue of mirror duality between $\mathcal{O}_{W^+_{\mathcal{E}}}$ and $\Lambda_{\mathcal{E}}$ when $\mathcal{E}$ is not very stable?\label{ques:MS}
\end{enumerate}

\subsection{Equivariant multiplicities}
Instead of answering these questions for moduli of Higgs bundles, we pass to the more general setup of pre-integrable systems.
Roughly speaking, these are $\T$-equivariant proper fibrations over affine spaces $\theta:\Hitch\to \Base$, which satisfy some properties resembling Hitchin fibration; see \cref{ssec:PIS} for details.
We call $\Core = \theta^{-1}(0)$ the \textit{core} of $\Hitch$.
As before, we have a bijection between $\pi_0(\Hitch^{\T})$ and $\mathrm{Irr}(\Core)$ up to some technicalities, see \cref{prop:irr-equal-pi0} for a precise statement.
The definitions of very stable and wobbly points also immediately extend to this setting.

To address question~\labelcref{ques:m-meaning}, we introduce two kinds of equivariant multiplicities, \textit{virtual} $\mu_F(t)$ and \textit{genuine} $m_F(t)$, for a fixed component $F\in \pi_0(\Hitch^{\T})$; see \cref{em-components}. 
The virtual equivariant multiplicity is a rational function, and should be understood as a $K$-theoretic version of equivariant multiplicity of Brion--Rossmann~\cite{brion1997equivariant}.
As in~\cite{hausel2022very}, it is easy to write $\mu_F(t)$ down explicitly, but it is typically not a polynomial.
On the other hand, the genuine equivariant multiplicity is much more subtle and harder to compute.
In a sentence, it is defined by removing the contributions of irreducible components of of $\Core$ which intersect $F$.
It is always a Laurent polynomial in $t$, and it specializes to the scheme-theoretic multiplicity of $\Core_F$ at $t=1$ (\cref{prop:non-eq-mult}), while both properties can fail for $\mu_{F}(t)$.
For $F$ without very stable points, these functions are always different.

\begin{thmintro}[{\cref{prop:HH-mult,prop:v-stable-HHvsMZ}}]
    Let $F\in \pi_0(\Hitch^{\T})$ be a fixed component, and $p\in F$.
    We have $\gamma_p \leq \mu_F(1)$, and the equality holds if and only if $p$ is very stable.
    $F$ contains a very stable point if and only if $m_F(t)=\mu_F(t)$.
\end{thmintro}

We conjecture that for any components without very stable points, the inequality $\mu_F(t)<m_F(t)$ holds for all $t\geq 1$.
This holds in all cases where we managed to compute genuine equivariant multiplicities.

\subsection{Hilbert schemes}
After establishing the general theory, we restrict our attention to the case of Hilbert schemes of points over elliptic surfaces $S\to \C$.
In this case, the pre-integrable system structure is given by the natural map $\theta:\Hilb^n S\to \Hilb^n \C$.
The enormous advantage of this example is that all non-equivariant multiplicities of core components can be explicitly computed.
The $\T$-fixed components of $\Hilb^n(S)$ are parameterized by multipartitions $\underline{\lambda}$ of $n$, coloured by (a subset of) $\pi_0(S^\T)$; see \cref{cor:Hilb-fixed-loci}.
\begin{thmintro}[{\cref{cor:HilbTE-mult}, \cref{prop:vst-eq-mult-parab,prop:mult-pain}}]\label{thmA}
    Let $\underline{\lambda}$ be a multipartition as above.
    In the notations of \cref{sec:mult-for-Hilb}, we have
    \[
        \qquad m_\lambda(t) = \qnom{n}{\lambda'}_t\frac{[e]_t^{(n)}}{\prod_{p\in \Fix(S)}[w_p]_t^{(|\lambda(p)|)}}\quad\text{for $\underline{\lambda}$ very stable};\qquad m_{\underline{\lambda}} = \binom{n}{\underline{\lambda}'}\prod_{p\in \Fix(S)} m_p^{|\lambda(p)|}.
    \]
    Very stable points and components are classified by \cref{prop:par-stable-descr,cor:v-st-points}.
\end{thmintro}

Let $S = T^*E \simeq E\times \C$, where $E$ is an elliptic curve.
The fixed components of $\Hilb^n(E\times \C)$ are parameterized by partitions $\lambda \vdash n$, all of them are very stable, and $m_\lambda=\binom{n}{\lambda'}$, $m_\lambda(t) = \qnom{n}{\lambda'}_t$ by the above.
In this case, we can prove more; namely, a certain generalization of~\cite[Th.~3.10]{hausel2022enhanced}, which answers a question of T. Hausel.
\begin{repcor}{cor:mult-alg-ell}
    Let $\lambda\vdash n$, and $I\in (\Hilb^n(E\times \C))^\T$ a very stable point in the corresponding $\T$-fixed component.
    Then the equivariant multiplicity algebra $Q_I$, as defined in~\cite{hausel2022enhanced}, is isomorphic to the cohomology ring of the partial flag variety $H^*(\mathscr{F}_{\lambda'})$.
\end{repcor}

As another byproduct of \cref{thmA}, we show that the multiplicities of core components do not have to be monotone with respect to the $\T$-orbits (\cref{prop:list-tidal}).
This result does not seem to have been anticipated by the experts.

Genuine equivariant multiplicities of wobbly components are much harder to get a hold of.
We compute them for Hilbert schemes of two points with plenty of help from \texttt{macaulay2}.
As a result, we obtain many examples where $m_F(t)$ has negative coefficients. 

\begin{thmintro}
    Genuine equivariant multiplicities for $\Hilb^2 S$ are given by \eqref{eq:eq-mult-Hilb2-sep-easy}, \eqref{eq:eq-mult-Hilb2-sep}, and \cref{table:eqmultH2}.
\end{thmintro}

\subsection{Mirror symmetry}
Hilbert schemes of points on certain elliptic surfaces can be interpreted as moduli of parabolic Higgs bundles~\cite{groechenig2014hilbert}.
This allows us to put Hilbert schemes into the framework of Donagi--Pantev's mirror symmetry, and ask about mirror duals of sheaves $\mathcal{W}_I \coloneqq \mathcal{O}_{W^+_I}$.
Our proposed answer is \textit{modified Procesi}\footnote{In the main text, we use the neologism ``Procesque'' instead.} bundles $\mathcal{P}_I$.
These are defined as partial symmetrizations, via the medium of isospectral Hilbert scheme $\IHilb^n S$, of tautological bundles on $S$; see \cref{ssec:Procesque} for details.
While Hausel--Hitchin motivate their guess by compatibility of mirror symmetry with Hecke operators, our guess arises from compatibility with parabolic induction.
We check that $\mathcal{W}_I$ is exchanged with $\mathcal{P}_I$ over the general locus of $\Hilb^n \C$, where mirror symmetry restricts to a Fourier--Mukai transform.
More importantly, we verify that their equivariant pairings agree, which is a highly non-generic check.

\begin{thmintro}[{\cref{thm:eq-index}}]
    Denote $\chi_t(\mathcal{F}_1,\mathcal{F}_2) \coloneqq \sum (-1)^k \chi_t(\mathrm{Ext}^k(\mathcal{F}_1,\mathcal{F}_2))$.
    Let $S$ be a surface of parabolic family (see \cref{ex:2d-int}).
    For any two very stable ideals $I,J\in \Hilb^n S$, we have 
    \[
        \chi_t(\mathcal{W}^\vee_I,\mathcal{P}_J) = \chi_t(\mathcal{P}^\vee_I,\mathcal{W}_J).
    \]
\end{thmintro}
We compute these pairings explicitly.
Compared to the analogous check in~\cite{hausel2022very}, a new behaviour emerges: $\chi_t(\mathcal{W}^\vee_I,\mathcal{P}_J)$ cannot be expressed purely in terms of $m_I(t)$ and $m_J(t)$.

\begin{rmk}
The sheaves $\mathcal{W}_I$ are by construction \textit{BAA-branes}, that is flat unitary vector bundles supported on holomorphic Lagrangians.
In line with Kapustin--Witten~\cite{KW}, we expect the modified Procesi bundles $\mathcal{P}_I$ to be \textit{BBB-branes}, that is hyperholomorphic vector bundles on $\Hilb^n(S)$. 
This is supported by the recent work of Markman~\cite{markman2024rational}, where he proved hyperholomorphicity of closely related bundles on {\hk} varieties of $K3^{[n]}$ type.
\end{rmk}

The definition of modified Procesi bundles $\mathcal{P}_I$ can be naturally extended to \textit{any} $\T$-fixed ideal $I$, and not just a very stable one.
In this case, the mirror $\mathcal{W}_I$ is not known, but we propose how to construct a candidate (\cref{rem:dual-Procesi}); in particular, we predict that the mirror of the usual Procesi bundle $\mathcal{P}_n$ is given by
\[
    \mathcal{P}_n^\vee = \rho_*\mathcal{O}_{\kappa(\C)^n\times_{S^n}\IHilb^n S},\qquad \rho : \IHilb^n S \to \Hilb^n S.
\]
To our knowledge, this is the first precise conjecture addressing question~\labelcref{ques:MS} in the literature.
Based on explicit computations of fibers of $\mathcal{P}_I$, we also conjecture a certain relation between the sheaves $\mathcal{W}_I$ and the core components $\Core_F$ (\cref{conj:reckless}).
While the actual statement is involved, and recklessly formulated for general integrable systems, it should in particular imply the inequalities
\[
    \mu_I(t) \leq \chi_t(\theta_*\mathcal{W}_I|_0) = \chi_t(\mathcal{P}_I|_{\kappa(0)}) \leq m_I(t)
\]
for all $t>1$, and $I$ wobbly.
This provides a conjectural answer to question~\labelcref{ques:difference}.

\subsection{Organization}
In \cref{sec:core-lag} we introduce the notion of pre-integrable systems, recall stratifications arising from \BB theory, and prove some basic results about very stable points.
In \cref{sec:mult} we introduce virtual and genuine equivariant multiplicities in great generality, and study their properties.
We restrict our attention to Hilbert schemes in \cref{sec:mult-for-Hilb}; there, we classify very stable components, and compute non-equivariant multiplicities for all components.
We also explicitly compute the flow order on fixed components.
This is leveraged in \cref{sec:coh-filtr}, where we show that multiplicity is rarely compatible with this order.
On top of that, we point out some curious numerical coincidences with perverse filtration. 
\cref{sec:ex-Hilb2} is dedicated to the study of genuine equivariant multiplicities for Hilbert schemes of two points, which we compute essentially by hand and/or computer.
Finally, in \cref{sec:MS} we recall the basics of mirror symmetry/Dolbeault Langlands of Donagi--Pantev, extract the prediction for mirror duals of upward flows on parabolic Hilbert schemes, and check it by computing equivariant pairings.
We conclude with a speculation on how to extend this duality beyond very stable points, and eventually beyond Hilbert schemes.

\subsection*{Acknowledgments}
It is our pleasure to thank Tamás Hausel, Tony Pantev, Sam Raskin and Junliang Shen for illuminating discussions.
A.M. would like to thank the organizers of QAMGAST workshop at SIMIS, where this work was presented; the time spent preparing for the talk heavily influenced the shape of \cref{sec:MS}.
The authors are also grateful to Max Planck Institute for Mathematics in Bonn and Simons Center for Geometry and Physics for their hospitality and financial support.

\section{Core Lagrangians}\label{sec:core-lag}
In this section, we set up some useful axiomatics describing the spaces we want to consider in the sequel.
\subsection{Pre-integrable systems}\label{ssec:PIS}
Let ${\T} \coloneqq \mathbb{G}_{m}$.
We denote a $1$-dimensional vector space with ${\T}$-action of weight $i$ by $\mathbb{C}_i$.
Given a $\T$-vector space $V$, we denote its Hilbert series by $\chi_t(V)$:
\[  
    \chi_t(V) \coloneqq \sum_{i\in \mathbb{Z}} \dim \Hom_\T(\mathbb{C}_i, V) t^i.
\]
Writing $\Sym V \coloneqq \bigoplus_{k\geq 0} \Sym^k V$ for a positively graded $V$, we have 
\[
    \chi_t(\Sym \mathbb{C}_i) = \frac{1}{1-t^i},\qquad \chi_t(\Sym (V\oplus W)) = \chi_t(\Sym V)\chi_t(\Sym W).
\]
We will often use quantum numbers to express such series:
\[
    [n]_t \coloneqq \frac{1-t^n}{1-t} = 1+t+\ldots + t^{n-1}.
\]
For instance, $\chi_t(\Sym \C_i) = [i]_t \chi_t(\Sym \C_1)$ for $i>0$.

The following definition is not standard.
It would perhaps be more honest to carry around the adjectives ``proper, graded'', but we sacrifice them on the altar of brevity.
\begin{de}\label{de: pre-integrable system}
    A \textit{(2n-dimensional) pre-integrable system} of weight $k>0$ is a map $\theta: \Hitch\to \Base$, where:
    \begin{enumerate}
        \item $\Hitch$ is a normal connected quasi-projective $2n$-dimensional complex algebraic variety, equipped with a (holomorphic) symplectic form\footnote{Defined and non-degenerate on the smooth locus of $\Hitch$.} $\omega$, and an action ${\T}\curvearrowright \Hitch$, which scales $\omega$ by a positive weight $k$;
        \item $\Base = \bigoplus_{i=1}^n \mathbb{C}_{e_i}$ is an $n$-dimensional $\T$-vector space, where $0<e_1\leq\ldots\leq e_n$;
        \item The map $\theta$ is surjective, proper and ${\T}$-equivariant.
    \end{enumerate}
    We call $\Core \coloneqq \theta^{-1}(0)\subset \Hitch$ the \textit{core} of $\theta$.
    A pre-integrable system $\theta$ is \textit{smooth} if $\Hitch$ is smooth, and \textit{connected} if each fiber $\theta^{-1}(a)$, $a\in\Base$ is.
    We say that $\theta$ is an \textit{integrable system} if the symplectic form vanishes along the smooth locus of each fiber $\theta^{-1}(a)$.
\end{de}

It follows from the definitions that any 2-dimensional pre-integrable system is an integrable system.
While the central fiber of a pre-integrable system is always Lagrangian (see \cref{prop:flows-Lagr}), it is not clear to the authors whether any pre-integrable system is automatically integrable.
While all examples we are interested in are integrable, we will not be interested in this condition until \cref{sec:MS}.
Note also that any integrable system satisfies the following property, see e.g.~\cite[Lm.~1.5]{huybrechts2022lagrangian}: 
\begin{equation*}\tag{AB}\label{eq:ab-cond}
    \begin{aligned}
        &\text{Generically in $\Base$, $\theta$ is a fibration in abelian varieties.}
    \end{aligned}
\end{equation*}

\begin{ex}\label{ex:2d-int}
    We have two families of $2$-dimensional integrable systems:
    \begin{enumerate}
        \item \label{parabolic2dimSpaces} (Parabolic family) Let $E$ be an elliptic curve with a faithful action of a finite group $\Gamma = \mathbb{Z}/\e$, $\e\in\{1,2,3,4,6\}$ by automorphisms.
        The quotient $(T^*E)/\Gamma$ has orbifold singularities and so admits a crepant resolution $S_\Gamma$. Using the isomorphism $T^*E \simeq E\times \mathbb{C}$, we obtain a map
        \[
            \theta: S_\Gamma \to T^*E/\Gamma \to \mathbb{C}/\Gamma \simeq \mathbb{C}.
        \] 
        The weight-$1$ ${\T}$-action on $T^*E$ by scaling the cotangent direction induces a ${\T}$-action on $S_\Gamma$. The latter admits a symplectic form of ${\T}$-weight $1$. Equipping $\C$ with weight-$\e$ action, the map $\theta: S_\Gamma\to \C$ becomes equivariant. We have $\theta^{-1}(v)\simeq E$ for each $v\neq 0$, and $\theta^{-1}(0)$ is a collection of $(-2)$-projective lines, intersecting transversally (unless $\e=1$, when $\theta^{-1}(0)=E$).
        \item \label{Painleve2dimSpaces} (Painlevé family) The map $S_\Gamma\to \mathbb{C}$ can be naturally compactified:
        \[
            \overline{S_\Gamma} \coloneqq (S_\Gamma \times\mathbb{C}\setminus \theta^{-1}(0)\times\{0\})/\mathbb{G}_{m} \to \mathbb{P}^1; 
        \]
        here $\mathbb{G}_{m}$ acts on $S_\Gamma$ via ${\T}$, and on $\mathbb{C}$ with weight $\e$.
        The complement $\overline{S_\Gamma}\setminus \theta^{-1}(0)\to \mathbb{P}^1\setminus\{0\}$ inherits a ${\T}$-action, and can be shown to possess a symplectic form of weight $\e-1$.
        One denotes this integrable system by $S^{\mathrm{I}}$, resp. $S^{\mathrm{II}}$, $S^{\mathrm{IV}}$, $S^{\mathrm{VI}}$ for $\e=6$, resp. $4$, $3$, $2$.
        The singular fiber of $S^{\mathrm{I}}$, resp. $S^{\mathrm{II}}$, $S^{\mathrm{IV}}$ is a $\mathbb{P}^1$ with a cuspidal point, resp. two $\mathbb{P}^1$'s meeting at a double point, three $\mathbb{P}^1$'s meeting transversally at a single point.
    \end{enumerate}
    One can check that $S_{\mathbb{Z}/2} \simeq S^{\mathrm{VI}}$.
    The combinatorics of these integrable systems are summarized in~\cref{table:surfaces}.
    This is the full list of smooth connected $2$-dimensional integrable systems~\cite[Prop.~3.15]{SZZ}.
    \begin{table}[!htbp]
        \[
        \begin{array}[c]{|c||c|c||c|c|c||c|c|c|}
            \hline
            \e & \textup{Equation} & \textup{$\Gamma$-action} & \textup{Parab.} & \textup{Core} & k & \textup{Painlevé} & \textup{Core} & k\\ \hline 
            1 & \textup{any} & - & T^*E & E & 1 & T^*E & E & 1 \\ \hline 
            2 & \textup{any} & p\mapsto -p & S_{\mathbb{Z}/2} & 
            \tikz[thick,xscale=.3,yscale=.3,baseline]{
                \draw (0,0) -- (3,0);
                \draw (0.3,-0.5) -- (0.3,1.5);
                \draw (1.1,-0.5) -- (1.1,1.5);
                \draw (1.9,-0.5) -- (1.9,1.5);
                \draw (2.7,-0.5) -- (2.7,1.5);
                \draw [draw opacity=0] (0,1.7) -- (3,-0.7);
		    }
            & 1 & S^{\mathrm{VI}} & 
            \tikz[thick,xscale=.3,yscale=.3,baseline]{
                \draw (0,0) -- (3,0);
                \draw (0.3,-0.5) -- (0.3,1.5);
                \draw (1.1,-0.5) -- (1.1,1.5);
                \draw (1.9,-0.5) -- (1.9,1.5);
                \draw (2.7,-0.5) -- (2.7,1.5);
                \draw [draw opacity=0] (0,1.7) -- (3,-0.7);
		    } & 1
             \\ \hline 
            3 & y^2 + y = x^3 & (x,y)\mapsto (\zeta_3 x,y) & S_{\mathbb{Z}/3} & 
            \tikz[thick,xscale=.3,yscale=.3,baseline]{
                \draw (0,0) -- (3,0);
                \draw (0.5,-0.5) -- (0.5,1.5);
                \draw (1.5,-0.5) -- (1.5,1.5);
                \draw (2.5,-0.5) -- (2.5,1.5);
                \draw (0.15,1) -- (0.85,1);
                \draw (1.15,1) -- (1.85,1);
                \draw (2.15,1) -- (2.85,1);
                \draw [draw opacity=0] (0,1.7) -- (3,-0.7);
		    }
            & 1 & S^{\mathrm{IV}} & 
            \tikz[thick,xscale=.3,yscale=.3,baseline]{
                \draw (0,0.5) -- (3,0.5);
                \draw (0.6,-0.5) -- (2.4,1.5);
                \draw (2.4,-0.5) -- (0.6,1.5);
                \draw [draw opacity=0] (0,1.7) -- (3,-0.7);
		    } & 2
             \\ \hline 
            4 & y^2 = x^3+x & (x,y)\mapsto (-x,iy) & S_{\mathbb{Z}/4} & 
            \tikz[thick,xscale=.3,yscale=.3,baseline]{
                \draw (0,0) -- (3,0);
                \draw (1,-0.5) -- (1,1.5);
                \draw (0.3,0.2) -- (0.3,1.5);
                \draw (0,0.5) -- (1.25,0.5);
                \draw (0,1.2) -- (0.7,1.2);
                \draw (1.5,-0.5) -- (1.5,1.5);
                \draw (2,-0.5) -- (2,1.5);
                \draw (2.7,0.2) -- (2.7,1.5);
                \draw (1.75,0.5) -- (3,0.5);
                \draw (2.3,1.2) -- (3,1.2);
                \draw [draw opacity=0] (0,1.7) -- (3,-0.7);
		    }
             & 1 & S^{\mathrm{II}} & 
             \tikz[thick,xscale=.3,yscale=.3,baseline]{
                \draw (0,-0.5) .. controls (0.9,0.8) and (2.1,0.8) .. (3,-0.5);
                \draw (0,1.5) .. controls (0.9,0.2) and (2.1,0.2) .. (3,1.5);
                \draw [draw opacity=0] (0,1.7) -- (3,-0.7);
		    } & 3
             \\ \hline 
            6 & y^2 + y = x^3 & (x,y)\mapsto -(\zeta_3 x,y) & S_{\mathbb{Z}/6} & 
            \tikz[thick,xscale=.3,yscale=.3,baseline]{
                \draw (0,0) -- (3,0);
                \draw (0.35,-0.5) -- (0.35,1.5);
                \draw (0,0.9) -- (0.7,0.9);
                \draw (0.9,-0.5) -- (0.9,1.5);
                \draw (1.5,-0.5) -- (1.5,1.5);
                \draw (2,0.2) -- (2,1.5);
                \draw (2.5,0.2) -- (2.5,1.5);
                \draw (1.3,0.5) -- (2.2,0.5);
                \draw (1.8,1.2) -- (2.6,1.2);
                \draw (2.3,0.5) -- (3,0.5);
                \draw [draw opacity=0] (0,1.7) -- (3,-0.7);
		    }
             & 1 & S^{\mathrm{I}} & 
             \tikz[thick,xscale=.3,yscale=.3,baseline]{
                \draw (0.5,0.5) .. controls (1.5,0.3) .. (2.5,-0.5);
                \draw (0.5,0.5) .. controls (1.5,0.7) .. (2.5,1.5);
                \draw [draw opacity=0] (0,1.7) -- (3,-0.7);
		    } & 5
              \\ \hline 
        \end{array}
        \]
        \caption{$2$-dimensional integrable systems}
        \label{table:surfaces}
    \end{table}
\end{ex}

\begin{ex}\label{ex:Hilb}
    Let $\theta: S\to \C$ be a smooth two-dimensional integrable system as in \cref{ex:2d-int}, and consider the Hilbert scheme of $n$ points $\Hilb^n S$.
    It is symplectic by a theorem of Beauville~\cite[Th.~1.17]{Nak99}, with symplectic form having the same positive ${\T}$-weight as the one on $S$. 
    Furthermore, the map $\theta$ induces a proper surjection $\Hilb^n S\to \Hilb^n\mathbb{C} \simeq \mathbb{C}^n$.
    Since ${\T}$ acts on $\mathbb{C}^n$ with weights $\e,2\e,\ldots,n\e$ (where $\e$ is the $\T$-weight of the codomain $\C$ of $\theta$),
    this is a smooth $2n$-dimensional pre-integrable system.
    All fibers have correct dimension, and moreover Lagrangian, since the symplectic form on $\Hilb^nS$ is generically induced by the form on $S^n$. 
    Hence $\Hilb^n S\to \mathbb{C}^n$ is integrable.
\end{ex}

\begin{ex}
    Let $C$ be a smooth projective curve of genus $g>1$, and $K_C$ the canonical line bundle on $C$.
    Consider the coarse moduli space $\Higgs^{r,d}(C)$ of semistable Higgs bundles $(E,\Phi)$, where $E$ is a vector bundle of rank $r$ and degree $d$ on $C$, and $\Phi\in \Hom(E,E\otimes K_C)$.
    It has the natural $\T$-action by scaling the Higgs field $\Phi$.
    The pre-integrable system structure is given by the Hitchin map $\theta: \Higgs^{r,d}(C) \to \bigoplus_{i=1}^r H^0(C,K_C^{\otimes i})$, which sends $(E,\Phi)$ to the coefficients of the characteristic polynomial of $\Phi$. 
    This is an integrable system of weight $1$ and dimension $2(r^2(g-1) + 1)$, and it is smooth when $\gcd(r,d)=1$.

    When $g=0$ or $1$, one can still construct well-behaved moduli spaces by fixing a simple divisor $D\subset C$, considering Higgs fields valued in $K_C(D)$, and fixing a flag of vector spaces on $E_p$ compatible with $\Phi$ at each point $p\in D$.
    These are called \textit{parabolic} Higgs bundles.
    In particular, Hilbert schemes on surfaces from the parabolic family \eqref{parabolic2dimSpaces} can be realized as moduli of parabolic Higgs bundles~\cite{groechenig2014hilbert}.
    We will return to this point in \cref{sec:MS}.
\end{ex}

\begin{rmk}
    It is expected that Hilbert schemes of points on Painlevé systems are isomorphic to moduli of wildly ramified Higgs bundles, see~\cite[Conj.~11.3]{boalch2012simply}.
\end{rmk}

\subsection{Flows and Lagrangians}
Let us recall some notions of \BB theory, following~\cite[Sec.~2]{hausel2022very}.
From now on, we only consider \textit{smooth pre-integrable systems} $\theta: \Hitch\to \Base$ for simplicity; see \textit{loc.cit.} for the correct statements in the general case.

Since $\Hitch$ is smooth, its ${\T}$-fixed locus $\Hitch^{\T}$ is also smooth.
As ${\T}$ acts on $\Base$ with positive weights, we have $\Hitch^{\T}\subset \Core$, and so by properness of $\theta$ the ${\T}$-fixed locus $\Hitch^{\T}$ is projective.
For any $p\in\Hitch^{\T}$, the tangent space $T_p\Hitch$ has a natural ${\T}$-action. This allows us to decompose it:
\begin{equation*}
    T_p\Hitch = T_p^+ \oplus T_p\Hitch^{\T} \oplus T_p^-,
\end{equation*}
where $T_p^+$, resp. $T_p^-$ is the sum of ${\T}$-characters of positive, resp. negative weight.
We define (see~\cite[Sec.~1]{drinfeld2013algebraic} for details) the spaces of algebraic $\T$-equivariant maps
\begin{equation*}
    \Hitch^+ \coloneqq \mathsf{Maps}^{\T}(\mathbb{C}_{1}, \Hitch), \quad \Hitch^- \coloneqq \mathsf{Maps}^{\T}(\mathbb{C}_{-1}, \Hitch).
\end{equation*}
Note that $\Hitch^{\T} = \mathsf{Maps}^{\T}(\mathrm{pt}, \Hitch)$. The natural ${\T}$-equivariant maps
\[
    \begin{tikzcd}
        \mathbb{C}_{1}\ar[r,two heads,shift left=.5ex] & \mathrm{pt}\ar[l,hook',shift left=.5ex]\ar[r,hook,shift right=.5ex] & \mathbb{C}_{-1}\ar[l,two heads,shift right=.5ex]
    \end{tikzcd}
\]
give rise to morphisms
\[
    \begin{tikzcd}
        \Hitch^+\ar[r,two heads,shift right=.5ex,"\lim\limits_{t\to 0}"'] & \Hitch^{\T}\ar[l,hook',shift right=.5ex]\ar[r,hook,shift left=.5ex] & \Hitch^-.\ar[l,two heads,shift left=.5ex,"\lim\limits_{t\to \infty}"]
    \end{tikzcd}
\]

Let us denote $\Fix(\theta) \coloneqq \pi_0(\Hitch^{\T})$; we write $\Fix = \Fix(\theta)$ when $\theta$ is clear from the context.
For any $p\in \Hitch^{\T}$, denote the corresponding connected component by $F_p\in \Fix$.
We call
\[
    W^+_p \coloneqq \sideset{}{^{-1}}\lim\limits_{t\to 0}(p),\qquad W^-_p \coloneqq \sideset{}{^{-1}}\lim\limits_{t\to \infty}(p)
\]
the \textit{upward flow}, resp. \textit{downward flow} of $p$.
Similarly, for each fixed locus component $F\in \Fix$, we call
\[
    W^+_F \coloneqq \sideset{}{^{-1}}\lim\limits_{t\to 0}(F),\qquad W^-_F \coloneqq \sideset{}{^{-1}}\lim\limits_{t\to \infty}(F)
\]
the \textit{attracting cell}, resp. \textit{repelling cell} of $F$.
The projections $W^+_F\to F$, $W^-_F\to F$ are affine fibrations.
In particular, each $W^+_p$, $W^-_p$ is an affine ${\T}$-space, and we have isomorphisms $W^+_p\simeq T^+_p$, $W^-_p\simeq T^-_p$.

It is known that upward/downward flows, as well as attracting/repelling cells, are locally closed smooth subvarieties of $\Hitch$; with inclusions induced by the evaluation maps
\[
    \mathsf{Maps}^{\T}(\mathbb{C}_{\pm 1}, \Hitch)\to \Hitch, \qquad f\mapsto f(1).
\]
As $\T$ acts on $\Base$ with positive weights, all repelling cells $W^-_F$ belong to $\Core$.
We have \BB decomposition of $\Hitch$ into attracting cells, and of the core $\Core$ into repelling cells:
\begin{equation*}
    \Hitch = \bigsqcup_{F\in \Fix} W^+_F,\qquad \Core^{\mathrm{red}} = \bigsqcup_{F\in \Fix} W^-_F.
\end{equation*}
We write $\Core^{\mathrm{red}}$ in the statement above because $\Core$ is typically not reduced as a scheme; this will become important in \cref{sec:mult}.
Until then, we write $\Core = \Core^{\mathrm{red}}$ by abuse of notation.

Recall that ${\T}$ scales the symplectic form on $\Hitch$ with weight $k$.
We say that $F\in\Fix$ satisfies the \textit{weight gap condition} if the following holds:
\begin{equation*}\tag{WG}\label{eq:weight-gap}
    \begin{aligned}
        &\text{For $p\in F$, weights below $k$ do not occur in $T^+_p$,}
    \end{aligned}
\end{equation*}
and denote the set of all such components by $\Fix^{\WG}\subset \Fix$. 
When $k=1$, we have $\Fix^{\WG}= \Fix$. 
\begin{prop}\label{prop:flows-Lagr}
    Let $F\in \Fix$, and $p\in F$ a point.
    The subvariety $W^-_{F}$, resp. $W^+_p$ is isotropic, resp. coisotropic in $\Hitch$.
    Moreover, they are both Lagrangian iff $F\in \Fix^{\WG}$.
\end{prop}
\begin{proof}
    Let $v,w\in T_p\Hitch$ be homogeneous vectors of weight $i$, $j$ respectively.
    It is easy to see that $\omega(v,w) = 0$ unless $i+j=k$, so that $T_p(W^-_{F}) = T^-_p\oplus T_p F$ is isotropic and $T^+_p$ is coisotropic in $T_p\Hitch$.
    The space $T_p(W_F^-)$ has weights $0,-1,\dots$, which are $\omega$-dual to weights $k,k+1,\dots$. Thus $F\in \Fix^{\WG}$ is equivalent to
    $T_p(W_F^-)$ being half-dimensional, hence Lagrangian. 
    For the non-fixed points, one concludes as in \cite[Prop.~2.10]{hausel2022very}.
\end{proof}

As before, denote $2n\coloneqq\dim \Hitch$.
By the above, $\dim W^-_F = n$ if and only if $F\in \Fix^{\WG}$.
Let us denote $\Core_F\coloneqq \overline{W^-_F}$.
Since the dimensions of irreducible components of $\Core$ are bounded below by $\dim \Hitch - \dim \Base = n$, a simple dimension count shows the following:
\begin{prop}\label{prop:irr-equal-pi0}
$\Irr(\Core)=\{\Core_F \mid F\in\Fix^{\WG} \}.$ 
In particular, $\Core$ is equidimensional of dimension $n$.\qed
\end{prop}
\begin{cor}\label{cor:flatness}
    The map $\theta$ is flat.
\end{cor}
\begin{proof}
    The argument is essentially due to Ginzburg~\cite{ginzburg2001global}.
    Since the closure of every $\T$-orbit in $\Base$ contains zero, and the dimensions of irreducible components of fibers cannot decrease under restriction to special points, the fiber $\theta^{-1}(a)$ is equidimensional of dimension $n$ for every $a\in\Base$.
    We conclude by miracle flatness. 
\end{proof}

\subsection{Products and finite covers}\label{subs:products-and-covers}
Let $\theta_1:\Hitch_1\to\Base_1$, $\theta_2:\Hitch_2\to\Base_2$ be two pre-integrable systems of weight $k$.
Their product $\theta = \theta_1\times \theta_2:\Hitch_1\times \Hitch_2 \to \Base_1\times \Base_2$, equipped with the diagonal $\T$-action, is clearly a pre-integrable system of the same weight.
Furthermore, we have $(\Hitch_1\times \Hitch_2)^\T = \Hitch_1^\T \times \Hitch_2^\T$, $\Fix(\theta) = \Fix(\theta_1)\times \Fix(\theta_2)$, $\Core_{F_1\times F_2} = \Core_{F_1}\times \Core_{F_2}$ for any $F_i\in \Fix(\theta_i)$, and $F_1\times F_2$ satisfies~\eqref{eq:weight-gap} if and only if both $F_1$ and $F_2$ do.

Similarly, let $\theta:\Hitch\to\Base$ be an pre-integrable system, $\Base'$ another $n$-dimensional $\T$-vector space, and $\pi:\Base\to \Base'$ a finite surjective $\T$-equivariant map.
The latter implies that $\Base'$ has positive $\T$-weights, and so $\theta' = \pi\circ\theta$ is a pre-integrable system.
Since set-theoretically $\pi^{-1}(0) = 0$, passing from $\theta$ to $\theta'$ does not change $\Hitch^\T$, so that $\Fix(\theta) = \Fix(\theta')$.

\subsection{Flow order}\label{subs:Flow order}
Equip $\mathbb{P}^1 = \C \cup \{\infty\}$ with the $\T$-action extending the standard weight-$1$ action on $\C$, and call a finite $\T$-equivariant map $f:\mathbb{P}^1\to\Hitch$ a \textit{$\T$-curve}.
Given $F,F'\in \Fix$, we write $F \preceq^+ F'$ if $\overline{W^+_F} \cap F' \neq \varnothing$, and $F \preceq^- F'$ if $F\cap \overline{W^-_{F'}} \neq \varnothing$.

\begin{lm}\label{lm:adherence-spreads}
    Let $p\in \overline{W^+_F}$.
    Then $\lim\limits_{t\to 0} (p)\in \overline{W^+_F}$.
    Moreover, if $p\in F'$ for $F'\neq F$, then $\overline{W^+_F}\cap W^-_p\neq \{ p\}$.
\end{lm}
\begin{proof}
    By $\T$-equivariance, the whole $\T$-orbit $\mathbb{O}_p$ of $p$ lies in $\overline{W^+_F}$.
    Since $\lim\limits_{t\to 0} (p) \in \overline{\mathbb{O}_p}$, the first claim follows.
    The second claim is proven in~\cite[Lem.~9]{BB76} under the assumption that $\Hitch^\T$ is finite, however the same proof applies to our situation.
\end{proof}

Let us denote the transitive closure of $\preceq^\pm$ with the same symbol by abuse of notation.

\begin{lm}
    Both $\preceq^+$ and $\preceq^-$ are antisymmetric, and hence partial orders.
\end{lm}
\begin{proof}
    Let $\mathbb{P}^N$ be a projective space, equipped with linear $\T$-action.
    It was shown in~\cite[Th.~3]{BB76} that for any locally closed $\T$-invariant subvariety $X\subset \mathbb{P}^N$, there are total orders $\leq^\pm$ on $\pi_0(X^\T)$ such that $\bigsqcup_{F'\leq^\pm F} W^\pm_{F'}$ is closed for all $F\in \pi_0(X^\T)$.
    As $\Hitch$ is normal and quasi-projective, this result applies to $\Hitch$ by Sumihiro's theorem.
    Since $\preceq^\pm$ are closure preorders, both of them must be refined by the total orders above.
    In particular, they are antisymmetric. 
\end{proof}

\begin{prop}\label{prop:flow-order-descr}
    The transitive closures of $\preceq^+$ and $\preceq^-$ coincide, and define a partial order $\preceq$ on $\Fix$.
    Moreover, $F \preceq F'$ if and only if there exists a collection of $\T$-curves $f_0,\ldots, f_{k-1}$ and components $F_0,\ldots,F_k\in\Fix$, such that $f_i(0) \in F_i$, $f_i(\infty) \in F_{i+1}$ for all $1\leq i\leq k$, and $F_0=F$, $F_k=F'$.
\end{prop}
\begin{proof}
    We say that $F'$ is chained to $F$ if the condition in the second statement holds.
    It is clear from definitions that for any $F'$ chained to $F$ we have $F \preceq^+ F'$, $F \preceq^- F'$.
    Thus it suffices to show the reverse implication.
    
    Suppose that $F \preceq^+ F'$, so that there exists a point $p\in F'$ in $\overline{W^+_F}\cap F'$.
    By \cref{lm:adherence-spreads}, the closure $\overline{W^+_F}$ contains a whole $\T$-curve $f:\mathbb{P}^1\to\Hitch$ with $f(\infty) = p$.
    Now replace $p$ with $f(0)$, and repeat this procedure recursively.
    Since $\Hitch$ has finitely many fixed components, this process will eventually terminate.
    However, since \cref{lm:adherence-spreads} applies as long as $F'\neq F$, we must terminate with a $\T$-curve starting in $F$.
    Thus we have shown that $F'$ is chained to $F$.

    For $F \preceq^- F'$, the proof is almost entirely analogous; the only difference being that in order to apply \cref{lm:adherence-spreads} to the opposite $\T$-action, we need to work with a closure of $\Hitch$ inside an arbitrary projective space $\mathbb{P}^N$ with linear $\T$-action.
\end{proof}

We call the partial order $\preceq$ the \textit{flow order} on $\Fix$.
By the first statement of \cref{lm:adherence-spreads}, the following subvarieties are closed:
\begin{align*}
    W^\ggcurly_F \coloneqq \bigsqcup_{F'\succeq F} W^+_{F'} \subset \Hitch,\qquad 
    W^\llcurly_F \coloneqq \bigsqcup_{F'\preceq F} W^-_{F'} \subset \Core.
\end{align*}
Since \BB decomposition $\Hitch=\sqcup W_F^+$ contains the unique element of maximal dimension, we have the least element $\Fbot$ of $\Fix$ under $\preceq$.
By the above, we have $W^\ggcurly_\Fbot = \Hitch$ and $\Fbot = W^\llcurly_\Fbot$; hence we call $\Fbot$ the \textit{bottom component} of $\Fix$.
Note that $\Fbot$ does not have to belong to $\Fix^\WG$; when it does, the corresponding Lagrangian $\Core_\Fbot = \Fbot$ is fixed by $\T$.
\begin{rmk}
    Unlike here, in the setup of conical symplectic resolutions one typically has a collection of commuting weight-$1$ $\T$-actions, which give rise to a family of Lagrangians $\Fbot$.
    These ``minimal Lagrangians'' were the main object of study in~\cite{vzivanovic2022exact}.
\end{rmk}
\begin{de}
    We call a function $w: \Fix \to \mathbb{R}$ \textit{tidal} if $w(F)\geq w(F')$ for any $F\preceq F'$.
\end{de}

In other words, a tidal function refines the (reversed) partial order $\preceq$ to a total preorder.

\begin{ex}\label{ex:MO-tidal}
    Following~\cite[3.2.4]{MO12}, pick an ample $\T$-equivariant line bundle $L$ on $\Hitch$, and define $w_L(F)$ to be the equivariant weight of $L$ at $F$.
    Let $f$ be a $\T$-curve going from $F$ to $F'$, so that $F\preceq F'$.
    By equivariant localization, the degree of the restriction of $L$ to any $f$ is precisely $w_L(F) - w_L(F')$, which has to be positive since $L$ is ample.
    Thus $w_L$ is tidal. 

    Given a line bundle $L$ as above, we can consider the associated embedding of $\Hitch$ into a projective space.
    This induces an $S^1$-invariant {\KH} structure on $\Hitch$ such that the $S^1$-action is Hamiltonian, i.e. has a moment map $H_L:\Hitch\to \R$ (see e.g. \cite[Lm.~4.1]{vzivanovic2022exact} for details).
    The $\R_+$-flow coincides with the gradient flow of $H_L$, so that the value of $H_L$ increases along it, i.e. $w\coloneqq-H_L$ is tidal.
    Up to a positive multiple $\lambda >0$, the value of $H_L$ at a fixed point $p\in \Hitch^\T$ is precisely $w_L(p)$, so that $w = \lambda w_L$.
\end{ex}

Given a tidal function $w$, for any $s\in \mathbb{R}$ we can define the closed subvarieties
\[
    W^{\leq s} \coloneqq \bigsqcup_{w(F)\leq s} W^+_{F} \subset \Hitch,\qquad 
    W^{>s} \coloneqq \bigsqcup_{w(F)> s} W^-_{F} \subset \Core.
\]

Note that $(W^{> s})_s$ is an ascending flag of subvarieties, while $(W^{\leq s})_s$ is descending.
Since \BB decomposition induces (non-canonical) direct sum decomposition of Borel--Moore homology, these flags induce ascending filtration on $H^*(\Hitch)$ and descending filtration on $H_*(\Core)$:
\[
    T_s \coloneqq H^\BM_*(W^{\leq s}) \subset H^\BM_*(\Hitch) 
    \simeq  H^{4n-*}(\Hitch),
    \qquad 
    T^s \coloneqq H^\BM_*(W^{> s}) \subset H^\BM_*(\Core) \equiv H_*(\Core).
\]
Under the duality $H^*(\Hitch)^\vee = H^*(\Core)^\vee = H_*(\Core)$, we have $T_s^\vee = T^s$.

\begin{de}\label{def:tidal-filt}
    We call an (ascending) filtration on $H^*(\Hitch)$ \textit{tidal} if it arises from a tidal function $w$.
\end{de}

Let us momentarily restrict to cohomological degree $2n$.
By the discussion after \cref{prop:flows-Lagr}, 
$H_{2n}(\Core)$ has a natural basis $\{ [\Core_F] : F\in\Fix^\WG \}$.
Dually, \BB decomposition implies that $H^{2n}(\Hitch)$ has a basis $\{ \UF_F : F\in\Fix^\WG \}$, where $\UF_F\coloneqq[\overline{W^+_p}]$, and $p\in F$ is a general point. 

\begin{prop}\label{prop:tidal-spans}
    For a tidal filtration $(T_s)$, we have $T_s H^{2n}(\Hitch) = \Span\{ \UF_F : w(F)\leq s \}$.
\end{prop}
\begin{proof}
    It is clear from definitions that $T^s H_{2n}(\Core) = \Span\{ [\Core_F] : w(F) >s \}$.
    By definition of the flow order we have $[\Core_F]\cap \UF_{F'} = 0$ unless $F'\preceq F$, and $[\Core_F]\cap \UF_F = 1$.
    This implies that $[\Core_F]^\vee = \UF_F + \sum_{F'\succ F} a_{F'F} \UF_{F'}$ for some $a_{F'F}\in \mathbb{Z}$.
    By the unitriangularity of the transition matrix we have
    \[
        T_s H^{2n}(\Hitch) = \Span\{ [\Core_F]^\vee : w(F)\leq s \} = \Span\{ \UF_F : w(F)\leq s \},
    \]
    and so we may conclude.
\end{proof}

\begin{ex}
The filtration associated to the tidal function $w=-H$ in \cref{ex:MO-tidal} is called the {\AB} filtration.
\end{ex}

\begin{ex}\label{ex:perv-not-tidal}
    Recall that $H^*(\Hitch)$ can be equipped with perverse Leray filtration $P_\bullet$, associated to the map $\theta$.
    For the parabolic family of two-dimensional integrable systems, $P_\bullet$ has been computed in~\cite{zhang2017multiplicativity}.
    For instance, $H^2(S_{\mathbb{Z}/2})$ acquires a two-step filtration, with the first step being given by the span of fundamental classes of the four non-central $\mathbb{P}^1$'s.
    In terms of the basis $\{ \UF_F\}$, these classes are expressed as $C - 2\UF_{p_i}$, where $C$ is the central $\mathbb{P}^1$, and $p_i$, $1\leq i\leq 4$ are the other fixed points.
    \cref{prop:tidal-spans} then implies that the perverse filtration is not tidal.
\end{ex}

\subsection{Very stable points}\label{subs:v-stable}
Following~\cite{hausel2022very}, we say that $p\in\Hitch^{\T}$ is \textit{very stable} if the upward flow $W^+_p$ is closed in $\Hitch$, and \textit{wobbly} otherwise.

\begin{lm}\label{lem:gen-finite}
    Let $F\in \Fix^\WG$, and $p\in F$.
    For a generic $a\in\Base$, the intersection $W^+_p\cap \theta^{-1}(a)$ is transversal of finite cardinality $\gamma_p$.
\end{lm}
\begin{proof}
    We follow the proof of~\cite[Lm.~4.6]{hausel2022very}.
    Without loss of generality, we can assume that the projection $\th|_{W^+_p}: W^+_p\to\Base$ is dominant, otherwise $\gamma_p=0$.
    Let $\Base^\circ$ be the open complement of $\theta(\overline{W^+_p}\setminus W^+_p)$ in $\Base$.
    Then $\theta$ restricts to a map $W^+_p\to \Base^\circ$.
    As it is a proper dominant map between smooth equidimensional varieties, it is generically étale.
    The claim immediately follows.
\end{proof}

\begin{prop}\label{prop:v-stable-lf-push}
    Let $p\in \Hitch^\T$.
    The following conditions are equivalent:
    \begin{enumerate}
        \item $p$ is very stable;
        \item $W^+_p\cap \Core = \{p\}$;
        \item $\theta_*\mathcal{O}_{W^+_p}$ is a locally free sheaf of rank $\gamma_p$.
    \end{enumerate}
\end{prop}

\begin{proof}
    We have
    \[
        W^+_p = \overline{W^+_p} \quad\Leftrightarrow\quad \text{no $\T$-curves starting at $p$} \quad\Leftrightarrow\quad W^+_p\cap \Core = \{p\},
    \]
    which proves the equivalence of the first two conditions.
    If $p$ is very stable, then $W^+_p\to \Base$ is a proper map between affine varieties, hence finite.
    By miracle flatness, this implies that $\theta_*\mathcal{O}_{W^+_p}$ is locally free.
    Conversely, let us assume (3). Since $W^+_p\to \Base$ is a map of affine varieties, it has to be finite.
    Hence, its fiber over $0$ is finite, and so does not contain a non-constant $\T$-orbit.
    This shows that $p$ is very stable.
\end{proof}
\begin{rmk}
    By Quillen--Suslin theorem, $\theta_*\mathcal{O}_{W^+_p}$ is locally free if and only if it is free.
    We will not use this result.
\end{rmk}
It follows from \cref{prop:v-stable-lf-push} that the set of very stable points is Zariski open in $\Hitch^{\T}$.
We will therefore abuse the notation and say $F\in\Fix$ is \textit{very stable} if it contains a very stable point, and \textit{wobbly} otherwise; note that a very stable fixed component might contain wobbly points.
Furthermore, the condition (2) from \cref{prop:v-stable-lf-push} is preserved under taking étale covers, so that we can check whether a point is very stable étale-locally.

\begin{cor}\label{cor:v-stable}
    Any fixed component which is maximal in flow order is very stable.
    Any very stable component belongs to $\Fix^\WG$.
    When $\Fbot\in \Fix^\WG$, it is very stable.
\end{cor}
\begin{proof}
    The first claim follows immediately from \cref{prop:v-stable-lf-push}(2).
    For the second claim, let $p\in F$ be very stable.
    As the dimension of $\theta^{-1}(a)\cap W^+_p$ is upper semi-continuous in $a\in \Base$, we have $\dim(W^+_p\cap \Core)\geq \dim(W^+_p)-n$.
    In particular, \cref{prop:flows-Lagr} implies that unless $F\in \Fix^\WG$, the intersection $W^+_p\cap \Core$ cannot be zero-dimensional.
    Finally, for the last claim, note that the only weights of $\Fbot$ are 
    ($\omega$-dual) $0$ and $k$, thus $\dim \Fbot = n$. 
    Since $\dim \Core_F =\dim \Fbot=n$ for any $F\in \Fix^\WG\setminus \Fbot$,  intersection $\Core\cap \Fbot$ has positive codimension in $\Fbot$.
    This implies that a general point $p\in \Fbot$ is very stable.
\end{proof}

Given a fixed point $p\in F$, $F \in \Fix$, denote
\begin{equation}\label{eq:HH-eq-mult}
    \mu_F(t) \coloneqq\frac{\chi_t(\Sym T^+_p)}{\chi_t(\Sym \Base)}\in \mathbb{Z}(t).
\end{equation}
By \cref{prop:v-stable-lf-push} and equivariant integration, if $p$ is very stable, $\theta_*\mathcal{O}_{W^+_p}$ is a (locally) free sheaf of $\T$-graded rank $\mu_F(t)$; in particular, $\gamma_p=\mu_F(1)$.

\begin{prop}\label{prop:HH-mult}
    For any $F\in \Fix^\WG$ and $p\in F$, we have $\gamma_p \leq \mu_F(1)$.
    Moreover, the equality holds if and only if $p$ is very stable.
\end{prop}
\begin{proof}
    Thanks to \BB decomposition, we can see the restriction of $\theta$ to $W^+_p$as a map $T^+_p \simeq W^+_p \to \Base$.
    Denote the equivariant weights of $T^+_p$ by $a_1,\ldots,a_n$.
    On the rings of functions, this map is given by a homogeneous map
    \[
        \mathbb{C}[\Base] = \mathbb{C}[x_1,\ldots,x_n]\xrightarrow{\theta^*} \mathbb{C}[y_1,\ldots,y_n] = \mathbb{C}[T^+_p],
    \]
    where $\deg x_i = e_i$, $\deg y_i = a_i$ (recall weights on $\Base$ are denoted by $e_i$).
    By definition, we have
    \[
        \mu_F(1) = \left.\frac{(1-t^{e_1})\ldots(1-t^{e_n})}{(1-t^{a_1})\ldots(1-t^{a_n})}\right|_{t=1} = \frac{\prod_i e_i}{\prod_i a_i}>0.
    \]
    If $\theta^*$ is not injective, then $p$ is wobbly by \cref{prop:v-stable-lf-push}, and $\theta|_{T^+_p}$ is not dominant, so that $\gamma_p=0$.
    We therefore assume from now on that $\theta^*$ is injective.

    Consider $\mathbb{C}[T^+_p]$ as a $\mathbb{C}[\Base]$-module, and let $M = \mathbb{C}[T^+_p]/\mathbb{C}[\Base]$, and $N$ the torsion submodule of $M$.
    By \cref{lem:gen-finite}, $\mathbb{C}[T^+_p]$ is generically finitely generated and locally free, and so $M'=M/N$ is finitely generated torsion free.
    We can find a (locally) free graded $\mathbb{C}[\Base]$-module $M''\supset M'$ of the same rank.
    Both $M''/M'$ and $N$ are supported non-generically on $\Base$; moreover, since $M''/M'$ is manifestly finitely generated, its Hilbert series $\chi_t(M''/M')$ have a pole of order strictly smaller than $n$ at $t=1$.
    As a consequence,
    \begin{align*}
        \gamma_p & = 1+\rk_{\mathbb{C}[\Base]} M''
        = 1+\left.\frac{\chi_t(M'')}{\chi_t(\Base)}\right|_{t=1}
        = 1+\left.\frac{\chi_t(M'') - \chi_t(M''/M')}{\chi_t(\Base)}\right|_{t=1}\\
        & = 1+\left.\frac{\chi_t(M) - \chi_t(N)}{\chi_t(\Base)}\right|_{t=1}
        \leq 1+\left.\frac{\chi_t(M)}{\chi_t(\Base)}\right|_{t=1}
        = \left.\frac{\chi_t(\mathbb{C}[T^+_p])}{\chi_t(\Base)}\right|_{t=1}
        = \mu_F(1).
    \end{align*}
    This proves $\gamma_p\leq \mu_F(1)$.
    It remains to show that the inequality is strict for wobbly $p$.
    In this case, $T^+_p$ contains a one-dimensional $\T$-orbit.
    In terms of rings of functions, it means that there exists a homogeneous $f\in \mathbb{C}[T^+_p]$ such that the induced map of vector spaces $\mathbb{C}[f] \to M$ is injective.
    Let us denote $d = \deg(f)$.
    There exists $1\leq j\leq n$ such that the set $\{y_1,\ldots,y_{j-1},f,y_{j+1},\ldots,y_n\}$ is algebraically independent in $\mathbb{C}[T^+_p]$; for notational convenience, let us assume $j=n$.
    Substituting $f$ with its positive power, we can assume that $\theta^*(x_i)\not\in (f^2)$ for all $i$.
    Nevertheless, for transcendence degree reasons $\mathbb{C}[\Base]\cap (f)\neq \{0\}$.
    This implies that for any monomial $q= q(y_1,\ldots,y_{n-1})$ of $y$-degree $k$, the element $qf^{k+1}\in M$ is torsion, which allows us to estimate $\chi_t(N)$.
    Let us write $\prod_{i=1}^{n-1} (1-t^{a_i})^{-1} = \sum_{i} b_i t^i$; then
    \begin{align*}
        \chi_t(N) \geq \sum_{k\geq 0} t^{d{k+1}}\sum_{i=0}^k b_it^i = \sum_{i\geq 0} \frac{b_it^{(i+1)d+i}}{1-t^d} = \frac{t^d}{1-t^d} \prod_{i=1}^{n-1}(1-t^{a_i})^{-1}.
    \end{align*}
    In particular, $\chi_t(N)$ has a pole of order $n$ at $t=1$.
    Therefore, $\chi_t(N)/\chi_t(\Base)|_{t=1}$ is strictly positive, and so the inequality above is strict.
\end{proof}

\section{Equivariant multiplicities}\label{sec:mult}
In this section, we give a geometric interpretation of equivariant multiplicities from~\cite{hausel2022very}, collect some of their structural properties, and state several conjectures.

\subsection{$K$-theoretic prerequisites}
Given a $\T$-scheme $X$, we denote the Grothendieck group of $\T$-equivariant coherent sheaves on $X$ with $\mathbb{Q}$-coefficients by $K^\T(X)$\footnote{This group is usually denoted by $G_0(X)_\mathbb{Q}$, but our abuse of notation is standard in the geometric representation theory community.}.
By pullback along the projection to a point, this group is always a $K^\T(\mathrm{pt})$-module. 
We use a slightly unusual convention that $K^\T(\mathrm{pt}) = \mathbb{Q}[t^{\pm 1}]$, where $t$ is the class of weight $-1$ character of $\T$.
The $\T$-scheme $X$ is \textit{(equivariantly) formal} if $K^\T(X)\simeq K(X)\otimes K^\T(\mathrm{pt}) = K(X)[t^{\pm 1}]$; for example, $X$ is always formal if the $\T$-action is trivial, and the total space of a $\T$-equivariant vector bundle over a formal space is formal.
When $X$ is smooth, every coherent sheaf on $X$ admits a finite length resolution by vector bundles.
Assigning to each vector bundle its rank, we get a $\mathbb{Q}$-linear map $K(X)\to \mathbb{Q}$ in non-equivariant $K$-theory.
In particular, when $X$ is smooth and formal we obtain a $\mathbb{Q}[t^{\pm 1}]$-linear map
\[
    \rk: K^\T(X) \simeq K(X)[t^{\pm 1}] \to \mathbb{Q}[t^{\pm 1}].
\]

\begin{lm}\label{lem:rank-restr}
    Let $X$ be a smooth scheme with a trivial $\T$-action, $\pi:E\to X$ a $\T$-equivariant vector bundle, and $s:X\to E$ the zero section.
    For any $M \in K^\T(E)$, we have $\rk M = \rk s^*(M)$.
    Moreover, if $V$ is a $\T$-equivariant vector bundle on $E$, we have $\rk[V] = \chi_t(V_{p})$ for any $p\in X$.
\end{lm}
\begin{proof}
    By Thom isomorphism, $K^\T(E)$ is spanned over $\mathbb{Q}[t^{\pm 1}]$ by classes $\pi^*(V)$, where $V$ is a (non-equivariant) vector bundle over $X$.
    We have 
    \[
        \rk s^*\pi^*(V) = \rk V = \rk s^*(V),
    \]
    and so the first claim follows by linearity of rank.
    For the second claim, note that $\chi_t(V_p) = [i^* V]$, where $i:\{p\}\hookrightarrow E$ is the inclusion, and use an analogous argument.
\end{proof}

For any closed embedding of schemes $Y\subset X$, denote by $C_{Y/X}$ the normal cone.
We have the deformation to normal cone $D_{Y/X}$, which is a flat family over $\mathbb{C}$ interpolating between $X$ and $C_{Y/X}$:
\[
    \begin{tikzcd}[row sep=small, column sep=small]
        C_{Y/X}\ar[r,hook]\ar[d] & D_{Y/X}\ar[d] & X\times \mathbb{C}^*\ar[d]\ar[l,hook']\\
        0\ar[r] & \mathbb{C} & \mathbb{C}^*\ar[l]
    \end{tikzcd}
\]
Let $\mathcal{I}\subset \mathcal{O}_X$ be the defining sheaf of ideals of $Y$.
Given a coherent sheaf $\mathcal{F}$ on $X$, its pullback to $X\times \mathbb{C}^*$ can be extended to a sheaf $\widetilde{\mathcal{F}}$ on $D_{Y/X}$, flat over $\mathbb{C}$, such that its fiber at $0$ is
\[
    \sp_Y(\mathcal{F}) \coloneqq \bigoplus_{k\geq 0} \mathcal{I}^k\mathcal{F}/\mathcal{I}^{k+1}\mathcal{F}.
\]
One can check that the assignment $[\mathcal{F}]\mapsto[\sp_Y(\mathcal{F})]$ defines a $\mathbb{Q}[t^{\pm 1}]$-linear map $\sp_Y:K^\T(X)\to K^\T(C_{Y/X})$, see e.g.~\cite[§5.3]{CGi97}.

Finally, let us recall some facts about equivariant localization.
For any $\T$-variety $X$, we denote $K^\T(X)_\loc \coloneqq K^\T(X)\otimes_{\mathbb{Q}[t^{\pm 1}]} \mathbb{Q}(t)$.
Pushforward along the inclusion $i_X: X^\T \hookrightarrow X$ induces the map 
\[
    (i_X)_*^\loc : K^\T(X^\T)_\loc \to K^\T(X)_\loc.
\]
The localization theorem states that this map is an isomorphism; the proof is essentially the same as for Chow groups~\cite[§2.3, Cor.~2]{brion1997equivariant}, see e.g.~\cite[Prop.~A.13]{minets2020cohomological} for details.

Consider a closed embedding of $\T$-varieties $f:Y \hookrightarrow X$, such that $Y^\T = X^\T$.
We have the following commutative square:
\[
    \begin{tikzcd}[row sep=small, column sep=small]
        K^\T(Y)\ar[r,"f_*"]\ar[d,"l_Y"] & K^\T(X)\ar[d,"l_X"]\\
        K^\T(Y)_\loc\ar[r,"f_*^\loc"] & K^\T(X)_\loc
    \end{tikzcd}
\]
By the localization theorem, the bottom map is an isomorphism, so that we can define 
\[
    \loc_Y \coloneqq (f_*^\loc)^{-1}\circ l_X: K^\T(X)\to K^\T(Y)_\loc.
\]
If $K^\T(Y)$ is $\mathbb{Q}[t^{\pm 1}]$-torsion free (for instance, when $Y$ is formal), the map $l_Y$ is injective.
Then $f_*$ is also injective, and so $\loc_Y$ can be thought of as a left inverse of $f_*$.
For instance, when both $X$ and $Y$ are smooth, \cite[Prop.~5.4.10]{CGi97} implies that
\begin{equation}\label{eq:loc-smooth}
    \loc_Y(-) = \frac{f^*(-)}{[\wedge^\bullet (N_Y X)^\vee]}
\end{equation}
where $[\wedge^\bullet E] = \bigoplus_{i=0}^{\rk E} [\wedge^i E]$ for a vector bundle $E$.

\subsection{Equivariant multiplicity of coherent sheaves}\label{subs:em-cohsh}
Let $Y\subset X$ be a locally closed embedding of $\T$-schemes, such that for any $y\in Y\cap X^\T$ the normal cone at $y$ has no directions preserved by $\T$.
Then $(C_{Y/X})^\T \subset Y$.
For example, this condition is satisfied when $X$ is smooth, and the intersection $X^\T\cap Y$ is a collection of connected components of $X^\T$.

\begin{de}\label{def:em-class}
    Keep the notations above.
    For $\mathcal{E}\in \Coh X$, we define the equivariant multiplicity class of $\mathcal{E}$ 
    along $Y$ by
    \[ 
        \mu(Y;\mathcal{E}) \coloneqq \loc_Y\circ \sp_Y([\mathcal{E}]) \in K^\T(Y)_\loc.
    \]
\end{de}

This definition is closely related to the notion of equivariant multiplicity of Brion and Rossmann.
Namely, let us replace $K^\T(X)$ by the equivariant Chow group $A_*^\T(X)$, and set $Y = \{x\}$ to be a fixed point of $X$ with non-zero tangent weights.
Even though it is stated in slightly different language, one can check that our definition then specializes to equivariant multiplicities in \cite[§4.2]{brion1997equivariant}.

For our purposes, it is important to study the generic behaviour of multiplicity.
Let $\mathcal{E}\in \Coh X$, and consider the maximal sheaf $\mathcal{T}_Y\subset\sp_Y(\mathcal{E})$ such that $\Supp(\mathcal{T}_Y) \cap Y \neq Y$.
Since the supports of all subsheaves of $\sp_Y(\mathcal{E})$ are conical, we can alternatively define $\mathcal{T}_Y$ as the subsheaf of $\mathcal{O}_Y$-torsion.
We denote $\sp_Y(\mathcal{E})_{\mathrm{tf}}\coloneqq \sp_Y(\mathcal{E})/\mathcal{T}_Y$.

\begin{de}\label{def:gen-em-class}
    We define the generic equivariant multiplicity class of $\mathcal{E}$ 
    along $Y$ by
    \[ 
        m(Y;\mathcal{E}) \coloneqq \loc_Y([\sp_Y(\mathcal{E})_{\mathrm{tf}}]) \in K^\T(Y)_\loc.
    \]
\end{de}

Since both $\loc_Y$ and $\sp_Y$ descend to $K$-theory, the assignment $[\mathcal{E}]\mapsto [\mu(Y;\mathcal{E})]$ defines a map of $\mathbb{Q}[t^{\pm 1}]$-modules $K^\T(X)\to K^\T(Y)_\loc$.
For instance, for any finite-dimensional $\T$-vector space $V$ we have $\mu(Y;V\otimes\mathcal{E}) = \chi_t(V)\mu(Y;\mathcal{E})$.
On the other hand, the operation $(-)_{\mathrm{tf}}$ is not preserved under taking extensions; therefore $m(Y;\mathcal{E})$ depends on $\mathcal{E}$ itself, and not its class in $K$-theory.

It is elementary to check that both $m$ and $\mu$ are compatible with direct products:
\begin{lm}\label{lem:shv-eq-mult-product}
    Let $Y_i\subset X_i$, $\mathcal{E}_i\in \Coh X_i$ be as in \cref{def:em-class} for $i=1,2$.
    Then the following equalities hold in $K^\T(Y_1\times Y_2)_\loc$:
    \[
        \mu(Y_1\times Y_2;\mathcal{E}_1\boxtimes \mathcal{E}_2) = \mu(Y_1;\mathcal{E}_1)\boxtimes \mu(Y_2;\mathcal{E}_2), \quad m(Y_1\times Y_2;\mathcal{E}_1\boxtimes \mathcal{E}_2) = m(Y_1;\mathcal{E}_1)\boxtimes m(Y_2;\mathcal{E}_2).
    \]
\end{lm}

Note that $\mu(Y;\mathcal{E})$ significantly depends on $Y$, and not just on $Y\cap X^\T$.
\begin{ex}\label{ex:C2-mult}
    Let $X = \mathbb{C}_{a_1} \oplus \mathbb{C}_{a_2}$ with $a_1,a_2\neq 0$, so that $X^\T = \{0\}$.
    Denote by $x_i$ the coordinate on $\mathbb{C}_{a_i}$, and take $\mathcal{E} = \mathcal{O}_{X}/(x_1^{b_1}x_2^{b_2})$.
    Note that by definition $[\mathcal{E}] = 1-t^{d}$ in $K^\T(X) \simeq \Q[t^{\pm 1}]$, where $d = a_1b_1+a_2b_2$.
    For $Y\in\{X, \mathbb{C}_{a_1}, \mathbb{C}_{a_2}, 0\}$, the specialization map sends $\mathcal{E}$ to $\mathcal{E}$ itself.
    Using formula~\eqref{eq:loc-smooth} and our weight conventions, we can easily see that 
    \[
        \mu(X;\mathcal{E}) = 1-t^{d},\quad \mu(\mathbb{C}_{a_1};\mathcal{E}) = \frac{1-t^{d}}{1-t^{a_2}},\quad \mu(\mathbb{C}_{a_2};\mathcal{E}) = \frac{1-t^{d}}{1-t^{a_1}},\quad \mu(\{0\};\mathcal{E}) = \frac{1-t^{d}}{(1-t^{a_1})(1-t^{a_2})}.
    \]
    For generic equivariant multiplicities, note that $\mathcal{T}_X = \mathcal{E}$ and $\mathcal{T}_{\{0\}} = \{0\}$, so that $m(X;\mathcal{E}) = 0$, $m(\{0\};\mathcal{E}) = \mu(\{0\};\mathcal{E})$.
    Furthermore, the quotient $\mathcal{E}_\mathrm{tf}$ is equal to $\mathcal{O}_{X}/(x_2^{b_2})$ for $\mathbb{C}_{a_1}$ and $\mathcal{O}_{X}/(x_1^{b_1})$ for $\mathbb{C}_{a_2}$, so that 
    \begin{equation}\label{eq:2d-m}
        m(\mathbb{C}_{a_1};\mathcal{E}) = \frac{1-t^{a_2b_2}}{1-t^{a_2}} = [b_2]_{t^{a_2}},\qquad m(\mathbb{C}_{a_2};\mathcal{E}) = \frac{1-t^{a_1b_1}}{1-t^{a_1}} = [b_1]_{t^{a_1}}.
    \end{equation}
\end{ex}

\subsection{Equivariant multiplicity of core Lagrangians}
Let us specialize the general discussion above to our situation of interest.
Namely, let $\theta: \Hitch\to \Base$ be a smooth pre-integrable system.
Every \BB stratum of the core is a vector bundle over a $\T$-invariant variety, and thus formal.
In particular, for each $F\in \Fix$ the rank map $\rk:K^\T(W^-_F)\to\mathbb{Q}[t^{\pm 1}]$ is well defined. 

\begin{de}\label{em-components}
    Set $X = \Hitch$. Given $F\in \Fix$, we also set $Y = W^-_F$. Define
    \[
        m_F(t) \coloneqq \rk m(W^-_F;\mathcal{O}_\Core),\qquad \mu_F(t) \coloneqq \rk \mu(W^-_F;\mathcal{O}_\Core).
    \]
    We say that $m_F(t)$ is the \textit{equivariant multiplicity} of $\Core_F$ in $\Core$, and $\mu_F(t)$ is the \textit{virtual equivariant multiplicity} of $\Core_F$ in $\Core$.
\end{de}

The change in terminology compared to \cref{subs:em-cohsh} can be motivated in several ways.
First, as we work with $W^-_F$ instead of $\Core_F$, both notions are in a sense generic.
Second, $m_F(t)$ is related to the classical notion of multiplicity of an irreducible component, see \cref{prop:non-eq-mult}.
Finally, note that virtual equivariant multiplicity is compatible with the function in \cref{subs:v-stable}, which was called virtual multiplicity in~\cite{hausel2022very}:

\begin{prop}\label{prop:sing-is-mu}
    Virtual equivariant multiplicity in the sense of \cref{em-components} is given by~\eqref{eq:HH-eq-mult}. 
\end{prop}
\begin{proof}
    As $\theta$ is flat, we have an equality $[\mathcal{O}_\Core] = [\wedge^\bullet \Base] = \chi_t(\Sym\Base)^{-1}$ in $K^\T(\Hitch)$.
    Since this is a rational function, we can use the linearity of $\mu$:
    \[
        \mu(W^-_F;\mathcal{O}_\Core) = \frac{\mu(W^-_F;\mathcal{O}_\Hitch)}{\chi_t(\Sym\Base)}.
    \]
    Specialization preserves the structure sheaf, and localization is given by the formula~\eqref{eq:loc-smooth}.
    Fix an arbitrary point $p\in F$.
    For any vector bundle $V$ on $W^-_F$ we have $\rk V = \chi_t(V_p)$ by \cref{lem:rank-restr}, therefore
    \[
        \mu_F(t) = \frac{\rk[\wedge^\bullet (N_{W^-_F} \Hitch)^\vee]^{-1}}{\chi_t(\Base)}
        = \frac{[\wedge^\bullet (T^+_p)^\vee]^{-1}}{\chi_t(\Base)}
        = \frac{\chi_t(\Sym T^+_p)}{\chi_t(\Base)},
    \]
    and so we are done.
\end{proof}

Recall that classically, the \textit{multiplicity} $m_F$ of component $\Core_F$ in $\Core$ is defined as the length of the local ring $\mathcal{O}_{\Core,\eta_F}$, where $\eta_F$ is the generic point of $\Core_F$.
Because of its "generic" definition, equivariant multiplicity recovers the usual one.
\begin{prop}\label{prop:non-eq-mult}
    Fix $F\in \Fix^\WG$.
    We have $m_F(t)\in \mathbb{Z}[t^{\pm 1}]$, and $m_F(1) = m_F$.
\end{prop}
\begin{proof}
    By definition, the sheaf $\mathcal{E}\coloneqq \sp_{W^-_F}(\mathcal{O}_\Core)_{\mathrm{tf}}$ is both torsion-free and cyclic as an $\mathcal{O}_{N_{W^-_F} \Hitch}$-module.
    As $F$ satisfies condition~\eqref{eq:weight-gap}, and the support of $\mathcal{O}_\Core$ is $n$-dimensional, we deduce that $\mathcal{E}$ has pure support in $W^-_F$, and so it is coherent as a $\mathcal{O}_{W^-_F}$-module.
    Denoting by $i:W^-_F\to N_{W^-_F} \Hitch$ the zero section, the sheaf $\mathcal{E}$ can therefore be obtained as a successive extension of sheaves of the form $i_*\mathcal{F}_j$, $\mathcal{F}_j\in \Coh(N_{W^-_F} \Hitch)$, so that $[\mathcal{E}] = i_*M$, $M\in K^\T(W^-_F)$.
    By formula~\eqref{eq:loc-smooth}, we have $\loc_{N_{W^-_F}}([\mathcal{E}]) = M\in K^\T(W^-_F)$.
    Therefore $m_F(t) = \rk M$ belongs to $K^\T(\mathrm{pt}) = \mathbb{Z}[t^{\pm 1}]$ as claimed.

    Let us compute the value $m_F(1)$.
    As shown above, we can replace localization in the definition of $m_F(t)$ by restriction to $W^-_F$.
    In particular, this allows us to set $t=1$ at the very beginning, and work in non-equivariant $K$-theory.
    Since $\mathcal{E}$ is torsion-free, we can pick a (non-equivariant) open $U\subset W^-_F$ such that $\mathcal{E}|_U$ is free.
    Then on one hand $m_F(1) = \rk \mathcal{E}|_U$, as rank can be computed locally, and on the other hand the length $m_F$ of $\mathcal{O}_{\Core,\eta_F}$ coincides with the rank of $\mathcal{E}|_U$.
\end{proof}

We expect that the form of $m_F(t)$ can be further constrained.
The following conjecture holds in all the situations where we were able to compute equivariant multiplicity, see \cref{ssec:2d-mult-ex,sec:ex-Hilb2}.
Note that equivariant multiplicities neither have positive coefficients nor are palindromic in general, see~\eqref{eq:eq-mult-Hilb2-sep} and~\eqref{eq:eq-mult-Hilb2-punct-y3}.
\begin{con}\label{conj:eq-mult-positivity}
    For any $F\in \Fix^\WG$, we have $m_F(t)\in \mathbb{Z}[t]$, and $m_F(0)=1$.
    Furthermore, $m_F(t)$ is palindromic if and only if $\sp_{W^-_F}(\mathcal{O}_\Core)_{\mathrm{tf}}$ is locally free.
\end{con}

Equivariant multiplicities can be computed étale-locally.
Namely, pick a $\T$-equivariant étale neighbourhood $j:U\to \Hitch$ such that $F|_U\coloneqq U\times_{\Hitch} F$ is non-empty.
Then $j^*$ intertwines $\loc_{W^-_F}$ with $\loc_{W^-_U}$ (by functoriality of pullback), commutes with rank map (since it is local in $F$), with $\sp_{W^-_F}$ (as deformation to normal cone construction is étale local), and with the procedure $(-)_\mathrm{tf}$ (because the support condition is étale local).
In particular, replacing $\Core$ with $\Core|_U = U\times_{\Hitch} \Core$, $F$ with $F|_U$ and $W^-_F$ with $W^-_{F|_U}$, we obtain that $m_F(t) = m_U(t)$.
The same also holds for $\mu$, because by \cref{prop:sing-is-mu} it is computed from tangent datum.

\begin{prop}\label{prop:product-mult}
    Let $\theta:\Hitch_1\times \Hitch_2\to \Base_1\oplus \Base_2$ be a product of pre-integrable systems. Then for any $F_1\times F_2\in\Fix(\theta)$, we have $\mu_{F_1\times F_2}(t) = \mu_{F_1}(t) \mu_{F_2}(t)$ and $m_{F_1\times F_2}(t) = m_{F_1}(t) m_{F_2}(t)$.
\end{prop}
\begin{proof}
    Follows immediately from \cref{lem:shv-eq-mult-product} and the fact that taking ranks commutes with Künneth map.
\end{proof}

Before stating the next property, we need some notations.
Let $\theta:\Hitch\to \Base$ be an pre-integrable system, $\Base'$ another $n$-dimensional $\T$-vector spaces with positive weights, and 
$$\pi:\Base\to \Base'$$ 
a finite $\T$-equivariant map.
Denote $\delta_\pi(t): = \chi_t(\Sym \Base)/\chi_t(\Sym \Base')$, and observe that $\delta_\pi(1) = \deg(\pi)$.
Let $\theta'\coloneqq\pi\circ \theta$, and denote the core of $\theta'$ by $\Core'$. 
We have $\Fix(\theta) = \Fix(\theta')$, therefore for any $F\in \Fix(\theta)$ we denote the corresponding element of $\Fix(\theta')$ by $F'$.

\begin{prop}\label{prop:cover-mult}
    In the notations above, we have $\mu_{F'}(t) = \delta_\pi(t) \mu_F(t)$, and $m_{F'} = \deg(\pi) m_F$.
\end{prop}
\begin{proof}
    The first claim follows immediately from \cref{prop:sing-is-mu}:
    \[
        \delta_{\pi}(t) = \frac{\mu_{F'}(t)}{\mu_{F}(t)} = \frac{\chi_t(\Sym T^+_p)}{\chi_t(\Sym \Base')} \frac{\chi_t(\Sym \Base)}{\chi_t(\Sym T^+_p)} = \frac{\chi_t(\Sym \Base)}{\chi_t(\Sym \Base')}.
    \]
    For the second claim, consider the following diagram with Cartesian squares:
    \[
    \begin{tikzcd}[row sep=small, column sep=small]
        \Core\ar[r,hook]\ar[d] & \Core'\ar[r,hook]\ar[d] & \Hitch\ar[d,"\theta"]\\
        0\ar[r,hook] & \pi^{-1}(0)\ar[r,hook]\ar[d] & \Base\ar[d,"\pi"]\\
         & 0\ar[r,hook] & \Base'
    \end{tikzcd}
    \]
    By our assumptions, $V_\pi\coloneqq \C[\pi^{-1}(0)]$ is an Artinian algebra of dimension $\deg(\pi)$.
    Passing to the local rings, we see that at the generic point $\eta_F$ we have $\mathcal{O}_{\Core',\eta_{F'}} \simeq V_\pi\otimes \mathcal{O}_{\Core,\eta_F}$, and so $m_{F'} = \deg(\pi) m_F$.
\end{proof}

\begin{con}\label{conj:finite-cover}
    For any $F\in \Fix^\WG(\theta)$ and $\pi$ as above, there exists a polynomial $d_{F,\pi}(t)\in \Z[t]$ such that $m_{F'}(t) = d_{F,\pi}(t) m_F(t)$.
\end{con}
Note that $d_{F,\pi}(t)$ depends non-trivially on $F$.
In general $d_{F,\pi}(t)\neq\delta_\pi(t)$, see \cref{ex:degree-correction}.

\subsection{Characterization of very stable components}
For a very stable $p\in F$, Hausel--Hitchin formally defined the equivariant multiplicity of $\Core_F$ by the equation~\eqref{eq:HH-eq-mult}, and showed that it specializes to $m_F$ at $t=1$~\cite[Th.~5.2]{hausel2022very}.
The following is a refinement of this result.

\begin{thm}\label{prop:v-stable-HHvsMZ}
    Let $F\in \Fix^\WG$.
    We have $\mu_F(t) = m_F(t)$ if and only if $F$ is very stable.
    In particular, for any very stable $F$ we have $\mu_F(1) = m_F$.
\end{thm}
\begin{proof}
    Let $F$ be very stable.
    Since the definition of $m_F(t)$ is local in $F$, we can restrict to the open subset $U\subset F$ of very stable points.
    By definition, the intersection of $W^-_U$ with all other core components is empty; therefore, $\sp_{W^-_F}(\mathcal{O}_\Core)$ is torsion-free on the nose, and so $m_F(t) = \mu_F(t)$.

    Now assume that $F$ is wobbly.
    Writing $\sp_{W^-_F}(\mathcal{O}_\Core)_{\mathrm{tf}}\coloneqq \sp_{W^-_F}(\mathcal{O}_\Core)/\mathcal{T}$, it follows from definitions that $\mu_F(t)-m_F(t) = \rk \loc_{W^-_F}([\mathcal{T}])$.
    We need to prove that this quantity is non-zero.
    By \cref{lem:rank-restr} and~\eqref{eq:loc-smooth}, it suffices to show that $\rk i^*[\mathcal{T}]\neq 0$, where $i:F\to N_{W^-_F}$ is the inclusion.
    Since all points of $F$ are wobbly, the support of $\mathcal{T}$ contains the entirety of $F$.
    In particular, we can pick an open $U\subset F$ such that $\mathcal{T}$ is locally free over $\mathcal{O}_U$. 
    Restricting to this open, we have $\rk i^*[\mathcal{T}] = \rk [i^* \mathcal{T}]\neq 0$, since $i^* \mathcal{T}$ is manifestly of positive rank.
\end{proof}
Recall that we denote the $\T$-weights on Hitchin base $\Base$ by $0<e_1\leq\ldots\leq e_n$.
\begin{cor}
    \cref{conj:eq-mult-positivity} holds for all very stable $F\in\Fix^\WG$.
    Moreover, in this case $m_F(t)$ is palindromic.
\end{cor}
\begin{proof}
    Let $p\in F$ be very stable, and use formula~\eqref{eq:HH-eq-mult}.
    Denoting the equivariant weights of $T^+_p$ by $a_1,\ldots, a_n$, we have 
    \begin{equation}\label{multiplicity_fla_very_stable}
        m_F(t) = \mu_F(t) = \frac{\prod_i [e_i]_t}{\prod_i [a_i]_t}.
    \end{equation}
    Since quantum numbers are palindromic, a Laurent monomial in them is palindromic as well.
    Furthermore $m_F(0)=1$, so that $m_F(t)$ is a polynomial.
\end{proof}

\begin{cor}\label{cor:correction-v-stable}
    Let $F\in\Fix^\WG$ be very stable, and take $\pi$ as in \cref{prop:cover-mult}.
    Then \cref{conj:finite-cover} holds for $d_{F,\pi}(t) = \delta_{\pi}(t)$.
\end{cor}
\begin{proof}
    Follows from \cref{prop:cover-mult}.
\end{proof}

We expect that \cref{prop:v-stable-HHvsMZ} can be strengthened to an inequality.
Observe that for any $F\in \Fix^\WG$ both numerator and denominator of $\mu_F(t)$ have a zero of degree $n$ at $t=1$; in particular, the specialization $\mu_F(1)$ is well defined.
\begin{con}\label{conj:mult-inequality}
    Assume that $\theta:\Hitch \to\Base$ is integrable.
    For any wobbly $F\in \Fix^\WG$, we have $\mu_F(1)<m_F$, and more generally $\mu_F(t)< m_F(t)$ for all $t>1$.
\end{con}

Assume that $\Hitch$ is of weight $k=1$.
Then the bottom component $\Fbot$ is in $\Fix^{\WG}$ is a very stable irreducible component of $\Core$ by \cref{cor:v-stable}.
Let us denote its equivariant multiplicity by $\mathbf{m}(t)$.
All normal weights of $\Fbot$ are equal to $1$, therefore by \cref{prop:v-stable-HHvsMZ}:
\begin{equation}\label{eq:mult-bottom}
    \mathbf{m}(t) = \mu_\Fbot(t) = \frac{\prod_{i=1}^n (1-t^{e_i})}{(1-t)^n} = \prod_{i=1}^{n} [e_i]_t.
\end{equation}

\begin{cor}\label{prop:divis}
    For each very stable $F\in \Fix^\WG$, the equivariant multiplicity $m_F(t)$ divides $\prod_{i} [e_i]_t.$
    In particular, $m_F(t) \mid \mathbf{m}(t)$ when $k=1$.
\end{cor}
\begin{proof}
    The first part follows \eqref{multiplicity_fla_very_stable}, and the second from \eqref{eq:mult-bottom}.
\end{proof}

\subsection{Two-dimensional systems}\label{ssec:2d-mult-ex}
Let us compute equivariant multiplicities for the integrable systems of \cref{ex:2d-int}.
For surfaces $S^{\mathrm{I}},S^{\mathrm{II}},S^{\mathrm{IV}}$ of Painlevé family, the fixed points are isolated.
The corresponding cores can be depicted as follows (see e.g.~\cite{SZZ}):
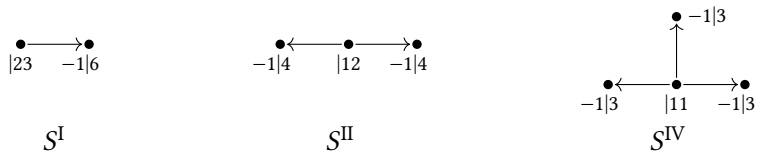
\[  
    \begingroup
    \arraycolsep=25pt
    \begin{array}[c]{ccc}
    \tikz[thick,xscale=.45,yscale=.45,font=\scriptsize,baseline=(current  bounding  box.center)]{
       \filldraw [black] (0,0) circle (3pt);
       \filldraw [black] (2,0) circle (3pt);
       \node at (0,-0.6) {$|23$};
       \node at (1.75,-0.6) {$-1|6$};
       \draw[->,thin] (0.2,0) -- (1.8,0);
		} &
    \tikz[thick,xscale=.45,yscale=.45,font=\scriptsize,baseline=(current  bounding  box.center)]{
       \filldraw [black] (0,0) circle (3pt);
       \filldraw [black] (2,0) circle (3pt);
       \filldraw [black] (-2,0) circle (3pt);
       \node at (0,-0.6) {$|12$};
       \node at (1.75,-0.6) {$-1|4$};
       \node at (-2.25,-0.6) {$-1|4$};
       \draw[->,thin] (0.2,0) -- (1.8,0);
       \draw[->,thin] (-0.2,0) -- (-1.8,0);
		} &
    \tikz[thick,xscale=.45,yscale=.45,font=\scriptsize,baseline=(current  bounding  box.center)]{
       \filldraw [black] (0,0) circle (3pt);
       \filldraw [black] (2,0) circle (3pt);
       \filldraw [black] (-2,0) circle (3pt);
       \filldraw [black] (0,2) circle (3pt);
       \node at (0,-0.6) {$|11$};
       \node at (1.75,-0.6) {$-1|3$};
       \node at (-2.25,-0.6) {$-1|3$};
       \node at (0.9,2) {$-1|3$};
       \draw[->,thin] (0.2,0) -- (1.8,0);
       \draw[->,thin] (-0.2,0) -- (-1.8,0);
       \draw[->,thin] (0,0.2) -- (0,1.8);
		} \\
    S^{\mathrm{I}} & S^{\mathrm{II}} & S^{\mathrm{IV}}
    \end{array}
    \endgroup
\]
Here, the nodes are connected components of $\Fix$, the arrows are the $\T$-curves between the elements of $\Fix$, and the numbers denote tangent weights; we separate positive from non-positive weights with a vertical line for readability.
The bottom component $\Fbot$ does not satisfy the weight gap condition, 
and the other components are maximal for the flow order, thus are very stable by \cref{cor:v-stable}.
As mentioned in \cref{ex:2d-int}, the $\T$-weight of the base is equal to $\e=6, 4, 3$ for $S^{\mathrm{I}},S^{\mathrm{II}},S^{\mathrm{IV}}$ respectively.
Thus by \cref{prop:v-stable-HHvsMZ}, we have $m_F(t)=\mu_F(t)=1$ for all $F\in \Fix \setminus \mathcal{N}$.
On the other hand, let us denote the tangent weights of $\Fbot$ by $a_1, a_2$.
Since $\Core_\Fbot=\Fbot$ is a single point, we have
\[
    m_\Fbot(t) = \mu_\Fbot(t) = \frac{1-t^\e}{(1-t^{a_1})(1-t^{a_2})} = 
    \begin{cases}
        \frac{1-t+t^2}{1-t},& \text{for }S^{\mathrm{I}}\\
        \frac{1+t^2}{1-t},& \text{for }S^{\mathrm{II}}\\
        \frac{1+t+t^2}{1-t},& \text{for }S^{\mathrm{IV}}
    \end{cases}
\]
This illustrates the importance of condition~\eqref{eq:weight-gap} in \cref{prop:non-eq-mult}.

For the surfaces of the parabolic family, the picture becomes richer.
All fixed components satisfy the weight gap condition, so each of them corresponds to an irreducible component of the core.
The equivariant multiplicities are given by the polynomials in blue in the core diagrams below:
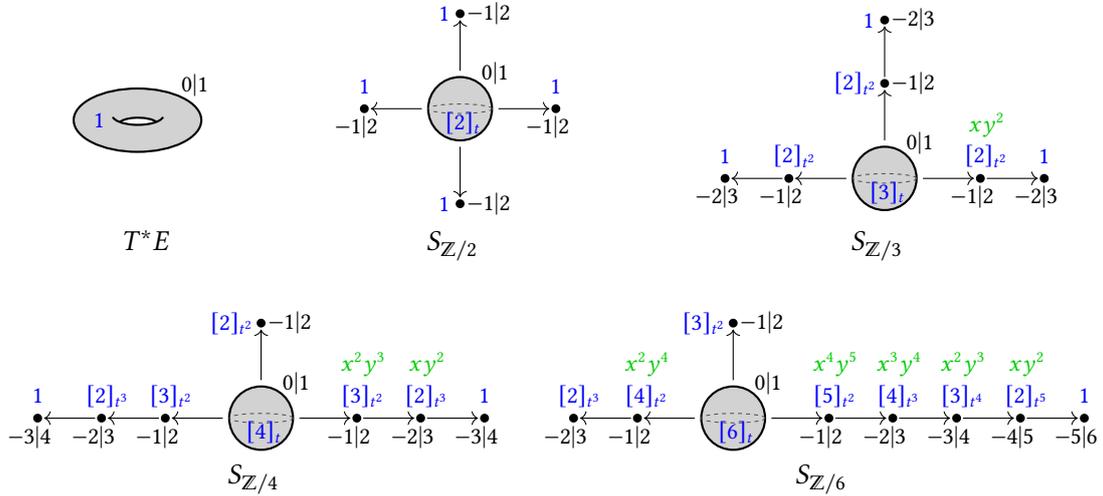
\begin{align*}
    \begingroup
    \arraycolsep=20pt
    \begin{array}[c]{ccc}
    \tikz[thick,xscale=.42,yscale=.42,font=\footnotesize,baseline=(current  bounding  box.center)]{
       \filldraw [color=black,fill=gray!35] (0,0) ellipse (2 and 1);
       \filldraw[color=gray!35,fill=white] (-0.7,-0.07) .. controls (0,0.17) .. (0.7,-0.07) .. controls (0,-0.2) ..  (-0.7,-0.07);
       \draw (0.8,0.1) arc (-10:-170:0.8 and 0.3);
       \draw (0.6,-0.07) arc (15:165:0.6 and 0.2);
       \node at (1.8,1.1) {$0|1$};
       \node[color=blue] at (-1.2,0) {$1$};
	} &
	\tikz[thick,xscale=.42,yscale=.42,font=\footnotesize,baseline=(current  bounding  box.center)]{
       \filldraw [color=black,fill=gray!35] (0,0) circle (1);
       \draw[dash pattern=on 1.5pt off 1.5pt,very thin,color=black!70] (0,0) ellipse (1 and 0.15);
       \filldraw [black] (0,3) circle (3pt);
       \filldraw [black] (0,-3) circle (3pt);
       \filldraw [black] (3,0) circle (3pt);
       \filldraw [black] (-3,0) circle (3pt);
       \node at (1.1,1.1) {$0|1$};
       \node[color=blue] at (0.1,-0.45) {$[2]_t$};
       \node at (2.75,-0.6) {$-1|2$};
       \node at (-3.25,-0.6) {$-1|2$};
       \node at (0.9,3) {$-1|2$};
       \node at (0.9,-3) {$-1|2$};
       \node[color=blue] at (3,0.7) {$1$};
       \node[color=blue] at (-3,0.7) {$1$};
       \node[color=blue] at (-0.5,3) {$1$};
       \node[color=blue] at (-0.5,-3) {$1$};
       \draw[->,thin] (0,1.2) -- (0,2.8);
       \draw[->,thin] (0,-1.2) -- (0,-2.8);
       \draw[->,thin] (1.2,0) -- (2.8,0);
       \draw[->,thin] (-1.2,0) -- (-2.8,0);
	} &
    \tikz[thick,xscale=.42,yscale=.42,font=\footnotesize,baseline=(current  bounding  box.center)]{
       \filldraw [color=black,fill=gray!35] (0,0) circle (1);
       \draw[dash pattern=on 1.5pt off 1.5pt,very thin,color=black!70] (0,0) ellipse (1 and 0.15);
       \filldraw [black] (0,3) circle (3pt);
       \filldraw [black] (3,0) circle (3pt);
       \filldraw [black] (-3,0) circle (3pt);
       \filldraw [black] (0,5) circle (3pt);
       \filldraw [black] (5,0) circle (3pt);
       \filldraw [black] (-5,0) circle (3pt);
       \node at (1.1,1.1) {$0|1$};
       \node[color=blue] at (0.1,-0.45) {$[3]_t$};
       \node at (2.75,-0.6) {$-1|2$};
       \node at (-3.25,-0.6) {$-1|2$};
       \node at (0.9,3) {$-1|2$};
       \node at (4.75,-0.6) {$-2|3$};
       \node at (-5.25,-0.6) {$-2|3$};
       \node at (0.9,5) {$-2|3$};
       \node[color=blue] at (3.2,0.7) {$[2]_{t^2}$};
       \node[color=blue] at (-2.8,0.7) {$[2]_{t^2}$};
       \node[color=blue] at (-0.9,3) {$[2]_{t^2}$};
       \node[color=blue] at (5,0.7) {$1$};
       \node[color=blue] at (-5,0.7) {$1$};
       \node[color=blue] at (-0.5,5) {$1$};
       \node[color=green!80!black] at (3.2,1.7) {$xy^2$};
       \draw[->,thin] (0,1.2) -- (0,2.8);
       \draw[->,thin] (1.2,0) -- (2.8,0);
       \draw[->,thin] (-1.2,0) -- (-2.8,0);
       \draw[->,thin] (0,3.2) -- (0,4.8);
       \draw[->,thin] (3.2,0) -- (4.8,0);
       \draw[->,thin] (-3.2,0) -- (-4.8,0);
	} \\
    T^*E & S_{\mathbb{Z}/2} & S_{\mathbb{Z}/3}
    \end{array}\endgroup\\[1em]
    \begin{array}[c]{cc}
    \tikz[thick,xscale=.42,yscale=.42,font=\footnotesize,baseline=(current  bounding  box.center)]{
       \filldraw [color=black,fill=gray!35] (0,0) circle (1);
       \draw[dash pattern=on 1.5pt off 1.5pt,very thin,color=black!70] (0,0) ellipse (1 and 0.15);
       \filldraw [black] (0,3) circle (3pt);
       \filldraw [black] (3,0) circle (3pt);
       \filldraw [black] (-3,0) circle (3pt);
       \filldraw [black] (5,0) circle (3pt);
       \filldraw [black] (-5,0) circle (3pt);
       \filldraw [black] (7,0) circle (3pt);
       \filldraw [black] (-7,0) circle (3pt);
       \node at (1.1,1.1) {$0|1$};
       \node[color=blue] at (0.1,-0.45) {$[4]_t$};
       \node at (2.75,-0.6) {$-1|2$};
       \node at (-3.25,-0.6) {$-1|2$};
       \node at (0.9,3) {$-1|2$};
       \node at (4.75,-0.6) {$-2|3$};
       \node at (-5.25,-0.6) {$-2|3$};
       \node at (6.75,-0.6) {$-3|4$};
       \node at (-7.25,-0.6) {$-3|4$};
       \node[color=blue] at (3.2,0.7) {$[3]_{t^2}$};
       \node[color=blue] at (-2.8,0.7) {$[3]_{t^2}$};
       \node[color=blue] at (-0.9,3) {$[2]_{t^2}$};
       \node[color=blue] at (5.2,0.7) {$[2]_{t^3}$};
       \node[color=blue] at (-4.8,0.7) {$[2]_{t^3}$};
       \node[color=blue] at (7,0.7) {$1$};
       \node[color=blue] at (-7,0.7) {$1$};
       \node[color=green!80!black] at (3.2,1.7) {$x^2y^3$};
       \node[color=green!80!black] at (5.2,1.7) {$xy^2$};
       \draw[->,thin] (0,1.2) -- (0,2.8);
       \draw[->,thin] (1.2,0) -- (2.8,0);
       \draw[->,thin] (-1.2,0) -- (-2.8,0);
       \draw[->,thin] (3.2,0) -- (4.8,0);
       \draw[->,thin] (-3.2,0) -- (-4.8,0);
       \draw[->,thin] (5.2,0) -- (6.8,0);
       \draw[->,thin] (-5.2,0) -- (-6.8,0);
	} & 
    \tikz[thick,xscale=.42,yscale=.42,font=\footnotesize,baseline=(current  bounding  box.center)]{
       \filldraw [color=black,fill=gray!35] (0,0) circle (1);
       \draw[dash pattern=on 1.5pt off 1.5pt,very thin,color=black!70] (0,0) ellipse (1 and 0.15);
       \filldraw [black] (0,3) circle (3pt);
       \filldraw [black] (3,0) circle (3pt);
       \filldraw [black] (-3,0) circle (3pt);
       \filldraw [black] (5,0) circle (3pt);
       \filldraw [black] (-5,0) circle (3pt);
       \filldraw [black] (7,0) circle (3pt);
       \filldraw [black] (9,0) circle (3pt);
       \filldraw [black] (11,0) circle (3pt);
       \node at (1.1,1.1) {$0|1$};
       \node[color=blue] at (0.1,-0.45) {$[6]_t$};
       \node at (2.75,-0.6) {$-1|2$};
       \node at (-3.25,-0.6) {$-1|2$};
       \node at (0.9,3) {$-1|2$};
       \node at (4.75,-0.6) {$-2|3$};
       \node at (-5.25,-0.6) {$-2|3$};
       \node at (6.75,-0.6) {$-3|4$};
       \node at (8.75,-0.6) {$-4|5$};
       \node at (10.75,-0.6) {$-5|6$};
       \node[color=blue] at (3.2,0.7) {$[5]_{t^2}$};
       \node[color=blue] at (-2.7,0.7) {$[4]_{t^2}$};
       \node[color=blue] at (-0.9,3) {$[3]_{t^2}$};
       \node[color=blue] at (5.2,0.7) {$[4]_{t^3}$};
       \node[color=blue] at (-4.8,0.7) {$[2]_{t^3}$};
       \node[color=blue] at (7.2,0.7) {$[3]_{t^4}$};
       \node[color=blue] at (9.2,0.7) {$[2]_{t^5}$};
       \node[color=blue] at (11,0.7) {$1$};
       \node[color=green!80!black] at (3.2,1.7) {$x^4y^5$};
       \node[color=green!80!black] at (5.2,1.7) {$x^3y^4$};
       \node[color=green!80!black] at (7.2,1.7) {$x^2y^3$};
       \node[color=green!80!black] at (9.2,1.7) {$xy^2$};
       \node[color=green!80!black] at (-2.7,1.7) {$x^2y^4$};
       \draw[->,thin] (0,1.2) -- (0,2.8);
       \draw[->,thin] (1.2,0) -- (2.8,0);
       \draw[->,thin] (-1.2,0) -- (-2.8,0);
       \draw[->,thin] (3.2,0) -- (4.8,0);
       \draw[->,thin] (-3.2,0) -- (-4.8,0);
       \draw[->,thin] (5.2,0) -- (6.8,0);
       \draw[->,thin] (7.2,0) -- (8.8,0);
       \draw[->,thin] (9.2,0) -- (10.8,0);
	} \\
    S_{\mathbb{Z}/4} & S_{\mathbb{Z}/6}
    \end{array}
\end{align*}
Observe that this confirms \cref{conj:eq-mult-positivity} for $2$-dimensional integrable systems.
Note also that the divisibility of \cref{prop:divis} does not hold for wobbly components. For instance, for the intermediate component in $S_{\mathbb{Z}/3}$ we have $[2]_{t^2} = 1+t^2 \nmid 1+t+t^2 = [3]_t$.

Let us explain how these multiplicities are obtained.
For bottom components, they are computed from formula~\eqref{eq:mult-bottom}; recall from \cref{ex:2d-int} that in $S_{\Z/e}$ the base has weight $\e$.
Top components in flow order are very stable by \cref{cor:v-stable}, so their multiplicities follow from \cref{prop:v-stable-HHvsMZ}.
For the intermediate components, passing to an étale neighborhood we find ourselves in the situation of \cref{ex:C2-mult}.
Let us compute the local equations of $\Core$.
At each isolated fixed point $p$, denote the local coordinate of negative, resp. positive, weight direction by $x$, resp. $y$.
When $p$ is very stable, the local equation of $\Core$ at $p$ is $y^{m_p}$.
When it is not, we instead get some monomial $x^{b_1}y^{b_2}$ of $\T$-weight $\e$.
However, for any $\T$-curve from 
$p$ to $q$ the power of $x$ at $p$ must be the same as the power of $y$ at $q$.
This uniquely determines the local equations, which are given by the polynomials in green in the diagrams above.
We conclude by appealing to formulas~\eqref{eq:2d-m}.

\begin{prop}\label{prop:2d-fin-cover}
    \cref{conj:finite-cover} holds for $2$-dimensional integrable systems.
\end{prop}
\begin{proof}
    We only need to treat wobbly components $p\in\Fix$ by \cref{cor:correction-v-stable}.
    Note that the only finite $\T$-equivariant map $\pi:\C_e\to \C_{de}$ of degree $d$ up to a scalar is $z\mapsto z^d$.
    As explained above, the local equation of $\Core$ at $p$ is of the form $x^{b_1}y^{b_2}$.
    We deduce from the formula~\eqref{eq:2d-m} that if $m_p(t) = [b_2]_{t^a}$, then $m_{p'}(t) = [db_2]_{t^a}$, and so $d_{p,\pi}(t) = [d]_{t^{ab_2}}$. 
\end{proof}

\begin{ex}\label{ex:degree-correction}
    Let us have a closer look at the component $\Core_p$ of multiplicity $[2]_{t^2}$ in $S_{\mathbb{Z}/3}$.
    As explained above, the local equation of $\Core$ at $p$ is $xy^2$.
    Let $\pi: \mathbb{C}_3 \to \mathbb{C}_6$, $\pi(z) = z^2$, and consider the composition $\theta'$ of integrable system $\theta: S_{\mathbb{Z}/3}\to \C_3$ with $\pi$.
    By the proof of \cref{prop:2d-fin-cover}, we have $d_{p,\pi}(t) = [2]_{t^4}$; in fact, locally at $p$ the system $\theta'$ is isomorphic to the neighborhood of the fixed point of multiplicity $[4]_{t^2}$ in $S_{\mathbb{Z}/6}$.
    On the other hand $\delta_\pi(t) = \frac{1-t^6}{1-t^3} = [2]_{t^3}$, and so $\delta_\pi(t)\neq d_{p,\pi}(t)$.
\end{ex}

\section{Multiplicities for Hilbert schemes}\label{sec:mult-for-Hilb}
In this section, we compute multiplicities for integrable systems of \cref{ex:Hilb}, i.e., Hilbert schemes of points on 2-dimensional integrable systems.
\subsection{\'Etale neighbourhoods}
Our arguments will proceed by reducing all computations to the case of Hilbert schemes of points on $\mathbb{C}^2$ using appropriate étale neighborhoods.

Let $S$ be a smooth quasi-projective surface, $p_1,\ldots,p_r\in S$ closed points, and fix an étale neighbourhood $\jmath_i:U_i\to S$, $U_i\ni p_i$ for each point.
Further, let $\underline{p} = \sum_{i=1}^{r} k_ip_i\in \Sym^n S$, $k_i\geq 0$, and $I\subset \mathcal{O}_S$ an ideal with support $\underline{p}$.
Note that we can write $I = I_1\cap\ldots \cap I_r$, where each $I_i$ is supported at $p_i$ and has colength $k_i$.

\begin{prop}\label{prop:Hilb-etale}
    Denote $\Sym^{\underline{k}} S \coloneqq \prod_i \Sym^{k_i} S$, and $\Delta\subset \Sym^{\underline{k}} S$ the closed subvariety of non-dijoint tuples. 
    Let $\widetilde{\Hilb}^{k_i} U_i\subset \Hilb^{k_i} U_i$ be the open subset of all ideals $J$ such that for any $p\in S$, at most one point of $\jmath_i^{-1}(p)$ lies in the set-theoretic support of $J$.
    Define
    \[
        U_I \coloneqq \left(\prod_i \widetilde{\Hilb}^{k_i} U_i\right) \times_{\Sym^{\underline{k}} S} \Sym^{\underline{k}} S\setminus \Delta.
    \]
    Then the map $\jmath_I:U_I\to \Hilb^n S$, $(J_1,\ldots,J_r)\mapsto J_1\cap\ldots \cap J_r$ is étale, and the following diagram commutes:
    \[
    \begin{tikzcd}[row sep=small, column sep=small]
        U_I\ar[r]\ar[d] & \Hilb^n S\ar[d] \\
        \prod_i \Sym^{k_i} U_i\ar[r] & \Sym^n S 
    \end{tikzcd}
    \]
\end{prop}

\begin{rmk}
    An analytically-minded reader can replace each étale neighbourhood $U_i$ with a small ball $\mathbb{B}$, and $\widetilde{\Hilb}^{k} U_i$ with the Douady space of $k$ points on $\mathbb{B}$.
\end{rmk}

\begin{proof}
    The commutativity of the diagram is obvious.
    We can factorize the map $\jmath_I$ as follows:
    \[
        U_I \to \left(\prod_i \Hilb^{k_i} S\right) \times_{\Sym^{\underline{k}} S} (\Sym^{\underline{k}} S\setminus \Delta) \to \Hilb^n S.
    \]
    It suffices to show that both these maps are étale.
    For the second one, see~\cite[Prop~3.3]{bertin}.
    For the first one, we consider the map $\phi:\widetilde{\Hilb}^k U_i \to \Hilb^k S$.
    Observe that it is well defined.
    Indeed, given $J\subset \mathcal{O}_{U_i}$, by adjunction we have $\mathcal{O}_S\to (\jmath_i)_*(\mathcal{O}_S/J)$, which is surjective by our assertion on the support on $J$.
    Thus $\phi(J) = \ker(\mathcal{O}_S\to (\jmath_i)_*(\mathcal{O}_S/J))$, and the fact that this map is étale follows from the same exact argument as in~\cite[Prop.~3.3]{bertin}.
\end{proof}

\begin{rmk}\label{rmk:tg-weights}
    Assume that $\T$ acts on $S$, all $p_i$'s are fixed by $\T$, and the étale neighbourhoods $U_i$ are equivariant.
    Then the neighbourhood $\jmath_I:U_I\to \Hilb^n S$ is also $\T$-equivariant.
    In particular, by Luna's slice theorem we can assume that each $U_i$ admits a $\T$-equivariant étale map $U_i\to T_{p_i} S$, and so computing tangent weights in $\Hilb^n S$ reduces to the case $S = \mathbb{C}^2$.
\end{rmk}

\subsection{Partitions}
As is customary with Hilbert schemes of points, we need some combinatorics of partitions.
We denote the set of partitions of $n$ by $\Part(n)$, and write $\Part = \bigsqcup_{n\geq 1} \Part(n)$.
An element of $\Part(n)$ is $\lambda = (\lambda_1\geq \lambda_2\geq\ldots)$ with $|\lambda|\coloneqq \sum_i \lambda_i = n$.
We draw $\lambda\in\Part$ as a Young diagram (in French notation), and denote by $\lambda'$ the transposed partition.
In particular, we write $\ell(\lambda) \coloneqq \lambda'_1$ for the length of $\lambda$.
For each box $s=(i,j)\in\lambda$ we define its leg, arm, hook lengths:
\[
\ytableausetup{centertableaux,boxsize=1.25em}
\begin{ytableau}
\none &  & a \\
\none[\vdots] &  & a & \\
\none[{\lambda_2}] &  & s & l & l & l\\
\none[{\lambda_1}] &  & & & & \\
\none & \none[{\lambda'_1}] & \none[{\lambda'_2}] & \none[\dots]
\end{ytableau}\qquad\qquad
\begin{aligned}
    &l(s) = \lambda_j-i,\quad a(s) = \lambda'_i-j,\\
    &h(s) = l(s)+a(s)+1.
\end{aligned}
\]
Given $\lambda\in\Part(n)$, we define its ($t$-)multinomial coefficient by
\[
    \binom{n}{\lambda} \coloneqq \frac{n!}{\prod_i \lambda_i!},\qquad \qnom{n}{\lambda}_t \coloneqq \frac{[n]!_t}{\prod_i [\lambda_i]!_t}, \quad [n]!_t = [n]_t[n-1]_t\dots [1]_t.
\]
For a finite set $X$, we denote the set of $X$-multipartitions by 
$\Part^X \coloneqq \bigsqcup_n \Part^X(n),$ 
where an element of $\Part^X(n)$ is an $X$-tuple of partitions $\underline{\lambda}=(\lambda(x))_{x\in X}$, satisfying $\sum_{x\in X} |\lambda(x)| = n$.
We can alternatively treat partitions as unordered multisets of positive numbers.
In these terms, we have a natural map $\cup: \Part^X(n)\to \Part(n)$, which sends $\underline{\lambda}$ to $\cup_x \lambda(x)$.

\subsection{Fixed loci}
Let $S$ be a two-dimensional smooth integrable system from \cref{ex:2d-int}. Recall that we denote the $\T$-weight of its symplectic form by $k$.
By \BB decomposition of $S$, at most one connected component of $S^{\T}$ is one-dimensional.
Let us denote this curve by $C$ (by abuse of notation, $C = \varnothing$ for Painlevé surfaces), and write 
\begin{equation}\label{eq:fix-S}
    S^{\T} = C \sqcup \{ p_1,\ldots,p_r\}.
\end{equation}
Given $p\in S^\T$ satisfying~\eqref{eq:weight-gap}, we have the homogeneous coordinates $(x,y)$ on $T_p S$, where $x$, resp. $y$ is the direction with non-negative, resp. positive weight.
For each $\lambda\in\Part(n)$, we then define an ideal 
\begin{equation}\label{eq:fixed-ideals}
    I^{\lambda}_p\coloneqq(x^{\lambda_1}, x^{\lambda_2}y,\ldots, x^{\lambda_q}y^{q-1},y^q).
\end{equation}
The following descriptions of ${\T}$-fixed loci of Hilbert schemes of points are well known.
\begin{prop}\label{prop:Hilb-fix-C2}
    Let $S=\mathbb{C}^2$, $0\in \C^2$ the origin, and ${\T}$ acts by $t(x,y) = (t^{-l}x, t^{k+l}y)$ for $k>0$, $l\geq 0$.
    There is a bijection $\Part(n) \to \Fix$, $\lambda\mapsto F_\lambda$.
    More precisely, for $\lambda = (\lambda_1,\lambda_2,\ldots)\in\Part(n)$:
    \begin{itemize}
        \item If $l\neq 0$, $F_\lambda\in \Fix$ is the ideal $I^{\lambda}_0$;
        \item If $l=0$, $F_\lambda\in \Fix$ is the attracting locus of $I^{\lambda}$ under the torus action $t(x,y) = (tx, y)$. It consists of ideals of the form $I_{x_1}^{\mu(1)}\cap I_{x_2}^{\mu(2)}\cap\ldots$, where $\underline{\mu}\in \cup^{-1}(\lambda)\subset \Part^s$, and $x_i\in \mathbb{C}$.
    \end{itemize}
\end{prop}
\begin{proof}
    See \cite[Cor.~7]{evain2004irreducible} for $l\neq 0$, and \cite[Sec.~7]{Nak99} for $l=0$.
\end{proof}

For $\lambda \in \Part$, we write $\Sym^\lambda C \coloneqq \prod_{i\geq 1} \Sym^{\lambda_i-\lambda_{i+1}} C$.\footnote{Compared to~\cite{Nak99}, we swap $\lambda$ with $\lambda'$ in our notations.}
\begin{cor}\label{cor:Hilb-fixed-loci}
    Let $S$ be as in \cref{ex:2d-int}.
    We have $\Fix(\Hilb^n S) = \Part^{\Fix(S)}(n)$, moreover
    \[
        (\Hilb^n S)^{\T} = \bigsqcup_{\underline{\lambda}\in\Part^{\Fix(S)}(n)} \Sym^{\lambda(C)} C \times \prod_{i=1}^r I_{p_i}^{\lambda(p_i)}
    \]
\end{cor}
\begin{proof}
    Restricting to the open $T^*C\subset S$, resp. passing to an étale cover as in \cref{prop:Hilb-etale}, it suffices to prove the claim for $S=T^*C$, resp. $\mathbb{C}^2$. 
    The latter is treated in \cref{prop:Hilb-fix-C2}, while for the former we have
    \[
        (\Hilb^n T^*C)^\T = \bigsqcup_{\lambda\in\Part(n)} \Sym^{\lambda} C
    \]
    by~\cite[Prop.~7.5]{Nak99}.
\end{proof}

Denote by $F_{\underline{\lambda}}$ the $\T$-fixed component in $\Hitch\coloneqq\Hilb^n S$ which corresponds to the multipartition $\underline{\lambda}\in\Part^{\Fix(S)}(n)$.
Let $l_{p_i} \coloneqq -\wt(T_{p_i}^-)$, $1\leq i\leq r$, and $l_C\coloneqq 0$.
Note that $\wt(T_{p_i}^+) = k+l_{p_i}$, and $\wt(T_{c}^+) = k$ for $c\in C$.
\begin{prop}\label{prop:fixed-tg-weights}
    Let $F_{\underline{\lambda}}\in \Fix$, and $x\in F_{\underline{\lambda}}$.
    We have
    \begin{equation*}
        \chi_t(T_x \Hitch) = \sum_{p\in\Fix(S)} \sum_{s\in\lambda(p)} t^{-l_ph(s)-ka(s)} + t^{l_ph(s)+k(a(s)+1)}.
    \end{equation*}
    In particular, separating positive and negative weights,
    \begin{equation}\label{eq:Tplus-tangent}
        \chi_t(T^+_x) = \sum_{p\in\Fix(S)} \sum_{s\in\lambda(p)} t^{l_ph(s)+k(a(s)+1)},\quad
        \chi_t(T^-_x) = - \lambda(C)_1 + \sum_{p\in\Fix(S)} \sum_{s\in\lambda(p)} t^{-l_ph(s)-ka(s)}.
    \end{equation}
\end{prop}
\begin{proof}
    The local case $S = \mathbb{C}^2$ is easily deduced from~\cite[Prop.~5.8]{Nak99}.
    Our statement then immediately follows from \cref{rmk:tg-weights}. 
\end{proof}

\subsection{Flow order}
The flow order on Hilbert schemes can be completely described for surfaces $S$ of the parabolic family.
For any $\underline{\lambda}\in \Part^{\Fix(S)}(n)$ and $F\in \Fix(S)$, we write $\lambda(F) = (\lambda(F)_1\geq \ldots \geq\lambda(F)_n)$, where $\lambda(F)_i=0$ for $i>\lambda(F)'_1$.
Given $\underline{\lambda},\underline{\mu}\in \Part^{\Fix(S)}(n)$, we say that $\underline{\lambda}$ \textit{dominates} $\underline{\mu}$ if for any $F\in \Fix(S)$ and $1\leq i\leq n$, we have
\begin{equation}\label{eq:multipart-dom}
    \sum_{F'\prec F}|\lambda(F')| + \sum_{j\leq i}\lambda(F)_j \geq \sum_{F'\prec F}|\mu(F')| + \sum_{j\leq i}\mu(F)_j.
\end{equation}
This clearly defines a partial order on $\Fix = \Fix(\Hilb^n S)$.
The following lemma gives us a convenient combinatorial preorder, whose transitive closure is the dominance order above.

\begin{lm}\label{lm:dom-skelethon}
    Suppose that $\underline{\lambda}$ dominates $\underline{\mu}$.
    There exists a sequence $\underline{\lambda} = \underline{\lambda}^0$, $\underline{\lambda}^1,\ldots, \underline{\lambda}^N = \underline{\mu}$ of elements of $\Part^{\Fix(S)}(n)$, such that for every $1\leq i \leq N$ we have:
    \begin{enumerate}[label=(\roman*)]
        \item either there exists $F\in \Fix(S)$ such that $\lambda^{i-1}(F') = \lambda^{i}(F')$ for all $F'\neq F$, and $\lambda^{i-1}(F)$ dominates $\lambda^{i}(F)$ in the usual sense in $\Part$;
        \item or there exist $F,G\in \Fix(S)$ together with a $\T$-curve from $F$ to $G$, such that $\lambda^{i-1}(F') = \lambda^{i}(F')$ for all $F'\neq F,G$, $(\lambda^{i-1}(F))'_j = (\lambda^{i}(F))'_j$ and $\lambda^{i-1}(G)_j = \lambda^{i}(G)_j$ for $j>1$, and $(\lambda^{i}(F))'_1 = (\lambda^{i-1}(F))'_1 - 1$, $\lambda^{i}(G)_1 = \lambda^{i-1}(G)_1 + 1$.
    \end{enumerate}
\end{lm}
\begin{proof}
    It is evident that both moves $(i)$ and $(ii)$ of \cref{lm:dom-skelethon} are compatible with the dominance order.
    Let $A\subset \Fix(S)$ be a subset such that $F'\in A$ for all $F'\preceq F$, $F\in A$.
    We proceed by induction on $A$.
    Namely, we fix a collection of such subsets $A_0=\Fix(S)\supset A_1\supset\ldots\supset A^N=\varnothing$ with $|A_i\setminus A_{i+1}|=1$, and construct the multipartitions $\underline{\lambda}^i$, such that $\lambda^i(F) = \lambda(F)$ for all $F\not\in A_i$, and each $\underline{\lambda}^{i+1}$ is obtained from $\underline{\lambda}^i$ by a sequence of moves $(i)$, $(ii)$.

    Let $F$ be the unique element in $A_i\setminus A_{i+1}$.
    By dominance assumption, $\lambda^i(F)$ dominates $\mu(F)$ as a partition.
    If $|\lambda^i(F)| = |\mu(F)|$, we can use the move $(i)$ to get from $\lambda^i(F)$ to $\mu(F)$, and set $\underline{\lambda}^{i+1}$ to be equal to $\underline{\lambda}^{i}$ on $\Fix(S)\setminus \{F\}$, and $\lambda^{i+1}(F) = \mu(F)$.
    Otherwise, let $\nu$ be the partition obtained from $\mu(F)$ by appending several $1$'s, in such a way that $|\nu|=\lambda^i(F)$.
    By our assumptions $\lambda^i(F)$ dominates $\nu$, so we can pass from the former to the latter with move $(i)$.
    Further, we can use the move $(ii)$ to pass from $\nu$ to $\mu(F)$ for the price of augmenting the partitions at components in $A_{i+1}$ adjacent to $F$.
    If $F\neq C$, there is only one such component (see figures in \cref{ssec:2d-mult-ex}), so we set $\underline{\lambda}^{i+1}$ to be the resulting element.
    If $F=C$, we are necessarily at the base $i=0$ of our induction.
    In this case, there is a unique choice of where to send all the $1$'s; it is determined by the conditions $\sum_{F\in B}|\mu(F)| = \sum_{F\in B}|\lambda^1(F)|$, where $B$ is the set of all components lying on one given leg of the core diagram. 
    Since at each step we only modified partitions in $A_i$, the sequence $\underline{\lambda}^{i}$ terminates in $\underline{\mu}$.
\end{proof}

It is a standard fact that we can decompose a move $(i)$ into a sequence of moves of the following type: 

\begin{enumerate}[label=\textit{(\roman*')}]
    \item There exists $F\in \Fix(S)$ and $j_1<j_2$ positive, such that $\lambda^{i-1}(F') = \lambda^{i}(F')$ for all $F'\neq F$, $\lambda^{i-1}(F)_j = \lambda^{i}(F)_j$ for all $j\neq j_1,j_2$, $\lambda^{i}(F)_{j_1} = \lambda^{i-1}(F)_{j_1}-1$, $\lambda^{i}(F)_{j_2} = \lambda^{i-1}(F)_{j_2}+1$, and $\lambda^{i}(F)_{j_1}=\lambda^{i}(F)_{j}\geq\lambda^{i}(F)_{j_2}$ for all $j_1< j<j_2$.
\end{enumerate}

\begin{ex}
    Let us illustrate the moves described above for $\Hilb^2 S_{\mathbb{Z}/3}$.
    In this case, there is no difference between moves $(i)$ and $(i')$.
    We concentrate on the fixed points supported on one of the legs of the core of $\Hilb^2 S_{\mathbb{Z}/3}$, which we denote by $p_0\in C$, $p_1$, $p_2$.
    We draw the tangent direction of $C$ as a thick line, $\T$-curves as normal lines, and the upward flow escaping from the top point as a dotted line.
    The moves are then depicted as follows:

    
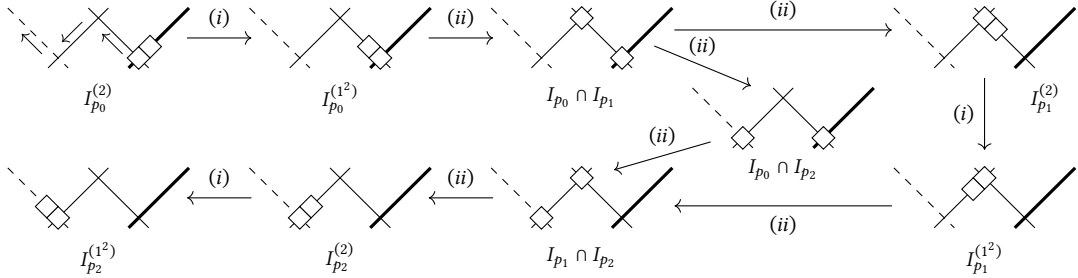
\[
    \tikz[xscale=.53,yscale=.53,font=\scriptsize]{
      \draw [dashed] (-0.25,1.25) -- (1.25,-0.25);
      \draw [very thin,->] (0.6,0.1) -- (0.1,0.6);  
      \draw (0.75,-0.25) -- (2.25,1.25);
      \draw [very thin,->] (1.6,0.9) -- (1.1,0.4); 
      \draw (1.75,1.25) -- (3.25,-0.25);
      \draw [very thin,->] (2.6,0.1) -- (2.1,0.6);  
      \draw [very thick] (2.75,-0.25) -- (4.25,1.25);
      \filldraw[draw=black,fill=white,rotate around={45:(3,0)}] (2.8,-0.2) rectangle (3.2,0.2);
      \filldraw[draw=black,fill=white,rotate around={45:(3,0)}] (3.2,-0.2) rectangle (3.6,0.2);

      \draw [->] (4.2,0.5) -- (5.8,0.5) node[midway,above] {$(i)$};

      \begin{scope}[shift={(6,0)}]
      \draw [dashed] (-0.25,1.25) -- (1.25,-0.25); 
      \draw (0.75,-0.25) -- (2.25,1.25); 
      \draw (1.75,1.25) -- (3.25,-0.25);
      \draw [very thick] (2.75,-0.25) -- (4.25,1.25);
      \filldraw[draw=black,fill=white,rotate around={45:(3,0)}] (2.8,-0.2) rectangle (3.2,0.2);
      \filldraw[draw=black,fill=white,rotate around={45:(3,0)}] (2.8,0.2) rectangle (3.2,0.6);
      \end{scope}

      \draw [->] (10.2,0.5) -- (11.8,0.5) node[midway,above] {$(ii)$};

      \begin{scope}[shift={(12,0)}]
      \draw [dashed] (-0.25,1.25) -- (1.25,-0.25); 
      \draw (0.75,-0.25) -- (2.25,1.25); 
      \draw (1.75,1.25) -- (3.25,-0.25);
      \draw [very thick] (2.75,-0.25) -- (4.25,1.25);
      \filldraw[draw=black,fill=white,rotate around={45:(3,0)}] (2.8,-0.2) rectangle (3.2,0.2);
      \filldraw[draw=black,fill=white,rotate around={45:(2,1)}] (1.8,0.8) rectangle (2.2,1.2);
      \end{scope}

      \draw [->] (16.3,0.7) -- (21.7,0.7) node[midway,above] {$(ii)$};
      \draw [->] (15.8,0.3) -- (18.2,-0.7) node[midway,above] {$(ii)$};
      \draw [->] (17.2,-2.1) -- (14.8,-2.8) node[midway,above] {$(ii)$};

      \begin{scope}[shift={(22,0)}]
      \draw [dashed] (-0.25,1.25) -- (1.25,-0.25); 
      \draw (0.75,-0.25) -- (2.25,1.25); 
      \draw (1.75,1.25) -- (3.25,-0.25);
      \draw [very thick] (2.75,-0.25) -- (4.25,1.25);
      \filldraw[draw=black,fill=white,rotate around={45:(2,1)}] (1.8,0.8) rectangle (2.2,1.2);
      \filldraw[draw=black,fill=white,rotate around={45:(2,1)}] (1.8,0.4) rectangle (2.2,0.8);
      \end{scope}

      \draw [->] (24,-0.5) -- (24,-2.3) node[midway,left] {$(i)$};
      \draw [->] (21.7,-3.7) -- (16.3,-3.7) node[midway,below] {$(ii)$};

      \begin{scope}[shift={(17,-2)}]
      \draw [dashed] (-0.25,1.25) -- (1.25,-0.25); 
      \draw (0.75,-0.25) -- (2.25,1.25); 
      \draw (1.75,1.25) -- (3.25,-0.25);
      \draw [very thick] (2.75,-0.25) -- (4.25,1.25);
      \filldraw[draw=black,fill=white,rotate around={45:(3,0)}] (2.8,-0.2) rectangle (3.2,0.2);
      \filldraw[draw=black,fill=white,rotate around={45:(1,0)}] (0.8,-0.2) rectangle (1.2,0.2);
      \end{scope}

      \begin{scope}[shift={(22,-4)}]
      \draw [dashed] (-0.25,1.25) -- (1.25,-0.25); 
      \draw (0.75,-0.25) -- (2.25,1.25); 
      \draw (1.75,1.25) -- (3.25,-0.25);
      \draw [very thick] (2.75,-0.25) -- (4.25,1.25);
      \filldraw[draw=black,fill=white,rotate around={45:(2,1)}] (1.8,0.8) rectangle (2.2,1.2);
      \filldraw[draw=black,fill=white,rotate around={45:(2,1)}] (1.4,0.8) rectangle (1.8,1.2);
      \end{scope}

      \begin{scope}[shift={(12,-4)}]
      \draw [dashed] (-0.25,1.25) -- (1.25,-0.25); 
      \draw (0.75,-0.25) -- (2.25,1.25); 
      \draw (1.75,1.25) -- (3.25,-0.25);
      \draw [very thick] (2.75,-0.25) -- (4.25,1.25);
      \filldraw[draw=black,fill=white,rotate around={45:(2,1)}] (1.8,0.8) rectangle (2.2,1.2);
      \filldraw[draw=black,fill=white,rotate around={45:(1,0)}] (0.8,-0.2) rectangle (1.2,0.2);
      \end{scope}
      
      \draw [->] (11.8,-3.5) -- (10.2,-3.5) node[midway,above] {$(ii)$};

      \begin{scope}[shift={(6,-4)}]
      \draw [dashed] (-0.25,1.25) -- (1.25,-0.25); 
      \draw (0.75,-0.25) -- (2.25,1.25); 
      \draw (1.75,1.25) -- (3.25,-0.25);
      \draw [very thick] (2.75,-0.25) -- (4.25,1.25);
      \filldraw[draw=black,fill=white,rotate around={45:(1,0)}] (0.8,-0.2) rectangle (1.2,0.2);
      \filldraw[draw=black,fill=white,rotate around={45:(1,0)}] (1.2,-0.2) rectangle (1.6,0.2);
      \end{scope}

      \draw [->] (5.8,-3.5) -- (4.2,-3.5) node[midway,above] {$(i)$};

      \begin{scope}[shift={(0,-4)}]
      \draw [dashed] (-0.25,1.25) -- (1.25,-0.25); 
      \draw (0.75,-0.25) -- (2.25,1.25); 
      \draw (1.75,1.25) -- (3.25,-0.25);
      \draw [very thick] (2.75,-0.25) -- (4.25,1.25);
      \filldraw[draw=black,fill=white,rotate around={45:(1,0)}] (0.8,-0.2) rectangle (1.2,0.2);
      \filldraw[draw=black,fill=white,rotate around={45:(1,0)}] (0.8,0.2) rectangle (1.2,0.6);
      \end{scope}

      \node at (2,-1) {$I^{(2)}_{p_0}$};
      \node at (8,-1) {$I^{(1^2)}_{p_0}$};
      \node at (14,-1) {$I_{p_0}\cap I_{p_1}$};
      \node at (25.5,-1) {$I^{(2)}_{p_1}$};

      \node at (19,-2.8) {$I_{p_0}\cap I_{p_2}$};

      \node at (2,-5) {$I^{(1^2)}_{p_2}$};
      \node at (8,-5) {$I^{(2)}_{p_2}$};
      \node at (14,-5) {$I_{p_1}\cap I_{p_2}$};
      \node at (24,-5) {$I^{(1^2)}_{p_1}$};
	}
    \]
\end{ex}

\begin{thm}\label{prop:Hilb-flow-is-dom}
    The flow order on $\Fix$ coincides with the antidominance order.
    In other words, we have $F_{\underline{\lambda}}\preceq F_{\underline{\mu}}$ if and only if $\underline{\lambda}$ dominates $\underline{\mu}$.
\end{thm}

The proof of this theorem is inspired by~\cite[§4]{nakajima1996jack}, and requires some preparation. 
Recall that the core of $S$ is the union of $C$ with projective lines $\overline{W^-_{p_i}}$, $1\leq i\leq r$.
Similarly, $\Core\subset\Hilb^n S$ consists of ideals $I$ with $\mathcal{O}_S/I$ supported on the core of $S$.
Slightly abusing notation, let us denote $\mathbb{P}_C \coloneqq C$, and $\mathbb{P}_p \coloneqq \overline{W^-_{p}}$.

For any $p\in \Fix(S)$, $0\leq i\leq n-1$, and $I\in \Core$, let us define
\[
    D_{p,i} \coloneqq i\mathbb{P}_{p} + \sum_{p'\preceq p} n\mathbb{P}_{p'},\qquad 
    a_{p,i}(I) \coloneqq \mathrm{length}(\mathcal{O}_S / (\mathcal{O}_S(-D_{p,i})+ I)).
\]
It is clear from the definition that $a_{p,i}$ is $\T$-invariant and upper semicontinuous.
In particular, we have closed subvarieties $\{I:a_{p,i}(I)\geq a\}$, $a\in \mathbb{Z}_{\geq 0}\subset \Core$ for any $p,i$.
Given $\underline{\lambda}\in\Part^{\Fix(S)}(n)$, we consider the subvariety $X_{\underline{\lambda}}\subset \Core$, which is an intersection of closed subvarieties above:
\[
    X_{\underline{\lambda}}\coloneqq \left\{I:a_{p,i}(I)\geq \sum_{p'\prec p}|\lambda(F')| + \sum_{j\leq i}\lambda(p)_j \text{ for all }p\in\Fix(S), 0\leq i\leq n-1\right\}.
\]

\begin{lm}\label{lem:geometric-dom}
    Let $\underline{\lambda},\underline{\mu} \in\Part^{\Fix(S)}(n)$. 
    The multipartition $\underline{\lambda}$ is dominated by $\underline{\mu}$ if and only if $F_{\underline{\lambda}}$ is contained in $X_{\underline{\mu}}$.
\end{lm}
\begin{proof}
    Using the explicit description of $\T$-fixed ideals in \cref{prop:Hilb-fix-C2}, it is straightforward to check that the inequalities in the definition of $X_{\underline{\mu}}$ specialize to~\eqref{eq:multipart-dom} for any ideal in $F_{\underline{\lambda}}$.
\end{proof}

\begin{lm}\label{lem:dom-loc-cst}
    Let $p\in \Fix(S)$, $0\leq i\leq n-1$.
    The function $a_{p,i}: \Core\to \Z_{\geq 0}$ is constant on each downward flow $W^-_F$, $F\in \Fix$. 
\end{lm}
\begin{proof}
    Let us decompose the core of $S$ into the disjoint union of downward flows $\bigsqcup_q W^-_q$.
    All the ideals we consider are supported on the core of $S$, and length of a sheaf is additive under decomposition of the base space.
    Let $I_{\overline{\lambda}} = \prod_p I_q^{\lambda(q)}$ be a $\T$-fixed ideal.
    Any ideal in $W^-_{I_{\overline{\lambda}}}$ can be written as $\prod_q I'_q$, where $\lim\limits_{t\to\infty} I'_q = I_q^{\lambda(q)}$, and each $I'_q$ is supported on $W^-_q$.
    In particular, we have 
    \[
        a_{p,i}\left(\prod_q I'_q\right) = \sum_q \mathrm{length}(\mathcal{O}_S / (\mathcal{O}_S(-b_q\mathbb{P}_q)+ I'_q)),
    \]
    where $b_q$ is the coefficient of $\mathbb{P}_q$ in $D_{p,i}$.
    This length is clearly $0$ for $q\npreceq p$, and equals to the length of $\mathcal{O}_S/I'_q$ for $q\prec p$.
    Therefore, it suffices to show that $a_{p,i}$ is constant on $W^-_{F_{\lambda(p)}}$.

    Let us suppress $p$ from notations for simplicity.
    It is easy to see from e.g. \cref{prop:Hilb-fix-C2} that $a_i(I) = \sum_{j=1}^i \lambda_i$ for any $I\in F_{\lambda(p)}$.
    As $a_{i}$ is $\T$-invariant and upper semicontinuous, it suffices to show that it attains the same value at a general point of $W^-_{F_{\lambda(p)}}$.
    For this computation, we can assume that $S = \mathbb{C}^2$.
    Observe that the subvariety $W^-_{F_{\lambda(p)}}\subset \Hilb(\mathbb{C}^2)$ is independent of the weights $-l, k+l$ of the $\T$-action, as long as $k>0$, $l\geq 0$.
    Indeed, consider the formula for tangent weights in \cref{prop:fixed-tg-weights}.
    It describes the decomposition of $T_{F_{\lambda(p)}}$ into weight spaces.
    As we change the parameters $k>0$, $l\geq 0$, the direct sum of positive and non-positive weight spaces does not change, as evidenced by~\eqref{eq:Tplus-tangent}.
    Hence $W^-_{F_{\lambda(p)}}$ is independent of $k$, $l$ around $F_{\lambda(p)}$, and so everywhere by $\T$-equivariance.

    We can therefore assume $k=1$, $l=0$.
    In this case, a general fixed point is of the form $I = I^{(1^{\lambda'_1})}_{p_1}\cap \ldots \cap I^{(1^{\lambda'_q})}_{p_q}$, $p_j\in \{y=0\}$.
    By additivity of length, it is enough to consider points $I^{(1^{k})}_p$.
    The downward flow of such a point consists of the ideals
    \[
        \{ I_{\underline{b}}\coloneqq(x + b_1y + \ldots + b_{k-1}y^{k-1},y^k) : b_i\in \mathbb{C} \},
    \]
    see e.g. the proof of~\cite[Prop.~4.14]{nakajima1996jack}. 
    For these ideals, we have $a_i(I_{\underline{b}}) = i$ for $0\leq i\leq k$, and $a_i(I_{\underline{b}}) = k$ for $i>k$.
    Summing up over all points $p_j$, we get precisely $a_i(I) = \sum_{j=1}^i \lambda_i$. 
    This coincides with the value of $a_i$ on the fixed points, so we may conclude.
\end{proof}

\begin{proof}[Proof of \cref{prop:Hilb-flow-is-dom}]
    We first prove the ``only if'' direction.
    By \cref{prop:flow-order-descr}, $F_{\underline{\lambda}}\preceq F_{\underline{\mu}}$ if and only if there is a chain of $\T$-curves from $F_{\underline{\lambda}}$ to $F_{\underline{\mu}}$.
    We have $F_{\underline{\mu}}\subset X_{\underline{\mu}}$ by \cref{lem:geometric-dom}.
    Successive applications of \cref{lem:dom-loc-cst} and the fact that $X_{\underline{\mu}}$ is closed imply that $F_{\underline{\lambda}}\subset X_{\underline{\mu}}$.
    Thus $\underline{\lambda}$ is dominated by $\underline{\mu}$ by \cref{lem:geometric-dom}.

    For the converse implication, applying \cref{lm:dom-skelethon} it suffices to construct an explicit $\T$-curve for each pair of multipartitions as in moves $(i')$ and $(ii)$.
    For the move $(i')$, restricting to an étale neighborhood of a point in $F$, we reduce to the case of $S = \mathbb{C}^2$ and two monomial ideals
    \begin{align*}
        I &= (x^{\lambda_1}, \ldots, x^{\lambda_{i-1}}y^{i-2},x^{\lambda_{i}}y^{i-1},x^{\lambda_{i}-1}y^{i},\ldots,x^{\lambda_{i}-1}y^{j-2},x^{\lambda_{j}}y^{j-1}, \ldots,y^q),\\
        J &= (x^{\lambda_1}, \ldots, x^{\lambda_{i-1}}y^{i-2},x^{\lambda_{i}-1}y^{i-1},x^{\lambda_{i}-1}y^{i},\ldots,x^{\lambda_{i}-1}y^{j-2},x^{\lambda_{j}+1}y^{j-1}, \ldots,y^q),
    \end{align*}
    where $\lambda_{i}-1\geq \lambda_{j}$.
    Note that $I/I\cap (x^{\lambda_j + 1}y^{i-1}) = J/J\cap (x^{\lambda_j + 1}y^{i-1})$.
    Replacing $\mathbb{C}[x,y]$ with $x^{\lambda_j + 1}y^{i-1}$, it suffices to find a $\T$-curve from $I' = (x^{a+1},x^a y, y^b)$ to $J' = (x^a, xy^b, y^{b+1})$ for arbitrary $a,b>0$.
    It is given by the family of ideals, with $t\in \mathbb{P}^1$, 
    \[
        (x^{a+1}, y^b+t x^a, y^{b+1}) = (x^{a+1}, t^{-1}y^b+ x^a, y^{b+1}).
    \]   
    The torus $\T$ acts on this family with a positive weight $(k+l)b+la$, therefore the $\T$-curve goes in the correct direction. 
    
    For the move $(ii)$, let $p,q\in S^\T$ be two points connected by a $\T$-curve.
    Restricting to an étale neighborhood of said curve, we can assume that we work in $\Hilb^n(T^*\mathbb{P}^1)$, and $p = 0$, $q = \infty$.
    Denote the local coordinates at $0$ by $(x,y)$, and at $\infty$ by $(x',y')$; on $T^*\mathbb{C}^*$, they are related by equations $x' = y^{-1}$, $y' = xy^2$.
    Let $I_\lambda = (x^{\lambda_1},\ldots,x^{\lambda_{q}}y^{q-1},y^q)$ be a monomial ideal at $0$, and $J_\mu = ((x')^{\mu_1},\ldots,(x')^{\mu_{q'}}(y')^{q'-1},(y')^{q'})$ a monomial ideal at $\infty$.
    Write $\lambda^+ = (\lambda_1,\ldots,\lambda_q,1)$, and $\mu = (\mu_1+1,\mu_2,\ldots,\mu_{q'})$.
    We want to construct a $\T$-curve from $I_{\lambda^+}\cap J_\mu$ to $I_\lambda\cap J_{\mu^+}$.
    Consider the following two families of ideals, with $t\in \mathbb{C}$:
    \begin{align*}
        X_t & = (x^{\lambda_1},\ldots,x^{\lambda_{q}}y^{q-1},(y-t)y^{q},y^{q+1})\cap J_\mu,\\
        X'_t & = I_\lambda \cap ((x')^{\mu_1}(x'-t),\ldots,(x')^{\mu_{q'}}(y')^{q'-1},(y')^{q'})
    \end{align*}
    It is clear that $X_0 = I_{\lambda^+}\cap J_\mu$ and $X'_0=I_\lambda\cap J_{\mu^+}$.
    Moreover, denoting by $I_t$ the colength $1$ ideal supported at $[t:1]\subset \mathbb{P}^1$, we see that $X_t = I_{\lambda}\cap I_t\cap J_\mu = X'_{t^{-1}}$ for $t\in \mathbb{C}^*$.
    Therefore these two families glue into a single $\T$-curve.
    The action of $\T$ on this family is of positive weight $k+l_p$, hence the $\T$-curve goes in the correct direction.
\end{proof}

\begin{rmk}
    We stress that one can, and does, have $\T$-curves other than the ones in the proof of \cref{prop:Hilb-flow-is-dom}.
    For instance, for any $c\in C$ and $n>0$ we have a $\T$-curve from $I_c^{(n)}$ to $I_c^{(1^n)}$ in $\Hilb^n(S)$, given by the family of ideals $\{x^n,y-\lambda x\}$, $\lambda \in \mathbb{P}^1$.
\end{rmk}

\subsection{Cotangent case}
Let $S = T^*E\simeq E\times \mathbb{C}$, $\Hitch = \Hilb^n S$, and $\Base = \Sym^n\C$.
In this case, $S^\T = E$, $\Fix = \Part(n)$, the component $F_\lambda$ of the fixed locus is isomorphic to $\Sym^\lambda E$, and the dominance order on Fix is the usual dominance order on partitions.

\begin{prop}\label{prop:ell-very-stable}
    For $S = T^*E$, every component $F_\lambda\in \Fix$ has a very stable point.
\end{prop}
\begin{proof}
    First, let $\lambda = (1^n)$, and fix a point $p\in E$.
    The closed embedding $i_p:\C \simeq T^*_p \subset T^*E$ induces the following equivariant section of $\theta$ by adjunction:
    \[
        i_*:\Base\simeq \Hilb^n \C \to \Hilb^n S, \quad I\mapsto \ker (\mathcal{O}_S \to (i_p)_* (\mathcal{O}_\C/I)) 
    \]
    The image of the origin $0\in\Base$ under $i_*$ is precisely $I_p^{(1^n)}$.
    By equivariance and dimension count we have $i_*(\Base) = W^+_{I_p^{(1^n)}}$, which proves that $I_p^{(1^n)}$ is very stable.
    Thus, all points of $F_{(1^n)}$ are very stable.

    Let now $\lambda\in\Part(n)$ be an arbitrary partition, and pick $m=\lambda_1$ distinct points $x_1,\ldots,x_m$ in $E$.
    Consider the ideal
    \[
        I\coloneqq I_{x_1}^{(1^{\lambda'_1})}\cap\ldots\cap I_{x_m}^{(1^{\lambda'_m})} \in F_\lambda.
    \]
    We claim that this is a very stable point.
    Indeed, $I\in \Hilb^n S$ has an étale neighborhood isomorphic to an étale neighborhood of $(I_{x_1}^{(1^{\lambda'_1})},\ldots, I_{x_m}^{(1^{\lambda'_m})})$ in $\prod_i \Hilb^{\lambda'_i} S$ by \cref{prop:Hilb-etale}. 
    The latter point is very stable by the computation above. 
\end{proof}

Thanks to \cref{prop:HH-mult}, we immediately obtain the equivariant multiplicities.
\begin{cor}\label{cor:HilbTE-mult}
    For any $\lambda\in\Fix$, we have $m_\lambda(t) = \qnom{n}{\lambda'}_t$, in particular $m_\lambda=\binom{n}{\lambda'}$.
\end{cor}
\begin{proof}
    Applying the formula~\eqref{eq:Tplus-tangent}, we have 
    \begin{align*}
        m_\lambda(t) &= \frac{\chi_t(\Sym T^+_p)}{\chi_t(\Sym \Base)}
        = \frac{\prod_{i=1}^n (1-t^i)}{\prod_{s\in \lambda}(1-t^{a(s)+1})}
        = \frac{\prod_{i=1}^n [i]_t}{\prod_{s\in \lambda}[a(s)+1]_t}\\
        &= \frac{[n]!_t}{\prod_{i}[\lambda'_i]!_t}
        = \qnom{n}{\lambda'}_t.
    \end{align*}
    The claim for $m_\lambda$ follows by specializing to $t=1$.
\end{proof}

In this case, we can prove a bit more.
Recall~\cite[Def.~3.7]{hausel2022enhanced} that for a very stable point $I\in \Hitch$, its \textit{equivariant multiplicity algebra} $Q_I$ is defined as the ring of functions on the scheme-theoretic preimage of $0$ under $\theta|_{W^+_I}:W^+_I\to \Base$.
This algebra is naturally graded by the action of $\T$; for convenience, we rescale it by doubling the degree of each element.
\begin{cor}\label{cor:mult-alg-ell}
    Let $\lambda$ be a partition of $n$, $x_1,\ldots,x_m$ distinct points in $E$, and $I = I_{x_1}^{(1^{\lambda_1})}\cap\ldots\cap I_{x_m}^{(1^{\lambda_m})}$.
    Then $Q_I\simeq H^*(\mathscr{F}_\lambda)$ as graded algebras, where $\mathscr{F}_\lambda\coloneqq GL_n/P_\lambda$ is the partial flag variety.
\end{cor}
\begin{proof}
    Using \cref{prop:Hilb-etale} as in the proof of \cref{prop:ell-very-stable}, we see that the map $\theta|_{W^+_I}$ is étale-locally isomorphic to the symmetrization map $\prod_{i} \Sym^{\lambda_i} \C \to \Sym^n \C$.
    On the rings of functions, this map is given by the inclusion
    \[
        \C[x_1,\ldots,x_n]^{\mathfrak{S}_n} \hookrightarrow \C[x_1,\ldots,x_n]^{\mathfrak{S}_\lambda},
    \]
    where each $x_i$ has degree $2$.
    Let $S^+\subset \C[x_1,\ldots,x_n]^{\mathfrak{S}_n}$ be the augmentation ideal.
    Denoting the coinvariant ring $\C[x_1,\ldots,x_n]/(S^+)$ by $R_n$, we obtain $Q_I \simeq R_n^{\mathfrak{S}_\lambda}$.
    On the other hand, it is well known that $H^*(\mathscr{F}_\lambda) \simeq H^*(\mathscr{F}_{(1^n)})^{\mathfrak{S}_\lambda}\simeq R_n^{\mathfrak{S}_\lambda}$.
    We may conclude.
\end{proof}

\subsection{Parabolic case}\label{subs:Hilbert of parabolic S}
Let us now consider the surfaces $S_{\mathbb{Z}/e}$ of the parabolic family.
In this case, the analogue of \cref{prop:ell-very-stable} fails to hold already for $n=2$; indeed, if $p\in \Fix(S)$ is a point with $m_p(t)=1$, the isolated fixed point $I_p^{(2)}$ admits a $\T$-curve to $I_p^{(1^2)}$, and hence is wobbly.
The following proposition can be interpreted as saying that, in the terminology of~\cite{hausel2022very}, being very stable in this case is the same as being generically of type $(1,1,\ldots,1)$.

\begin{prop}\label{prop:par-stable-descr}
    Let $\underline{\lambda}\in \Part^{\Fix(S)}(n)$.
    The fixed component $F_\ulambda \in \Fix$ is very stable 
    if and only if $\lambda(p_i)=\varnothing$ for every wobbly $p_i\in S^\T$, and $\lambda(p_i)=(1^{k_i})$ for every very stable $p_i\in S^\T$ and some $k_i\geq 0$.
\end{prop}
\begin{proof}
    Let $I\in F_{\ulambda} = I'\cap I_{p_1}^{\lambda(p_1)}\cap\ldots\cap I_{p_k}^{\lambda(p_k)}$, where $I'$ is an ideal supported on $C$.
    Recall that we can check whether a point is very stable étale-locally, using the equivalent definition (2) of \cref{prop:v-stable-lf-push}.
    \cref{prop:Hilb-etale} then implies that $I$ is very stable if and only if each of the ideals $I'$, $I_{p_j}^{\lambda(p_j)}$ is very stable in the corresponding smaller Hilbert scheme of points.
    Since any two smooth curves are étale-locally isomorphic, $I'$ can be chosen to be very stable by \cref{prop:ell-very-stable}.
    Each of the ideals $I_{p_i}^{\lambda(p_i)}$ is isolated, and by definition an isolated fixed point is very stable if and only if it is maximal in the flow order.
    Thanks to \cref{prop:Hilb-flow-is-dom}, we know that $I_{p_i}^{\lambda(p_i)}$ is maximal if and only if $p_i$ is maximal (hence very stable) in $S$, and $\lambda(p_i) = (1^{|\lambda(p_i)|})$.
    This concludes the proof.
\end{proof}

Let us compute the equivariant multiplicities of very stable points.
Fix a very stable multipartition $\underline{\lambda}$; by the previous proposition we have $\lambda(p) = \varnothing$ for wobbly $p\in \Fix(S)$, and $\lambda(p) = (1^{\lambda(p)'})$ for very stable $p\in \Fix(S)\setminus \{C\}$.
We denote $w_p \coloneqq k+l_{p}$,
By \cref{prop:v-stable-HHvsMZ} we have
\begin{align*}
    m_{\underline{\lambda}}(t) 
    & = \mu_{\underline{\lambda}}(t)
    = \frac{(1-t^e)\ldots (1-t^{ne})}{\chi_t(\Sym T^+_{I})^{-1}} 
    = \frac{\prod_{i=1}^{n}[ie]_t}{\left(\prod_{i=1}^{\lambda(C)_1} [\lambda(C)'_i]!_t\right)\left( \prod_{i=1}^r \prod_{j=1}^{\lambda(p_i)'}[jw_{p_i}]_t\right)} \\
    & = \frac{[n]!_t}{\left(\prod_{i=1}^{\lambda(C)_1} [\lambda(C)'_i]!_t\right)\left(\prod_{i=1}^r[\lambda(p_i)']!_t\right)} \frac{\prod_{i=1}^{n}[e]_{t^i}}{\prod_{i=1}^r \prod_{j=1}^{\lambda(p_i)'}[w_{p_i}]_{t^j}},
\end{align*}
where in the first line $I\in F_{\underline{\lambda}}$ very stable.
For any polynomial $p(t)$, let us write 
\begin{equation}\label{eq:t-powers}
    p(t)^{(b)} \coloneqq \prod_{i=1}^{b}p(t^i),   
\end{equation}
and treat $\underline{\lambda}'$ as an unordered tuple of numbers, obtained by concatenating all $\lambda(p)'$'s.
We obtain
\begin{prop}\label{prop:vst-eq-mult-parab}
Equivariant multiplicity of a very stable component $F_\ulambda \in \Fix(\Hilb^n(S_{\Z/e}))$ is
\begin{equation}\label{eq:vst-eq-mult-parab}
    \pushQED{\qed} 
    m_{\underline{\lambda}}(t) = \qnom{n}{\underline{\lambda}'}_t\cdot\frac{[e]_t^{(n)}}{\prod_{p}[w_p]_t^{(\lambda(p)')}}.\qedhere
    \popQED
\end{equation}
\end{prop}
\begin{rmk}
    One can rewrite formula~\eqref{eq:vst-eq-mult-parab} to make its polynomiality apparent:
    \begin{equation}\label{eq:vst-eq-mult-parab-poly}
        m_{\underline{\lambda}}(t) = \qnom{n}{|\lambda(C)|,|\lambda(p_1)|,\ldots, |\lambda(p_r)|}_{t^e} \prod_{p\in \Fix(S)} \qnom{|\lambda(p)|}{\lambda(p)'}_{t^{w_p}} m_p(t)^{(|\lambda(p)|)}.
    \end{equation}
    Note that with our choice of $\underline{\lambda}$, we have $\qnom{|\lambda(p)|}{\lambda(p)'}_{t^{w_p}} = 1$ for all $p\neq C$.
\end{rmk}

For wobbly components, the non-equivariant multiplicities can also be computed.
\begin{thm}\label{prop:mult-parab}
    For any $F_\ulambda\in \Fix$, we have
    \begin{equation}\label{eq:neq-mult-parab}
        m_{\underline{\lambda}} = \binom{n}{\underline{\lambda}'}\prod_{p\in\Fix(S)} m_p^{|\lambda(p)|}.
    \end{equation}
\end{thm}
\begin{proof}
    Using étale neighbourhoods of \cref{prop:Hilb-etale} together with \cref{prop:cover-mult}, we see that
    \[
        m_{\underline{\lambda}} = \binom{n}{\{ |\lambda(p)| \}_{p\in\Fix(S)}}\prod_{p\in\Fix(S)} m_{\lambda(p)},
    \]
    where each $m_{\lambda(p)}$ is computed in a smaller Hilbert scheme.
    Thus, it suffices to prove the formula for an ideal $I$ completely supported on one component of $\Fix(S)$.
    For the bottom component $C\in \Fix(S)$, the result follows by specializing equation~\eqref{eq:vst-eq-mult-parab} to $t=1$.
    Therefore from now on, we can assume that we are in the local situation of \cref{prop:Hilb-fix-C2}, suppress $p$ from notations, and write $\lambda \in \Part(n)$.

    Observe that the multiplicities $m_\lambda$ are independent of the choice of $\T$-action, as long as the downward flow does not change.
    Hence, we can replace the action $t(x,y) = (t^{-l}x, t^{k+l}y)$ with $t(x,y) = (x, ty)$; in particular, our fixed component ceases to be punctual.
    The integrable system structure is induced by a map $\pi:\C^2\to \C$, $(x,y)\mapsto (x^ay^m)$, where $a < m = m_p$.
    As $m_\lambda$ does not depend on a choice of a point in the fixed component, we can move the support of $I$ from $(0,0)$ to $(c,0)$, $c\neq 0$.
    Passing to an étale neighbourhood of this point, the map $\pi$ becomes simply $(x,y)\mapsto (y^m)$.
    Applying \cref{prop:Hilb-etale} again, we are reduced to computing the multiplicity of the core component $\Core_{F_\lambda}$ of $\Hilb^n T^*E$, where the map $T^*E\to \C$ is the composition of the natural projection with $f:z\mapsto z^m$.
    Using \cref{cor:HilbTE-mult,prop:cover-mult}, we obtain
    \[
        m_\lambda = m_{\Core_{F_\lambda}}(1)\deg(\Sym^n f) = \binom{n}{\lambda'} m^n,
    \]
    which is exactly what remained to be proved.
\end{proof}

\begin{cor}
    The first statement of \cref{conj:mult-inequality} holds for $\Hilb^n(S_{\mathbb{Z}/e})$.
\end{cor}
\begin{proof}
    Using \cref{prop:product-mult,prop:cover-mult} for $\mu(t)$ instead of $m$ and repeating the proof above, we reduce the claim to the case $n=1$, which can be easily checked using the explicit formulas in \cref{ssec:2d-mult-ex}.
\end{proof}

\begin{rmk}
    It is very tempting to guess that the equivariant multiplicities of wobbly components are also given by~\eqref{eq:vst-eq-mult-parab-poly}.
    Our computations in \cref{sec:ex-Hilb2} show that this is not true, see \cref{table:eqmultH2}.
    We wonder whether the ``naive'' polynomials in~\eqref{eq:vst-eq-mult-parab-poly} can be imbued with some other geometric meaning.
    This appears to be the case for more sophisticated polynomials; see \cref{ssec:main-conj}.

    The proof of \cref{prop:mult-parab} relies on two non-equivariant tricks: applying \cref{prop:cover-mult} and changing $\T$-action.
    If the latter was replaced by some other argument, and \cref{conj:finite-cover} held true, then the same proof would at least verify \cref{conj:eq-mult-positivity} for all Hilbert schemes.
\end{rmk}

\subsection{Painlevé case}\label{subs:Hilbert of Painleve}
Let $S=S^\mathrm{X}$, $\mathrm{X}\in \{\mathrm{I},\mathrm{II},\mathrm{IV}\}$ be a surface in the Painlevé family, and denote the bottom fixed component of $S$ by $b$. 
Recall that $\Hilb^n S$ has isolated fixed points, and pick a $\T$-fixed ideal $I$ completely supported at $b$.
Since $b$ has strictly positive tangent weights, the closure $\overline{W^-_I}$ is contained in the punctual Hilbert scheme at $b$.
By Briançon's theorem~\cite[Cor.~V.3.3]{brianccon1977description}, the latter is irreducible of dimension $n-1$; in particular, none of the isotropic subvarieties $W^-_I\subset \Hilb^n S$ can be Lagrangian.
Passing to an étale neighbourhood as in \cref{prop:Hilb-etale}, we conclude that a point $F_{\underline{\lambda}}\in \Fix$ satisfies~\eqref{eq:weight-gap} if and only if $\lambda(b)=\varnothing$.
For such points, the proof of \cref{prop:mult-parab} applies.
Moreover, since all core components in $S^\mathrm{X}$ have multiplicity $1$, we get $m_{\underline{\lambda}} = \binom{n}{\underline{\lambda}'}$.

Similarly, by the argument in \cref{prop:par-stable-descr} the only very stable points in $\Hilb^n S$ are the ideals $I_{\underline{\lambda}}\coloneqq\bigcap_{p\in \Fix(S)\setminus \{b\}} I_p^{1^{\ell(\lambda(p))}}$, where $\sum_p \ell(\lambda(p)) = \sum_p \lambda(p)'_1 = n$.
Note that for each point $p\in\Fix(S)$, $p\neq b$ we have $w_p=e$. 
Analogously to~\eqref{eq:vst-eq-mult-parab}, we obtain
\begin{align*}
    m_{\underline{\lambda}}(t) 
    = \mu_{\underline{\lambda}}(t)
    = \frac{\prod_{i=1}^n(1-t^{ie})}{\prod_p \prod_{i=1}^{\ell(\lambda(p))}(1-t^{ie})}
    = \frac{\prod_{i=1}^n[i]_{t^e}}{\prod_p \prod_{i=1}^{\ell(\lambda(p))}[i]_{t^e}} 
    = \qnom{n}{\underline{\lambda}'}_{t^e}.
\end{align*}

Summarizing the discussion above, we obtain
\begin{prop}\label{prop:mult-pain}
    $F_\ulambda\in \Fix$ satisfies~\eqref{eq:weight-gap} if and only if $\lambda(b)=\varnothing$.
    $F_\ulambda\in \Fix^{\WG}$ is very stable if and only if $\lambda(p_i)=(1^{k_i})$ for every $p_i\in S^\T\setminus\{b\}$ and some $k_i\geq 0$. 
    Furthermore, we have
    \begin{equation}
         m_{\underline{\lambda}} = \binom{n}{\underline{\lambda}'},\quad F_\ulambda\in \Fix^{\WG};\qquad 
         m_{\underline{\lambda}}(t)=\qnom{n}{\underline{\lambda}'}_{t^e},\quad \textrm{$F_\ulambda$ very stable}.
    \end{equation}
\end{prop}

The proofs of \cref{prop:ell-very-stable,prop:par-stable-descr,prop:mult-pain} also imply the following.
\begin{cor}\label{cor:v-st-points}
    A $\T$-fixed ideal $I_{x_1}^{\lambda_1}\cap \ldots \cap I_{x_s}^{\lambda_s}$ is very stable if and only if $\lambda_i = (1^{|\lambda_i|})$ and $x_i$ is very stable in $S$ for all $i$. \qed
\end{cor}

\begin{rmk}
    A similar analysis can be performed for the higher rank version of Hilbert schemes, namely moduli of torsion-free sheaves on $S$ equipped with a trivialization at infinity.
    We will treat it in a subsequent work.
\end{rmk}

\section{Cohomological filtrations}\label{sec:coh-filtr}
This section collects some observations about the multiplicities of core components of Hilbert schemes.
\subsection{Multiplicity filtration}
Let $\theta: \Hitch\to \Base$ be a pre-integrable system. 
A natural question to ask is whether multiplicity of core components defines a tidal filtration on $H^*(\Hitch)$ (see \cref{def:tidal-filt}), provided that the condition~\eqref{eq:weight-gap} holds.

\begin{prop}\label{prop:tidal-ell}
    Multiplicity is tidal on $\Hilb^n(T^*E)$.
\end{prop}
\begin{proof}
Let $\lambda,\mu\in \Part(n)$.
By \cref{prop:Hilb-flow-is-dom}, we have $F_{{\lambda}}\preceq F_{{\mu}}$ if and only if $\lambda$ dominates $\mu$, or equivalently $\mu'$ dominates $\lambda'$, which means that $\sum_{i=1}^j \mu'_i\geq\sum_{i=1}^j \lambda'_i$ for all $j$.

Consider a strictly convex function $\phi: [0,\infty)\to \mathbb{R}$, $\phi(x)=\log \Gamma(x+1)$.
By Karamata's inequality,
\[
\sum_i \phi(\mu'_i) \geq \sum_i \phi(\lambda'_i),
\]
and the inequality is strict for $\lambda \neq \mu$.
Exponentiating both sides of the inequality, we obtain
\[
m_\mu = \binom{n}{\mu'} =\frac{n!}{\prod_i \mu'_i!}\leq
\frac{n!}{\prod_i \lambda'_i!}=\binom{n}{\lambda'}=m_\lambda,
\]
where we used \cref{cor:HilbTE-mult}, and the fact that $\exp(\phi(n)) = n!$ for $n\in \mathbb{Z}_{\geq 0}$.
\end{proof}

\begin{rmk}
    Let $\lambda = (5,1,1,1)$ and $\mu = (4,4)$.
    Then $m_\lambda = 336 > 70 = m_\mu$, while for the Atiyah--Bott tidal function of \cref{ex:MO-tidal} we have $w(\lambda) = 10 < 12 = w(\mu)$.
    In particular, there can be no $\T$-curves between $F_{\lambda}$ and $F_{\mu}$.
\end{rmk}

Unfortunately, \cref{prop:tidal-ell} does not generalize much beyond $S = T^*E$.

\begin{prop}\label{prop:list-tidal}
    Let $S$ be a surface of parabolic family.
    Multiplicity is tidal for $\Hilb^n S$ exactly in the following list of cases:
    \begin{itemize}
        \item $n=1$ or $S = T^*E$;
        \item $S = S_{\Z / 2}$, and $n=2$.
    \end{itemize}
\end{prop}
\begin{proof}
    In the first case, multiplicity is tidal by \cref{ssec:2d-mult-ex,prop:tidal-ell}.
    In the second case, this can be checked by a direct computation, using \cref{prop:Hilb-flow-is-dom} for a description of flow order and \cref{prop:mult-parab} for multiplicities.

    For $\Hilb^n(S_{\Z / e})$, $e>2$, $n\geq 2$, consider the multipartitions $\underline{\lambda}=((1^n),\varnothing)$ and $\underline{\mu}=((1^{n-1}),(1))$, where the first partition lies on the bottom core component $C\subset S$, and the second is supported on $p\subset S^\T$ with $m_p = e-1$.
    Clearly $\underline{\lambda}$ dominates $\underline{\mu}$; at the same time 
    \[
        m_{\underline{\lambda}} = e^n < n e^{n-1}(e-1) = m_{\underline{\mu}},
    \]
    which contradicts tidality.
    The same pair also rules out $\Hilb^n(S_{\Z/2})$, $n\geq 3$.
\end{proof}

\begin{rmk}
    Curiously, this classification superficially resembles Tirelli's classification of moduli of Higgs bundles of degree $0$ admitting a symplectic resolution~\cite{tirelli2019symplectic}.
\end{rmk}

One might ask whether the failure of tidality only occurs at wobbly components. This is not the case either.
Let us call a function $w: \Fix\to \mathbb{R}$ \textit{stably tidal} if for any very stable $F,F'\in \Fix$ with $F\preceq F'$ we have $w(F)\geq w(F')$.

\begin{prop}
    Let $S$ be a surface of parabolic family.
    Multiplicity is stably tidal for $\Hilb^n S$ exactly in the following list of cases:
    \begin{itemize}
        \item $n\leq 2$ or $S = T^*E$;
        \item $S = S_{\Z / 3}$, and $n=3$.
    \end{itemize}
\end{prop}
\begin{proof}
    The cases $n=1$ or $S = T^*E$ follow from \cref{prop:list-tidal}.
    In the rest of the cases, stable-tidality can be checked via an explicit computation.

    For $\Hilb^n(S_{\Z / e})$, $e\neq 3$, $n\geq 3$, consider the multipartitions $\underline{\lambda}=((1^n),\varnothing)$ and $\underline{\mu}=((1^{n-1}),(1))$, where the first partition lies on the bottom core component $C\subset S$, and the second is supported on a very stable point $p\subset S^\T$ with $m_p = e/2$.
    Clearly $\underline{\lambda}$ dominates $\underline{\mu}$; at the same time
    \[
        m_{\underline{\lambda}} = e^n < n e^{n-1}\cdot e/2 = m_{\underline{\mu}},
    \]
    which contradicts stable tidality.

    A similar pair also rules out $\Hilb^n(S_{\Z / 3})$ for $n\geq 4$.
    We do not change the multipartitions, and take $p$ with $m_p = 1$.
    Then $m_{\underline{\lambda}} = 3^n < n\cdot 3^{n-1} = m_{\underline{\mu}}$, hence a contradiction.
\end{proof}

\begin{rmk}
    Let $\Hitch$ be the moduli space of stable Higgs bundles of rank $2$.
    In this case, all core components are very stable, and their multiplicities are explicit powers of $2$ by~\cite{hausel2022very}. 
    Moreover, all these powers of $2$ are distinct, and the flow order is linear; from this it is easy to conclude that multiplicity is tidal.
    S. Raskin has informed us that he can prove tidality for the moduli of stable Higgs bundles of arbitrary rank.
    It would be very interesting to understand what accounts for such divergence of behaviour between the unramified and tamely ramified settings.
\end{rmk}

Let us briefly comment on Hilbert schemes on Painlevé surfaces.
While multiplicity is only defined on a subset $\Fix^\WG\subset \Fix$, we can still ask whether it is monotonous with respect to the restriction of flow order to $\Fix^\WG$.

\begin{prop}
    Multiplicity is tidal on $\Fix^\WG(\Hilb^n(S^{\mathrm{X}}))$, $\mathrm{X}\in \{\mathrm{I},\mathrm{II},\mathrm{IV}\}$.
\end{prop}
\begin{proof}
    Recall from \cref{prop:mult-pain} that a multipartition belongs to $\Fix^\WG$ if and only if it is supported on very stable points of $S = S^{\mathrm{X}}$.
    Since these points in $S$ do not have $\T$-curves between them, two such multipartitions are comparable only if at each point they restrict to a partition of the same number.
    Since
    \[
        m_{\underline{\lambda}} = \binom{n}{\underline{\lambda}'} = \binom{n}{\{|\lambda(p)|\}_p}\prod_p \binom{|\lambda(p)|}{\lambda(p)'},
    \]
    we are reduced to checking tidality only for the multipartitions supported at a single point.
    This is achieved by the same argument as in \cref{prop:tidal-ell}.
\end{proof}

\subsection{Dimension filtration}
From now on and until the end of this section, we assume that $k=1$, so that $\Fix = \Fix^\WG$.
The discussion above shows that multiplicity of fixed components is not well suited for defining cohomological filtration: it is almost never tidal, and even when it is, the possible values of multinomial coefficients are too sparse.
Let us consider a coarser filtration.

\begin{de}
    Define the dimension function $d: \Fix \to \Z$ by $d(F) = \dim F$.
\end{de}

The dimension function can be recovered from virtual multiplicities $\mu(t)$; we expect that a similar statement also holds for $m(t)$.
\begin{lm}
    Let $1\leq A\leq n$ be the maximal index with $e_i = 1$. Then $\mu_F(t) = 1 + (d(F) - A)t + O(t^2)$ for any $F\in\Fix$.
\end{lm}
\begin{proof}
    Let $0\leq a_1 \leq \ldots a_n$ be the positive tangent weights of $F$.
    Since the symplectic form $w$ has weight $1$, it pairs weight-$0$ directions with weight-$1$ directions.
    In particular, $a_i = 1$ if and only if $i\leq d(F)$.
    By \cref{prop:sing-is-mu}, we have
    \[
        \mu_F(t) = \frac{(1-t^{e_1})\ldots(1-t^{e_n})}{(1-t^{a_1})\ldots(1-t^{a_n})} = (1-t)^{A-d(F)}\frac{(1-t^{e_{A+1}})\ldots(1-t^{e_n})}{(1-t^{a_{d(F)+1}})\ldots(1-t^{a_n})} = 1 + (d(F) - A)t + O(t^2),
    \]
    and so we may conclude.
\end{proof}

\begin{prop}
    Let $\Hitch = \Hilb^n S$, where $S$ is a surface of parabolic family.
    Then the dimension function $d$ is tidal.
\end{prop}
\begin{proof}
    Recall the description of fixed loci in \cref{cor:Hilb-fixed-loci}.
    It immediately follows that $d(F)_{\underline{\lambda}} = \lambda(C)_1$.
    As the flow order on $\Fix$ coincides with the antidominance order by \cref{prop:Hilb-flow-is-dom}, we have $\lambda(C)_1 \geq \nu(C)_1$ for any $\lambda\preceq\nu$.
\end{proof}

Recall that the ranks of the \textit{perverse} filtration $P_\bullet$ on $H^*(\Hilb^n S)$ were computed in~\cite{shen-zhang}.
Let us denote $P_n(q) \coloneqq \sum_{k=0}^{\infty} q^{k} \dim \gr^P_k H^{\mathrm{top}}(\Hilb^n S)$.
As a consequence of~\cite[Th.~3.6]{shen-zhang}, we have 
\begin{equation}\label{eq:gen-fun-P}
    \sum_{n=0}^{\infty} s^n P_n(q) = \prod_{m=1}^{\infty} \frac{1}{(1-s^mq^m)^r(1-s^mq^{m+1})},
\end{equation}
where $r$ is defined as in~\eqref{eq:fix-S}.
Denoting the tidal filtration associated to $d$ by $D_\bullet$, let us similarly write $D_n(q) \coloneqq \sum_k q^k \dim \gr^D_k H^{\mathrm{top}}(\Hilb^n S)$.

\begin{cor}
    We have $q^n D_n(q) = P_n(q)$.
\end{cor}
\begin{proof}
    The generating function of $q^nD_n(q)$ can be easily computed:
    \begin{align*}
        \sum_{n=0}^{\infty} (sq)^n D_n(q) 
        & = \sum_{\underline{\lambda}\in\Part^\Fix(S)} s^{|\underline{\lambda}|} q^{|\underline{\lambda}|+\lambda(C)_1} 
        = \prod_{p\in \Fix(S)} \sum_{\lambda_in \Part} s^{|\lambda|} q^{|\lambda| + \dim(p)\lambda_1}\\
        & = \prod_{m=1}^\infty \frac{1}{1 - s^mq^{m+1}} \prod_{i=1}^r \prod_{m=1}^\infty \frac{1}{1 - s^mq^{m}} 
        = \sum_{n=0}^{\infty} s^n P_n(q).
    \end{align*}
    We may conclude.
\end{proof}

While this observation is very curious, let us warn the reader against simplistic generalizations.
First, as illustrated by \cref{ex:perv-not-tidal} this equality is purely numerical, and does not arise from an equality of filtrations.
Second, it is very particular to Hilbert schemes of points, as the following example shows.

\begin{ex}
    Let $C$ be a smooth projective curve of genus $g>1$, and let $\Hitch_{r}$ be the moduli of stable Higgs bundles on $C$ of rank $r$ and degree $1$.
    For $r=2$, there are $g$ fixed components.
    It is easy to check that in this case $d$ is tidal, and $D(q) = q^{4g-3} + \sum_{i=1}^{g-1}q^{g+2i-1}$.
    On the other hand \cite[Th.~1.1.6]{hausel2008mixed} yields, up to a global power of $q$, that $P(q) = \sum_{i=1}^{g}q^{2g+2i-3}$.

    While the discrepancy above is mild, consider the case $r=3$, $g=2$.
    Here, $\frac{1}{2}\dim \Hitch_{r} = n = 10$.
    As explained in~\cite[Sec.~8]{garcia2014motives}, $\Hitch^\T$ has $6$ connected components, and the dimension count~\cite[Th.~3.8]{alvarez2006geometry} yields $D(q) = q^4 + 2q^5 + 2q^6 + q^{10}$.
    However, by purity conjecture~\cite[Rmk~4.4.2]{hausel2008mixed} and curious Poincaré duality (which follows from $P=W$) the polynomial $P(q)$ coincides with the Kac polynomial of $2$-loop quiver with dimension vector $3$.
    By a direct computation, we get $q^{-10}P(q) = q^4 + q^5 + q^6 + q^7 + q^8 + q^{10}$.
\end{ex}

\subsection{Floer-theoretic filtration}
For a pre-integrable system $\Hitch$, a filtration on $H^*(\Hitch)$ of a completely different nature is defined in~\cite{RZ1}, where 
an interested reader can find the details.
Let us sketch the construction here.

Consider $\Hitch$ as a {\KH} manifold with a Hamiltonian action of $S^1\subset \T$, and the moment map $H:\Hitch\to \R$.
One considers Hamiltonian Floer cohomologies $HF^*(p_+ H)$, where $p_+ = p+\varepsilon$, $p$ are periods of the $S^1$-action, and $\varepsilon >0$ is very small.
When $p'<p''$, there are natural Floer-theoretic maps $HF^*(p_+'H)\to HF^*(p_+'' H)$, which give rise to a direct system starting with the ordinary cohomology $H^*(\Hitch)\simeq HF^*(0_+ H)$, whose direct limit (as $p\to \infty$) is zero.
Thus, the cohomology $H^*(\Hitch)$ is filtered by the kernels $\FF_p\coloneqq\ker (H^*(\Hitch)\to HF^*(p_+ H))$.

As with most Floer-theoretic invariants, this filtration is extremely difficult to compute in full generality.
Some partial results have been proved in \cite{RZ1}.
However, an algorithmic way of computing its \textit{ranks} is given by the adjacent Morse--Bott--Floer spectral sequence $(E_p^{ij})_{p\in \Q_{\geq 0}}$ defined in \cite{RZ2}.
Its $0$-th page has $H^*(\Hitch)$ for the first column, and the pages are numbered by the periods $p$ of $S^1$-action.
As a consequence of the direct limit above being zero, this spectral sequence converges to zero.
Hence $H^*(\Hitch)$ acquires an increasing filtration $F_p\subset H^*(X)$, given by the subspace annihilated on page $E_p$, which coincides with $\FF_p$.
Using this spectral sequence, one can in most cases compute the ranks of $\FF_\bullet$ on the top-degree cohomology $H^{\mathrm{top}}(\Hitch)$.

The following observation, which relates ranks of perverse and Floer-theoretic filtration, was one of our inspirations for the present project.
As the techniques for working with $\FF_\bullet$ are wildly different from the algebro-geometric setup of this paper, the detailed proof will appear elsewhere~\cite{MZ26}.

\begin{thm}\label{thm:Floer-vs-perv}
    Let $F_n(q) := \sum_p q^{1/p} \dim \gr^\FF_p H^{\mathrm{top}}(\Hilb^n T^*E)$.
    Then $q^n F_n(q) = P_n(q)$.
\end{thm}

This result is very special to $\Hilb^n T^*E$.
Indeed, it was observed in~\cite{SZZ} that the two generating functions diverge already for $n=1$ for surfaces of parabolic and Painlevé families.
Furthermore, this coincidence is purely numerical, as we expect that the filtration $\FF_\bullet$ is tidal on $H^{\mathrm{top}}$, whereas perverse filtration almost never is.
It would be interesting to study whether this equality continues to hold beyond top degree cohomology for $\Hilb^n(T^*E)$.

\section{Equivariant multiplicities for $\Hilb^2$}\label{sec:ex-Hilb2}
In this section, we compute all equivariant multiplicities for Hilbert schemes of two points on surfaces $S$ of \cref{ex:2d-int}.
Thanks to \cref{prop:vst-eq-mult-parab}, we do not need to consider very stable components.
Note that the points in wobbly components of $\Fix^\WG(\Hilb^2 S)$ are of one of the following types:
\begin{enumerate}
    \item (separated) $I_p^{(1)}\cap I_q^{(1)}$, where $p,q\in S^\T$, $p$ is wobbly and $p\not\in C$; 
    \item (very stable support) $I_q^{(2)}$, where $q\in S^\T\setminus C$ is very stable;
    \item (wobbly support) $I_p^{(1^2)}$ and $I_p^{(2)}$, where $p\in S^\T\setminus C$ is wobbly.
\end{enumerate}

\subsection{General strategy}\label{ssec:strategy}
Before launching into the computation, let us sketch our approach.
For each wobbly point $I$ in $\Hilb^2(S)$, we construct an explicit $\T$-equivariant open $U\subset \Hilb^2(S)$, such that $U\simeq \C^4$ and $I$ is the origin.
Let us momentarily write $U = \Spec A$, $A= \C[x_1,x_2,y_1,y_2]$, where $x_i$, resp. $y_i$ are the directions of non-positive, resp. positive weight, and $\Base = \Spec B$, $B=\C[z_1,z_2]$.
Then we are working with the ideal $\mathcal{J}\subset A$, generated by the images of $z_1, z_2$ under the map $(\theta|_U)^*:B\to A$.
Observe that $\Core_I\cap U$ is given by the ideal $\mathcal{I} = (y_1,y_2)$, and write
\begin{align*}
(\theta|_U)^*(z_1) = f = \sum_{i,j} f_{ij}(x_1,x_2)y_1^iy_2^j, \qquad (\theta|_U)^*(z_2) = g = \sum_{i,j} g_{ij}(x_1,x_2)y_1^iy_2^j.
\end{align*}
Let us denote $\mathbf{R} := \C[x_1,x_2]$, so that $\Core_I\cap U\simeq \Spec \mathbf{R}$.
By definition, the equivariant multiplicity of $\Core_I$ is given as the sum over all $k\geq 0$ of the classes of torsion-free quotients of the $\mathbf{R}$-modules $\gr^d\coloneqq \mathcal{I}^d/(\mathcal{I}^d\cap \mathcal{J} + \mathcal{I}^{d+1})$ in $K^\T(\Core_I\cap U) \simeq \Q[t^{\pm 1}]$.

The computation of each $\gr^d$ is relatively straightforward.
Let $d_f$ be the smallest positive integer such that $f_{ij}\neq 0$ for some $i,j$ with $i+j=d_f$, and define $d_g$ analogously.
In our examples, we always have $d_f < d_g$.
If both $f$ and $g$ happen to be homogeneous in $y$, then we clearly have
\[
    \gr^d = \Big(\bigoplus_{i+j=d} \mathbf{R} y_1^iy_2^j\Big)/\Big(\bigoplus_{i+j=d-d_f} \mathbf{R} fy_1^iy_2^j \oplus \bigoplus_{i+j=d-d_g} \mathbf{R} gy_1^iy_2^j  \Big).
\]
If $f$ and $g$ are not homogeneous, this description still holds for $d<d_f$, where it becomes $\gr^d = \bigoplus_{i+j=d} \mathbf{R} y_1^iy_2^j$.
For $d_f\leq d < d_g$, the polynomial $g$ does not contribute, and so $\mathcal{I}^d\cap \mathcal{J}/\mathcal{I}^{d+1}\cap \mathcal{J}$ is generated by the images of $fp(y_1,y_2)$, where $fp\in \mathcal{I}^d$.
Denote by $\bar{f}$ the sum of monomials in $f$ of $y$-degree $d_f$.
Then $fp\in \mathcal{I}^d$ if and only if $d_p \geq d-d_f$, and the span of images of these elements is the same as span of $\bar{f}y_1^iy_2^j$ with $i+j=d-d_f$.
Thus we have 
\begin{align*}
    \gr^d = \Big(\bigoplus_{i+j=d} \mathbf{R} y_1^iy_2^j\Big)/\Big(\bigoplus_{i+j=d-d_f} \mathbf{R} \bar{f}y_1^iy_2^j \Big),\qquad d_f\leq d < d_g. 
\end{align*}
As soon as $d \geq d_g$, the ideal $\mathcal{J}/\mathcal{I}^{d+1}\subset A/\mathcal{I}^{d+1}$ ceases to be principal, and so the intersection $\mathcal{I}^d\cap \mathcal{J}/\mathcal{I}^{d+1}\cap \mathcal{J}$ has to be computed by hand, e.g. via Buchberger's algorithm.

Removing the torsion of $\gr^d$ is the complicated part.
If e.g. the variables in $f$ separate, that is $f = f'(x_1,x_2)f''(y_1,y_2)$, this immediately tells us that $f''$ generates a torsion submodule.
In general, we detect the torsion using \texttt{macaulay2}~\cite{M2}.
Once $\gr^d$ is fully torsion, all modules $\gr^{d'}$, $d'>d$ are fully torsion as well, so the computation terminates.
We spell out this procedure explicitly in simpler cases, and state the results for the rest in \cref{subs:punc-wobbly}.

For reader's convenience, we summarize the numerical data for all types of isolated fixed points that can appear in 2-dimensional integrable systems in \cref{table:points}; see \cref{ssec:2d-mult-ex} for details.
\begin{table}[!htbp]
    \small
    \[
    \begin{array}[c]{|c|c|c|}
        \hline
        \textup{Equation} & \textup{Weights} & \textup{Surface}\\ \hline 
        y & -1|2 & S_{\mathbb{Z}/2}\\ \hline
        y & -2|3 & S_{\mathbb{Z}/3}\\ \hline
        y & -3|4 & S_{\mathbb{Z}/4}\\ \hline
        y & -5|6 & S_{\mathbb{Z}/6}\\ \hline
        y & -1|6 & S^{\mathrm{I}}\\ \hline
        y & -1|4 & S^{\mathrm{II}}\\ \hline
        y & -1|3 & S^{\mathrm{IV}}\\ \hline
        y^2 & -1|2 & S_{\mathbb{Z}/4}\\ \hline
        y^2 & -2|3 & S_{\mathbb{Z}/6}\\ \hline
        y^3 & -1|2 & S_{\mathbb{Z}/6}\\ \hline
    \end{array}\qquad
    \begin{array}[c]{|c|c|c|}
        \hline
        \textup{Equation} & \textup{Weights} & \textup{Surface}\\ \hline 
        xy^2 & -1|2 & S_{\mathbb{Z}/3}\\ \hline
        xy^2 & -2|3 & S_{\mathbb{Z}/4}\\ \hline
        xy^2 & -4|5 & S_{\mathbb{Z}/6}\\ \hline
        x^2y^3 & -1|2 & S_{\mathbb{Z}/4}\\ \hline
        x^2y^3 & -3|4 & S_{\mathbb{Z}/6}\\ \hline
        x^2y^4 & -1|2 & S_{\mathbb{Z}/6}\\ \hline
        x^3y^4 & -2|3 & S_{\mathbb{Z}/6}\\ \hline
        x^4y^5 & -1|2 & S_{\mathbb{Z}/6}\\ \hline
    \end{array}
    \]
    \caption{Types of isolated fixed points (left: very stable, right: wobbly)}
    \label{table:points}
\end{table}

\subsection{Separated ideals}\label{subs:separated}
Let $I=I_p^{(1)}\cap I_q^{(1)}$, where $p,q\in S^\T$.
Let $x_1^{a_1}y_1^{b_1}$ be the local equation at $p$, and $x_2^{a_2}y_2^{b_2}$ the local equation at $q$; we also write $l_i = -\wt x_i$, and $w_i = l_i + k = \wt y_i$, and recall that $-a_il_i+b_iw_i=e$.
By \cref{prop:Hilb-etale}, the map $\theta$ can be étale locally presented as follows:
\begin{gather*}
    \theta:\Spec \C[x_1,x_2,y_1,y_2]\to \Spec \C[z_1,z_2], \\
    \theta^*(z_1) = x_1^{a_1}y_1^{b_1} + x_2^{a_2}y_2^{b_2}, \quad \theta^*(z_2) = x_1^{a_1}x_2^{a_2}y_1^{b_1}y_2^{b_2}.
\end{gather*}
Without loss of generality assume that $b_1\leq b_2$, and denote $f= \theta^*(z_1)$, $g= \theta^*(z_2)$.
Following the procedure described in \cref{ssec:strategy}, we need to distingush two cases.

When $b_1 < b_2$, we get $\bar{f} = x_1^{a_1}y_1^{b_1}$, and $\mathcal{J} = (f,g) = (f, x_2^{2a_2}y_2^{2b_2})$.
As the variables in $\bar{f}$ separate, all direct summands $\mathbf{R}y_1^iy_2^j\subset \mathcal{I}^d/\mathcal{I}^{d+1}$ with $i\geq b_1$ quotient to torsion submodules in $\mathrm{gr}^d$.
Furthermore, as variables in $x_2^{2a_2}y_2^{2b_2}$ separate, all such terms with $j\geq 2b_2$ quotient to torsion modules as well.
We have thus obtained a surjection
\[
    \Big(\bigoplus_{d} \gr^d\Big)_{\mathrm{tf}} \twoheadrightarrow \bigoplus_{\substack{0\leq i < b_1\\0\leq j < 2b_2}} \mathbf{R}y_1^iy_2^j.
\]
By \cref{prop:mult-parab}, we have $m_I = 2b_1b_2$, so that both sides are torsion free of the same rank.
This means that the map above is an isomorphism, and so 
\begin{equation}\label{eq:eq-mult-Hilb2-sep-easy}
    m_I(t) = [b_1]_{t^{w_1}}[2b_2]_{t^{w_2}}.
\end{equation}

When $b_1 = b_2 = b$, both $f$ and $g$ are homogeneous in $y_i$'s.
For $0\leq d<b$, we have $(\gr^d)_{\mathrm{tf}} = \gr^d = \bigoplus_{i+j = d}\mathbf{R}y_1^iy_2^j$.
When $b\leq d<2b$, we have
\begin{align*}
    \gr^d &= \bigoplus_{i+j = d}\mathbf{R}y_1^iy_2^j / \bigoplus_{i+j = d-b}\mathbf{R}fy_1^iy_2^j\\
    &= \bigoplus_{d-b\leq i< d-1} \mathbf{R}y_1^iy_2^{d-i} \oplus \bigoplus_{0\leq i< d-b} \Big( \mathbf{R}y_1^{i}y_2^{d-i} \oplus \mathbf{R}y_1^{b+i}y_2^{d-b-i} \Big)/\mathbf{R}fy_1^iy_2^{d-b-i}.
\end{align*}
The modules in the second direct sum are torsion-free of rank $1$, but not necessarily locally free; the latter holds only when one of $a_i$'s vanishes.
In any case, the resulting class in $K$-theory is 
$\sum_{i+j=d} t^{iw_1+jw_2} - \sum_{i+j=d-b} t^{e + iw_1 + jw_2}$.
When $2b\leq d<3b-1$, we begin to acquire torsion submodules $y_1^{b+i}y_2^{b+j}\mathbf{R}/(x_1^{a_1}x_2^{a_2})$ thanks to the presence of $g$.  
As a consequence, we have 
\begin{align*}
    (\gr^d)_{\mathrm{tf}} = \bigoplus_{d-2b-1\leq i< b-1} \Big( \mathbf{R}y_1^{i}y_2^{d-i} \oplus \mathbf{R}y_1^{b+i}y_2^{d-b-i} \Big)/\mathbf{R}fy_1^iy_2^{d-b-i}.
\end{align*}
The resulting class in $K$-theory is $\sum_{d-2b-1\leq i< b-1} t^{iw_1+(d-i)w_2}(1+t^{bw_1-bw_2} -t^{e -bw_2})$.
Finally, when we reach $d=3b-1$, we get $(\gr^d)_{\mathrm{tf}}=0$, so this is where we stop.
Summing everything up, we get
\begin{equation}\label{eq:eq-mult-Hilb2-sep}
\begin{aligned}
    m_I(t) &= \sum_{0\leq i,j <b} (t^{iw_1+jw_2} + t^{(b+i)w_1+jw_2} + t^{iw_1+(b+j)w_2}) - \sum_{0\leq i,j <b} t^{e+iw_1+jw_2}\\
    &= (1+t^{bw_1}+t^{bw_2}-t^e)[b]_{t^{w_1}}[b]_{t^{w_2}}.
\end{aligned}
\end{equation}
In particular, since $b_iw_i-e = a_il_i$, the equivariant multiplicity $m_I(t)$ might acquire some negative coefficients.
This is indeed what happens for e.g. $S = S_{\Z/3}$, $p$ and $q$ distinct isolated wobbly points. 

\begin{rmk}
    Note that in both cases, we have $m_I(t) = d(t)[b]_{t^{w_1}}[b]_{t^{w_2}}$, where $d(t)$ is either $1+t^{b_2w_2}$ or $1+t^{b_1w_1}+t^{b_2w_2}-t^e$.    
    Thus \cref{conj:finite-cover} holds in this particular case.
\end{rmk}

\subsection{Wobbly ideals with very stable support}\label{subs:wobbly-stable}
In order to provide a local model for punctually supported sheaves, recall that $\Hilb^2 S$ can be obtained as a blowup of $\Sym^2 S$ along the diagonal.
Let us fix a point $p\in S^\T\setminus C$, and pick an étale neighbourhood $U = \Spec \C[x,y]$, such that the local equation of $\Core$ at $p$ is $\mathscr{E}(x,y)=x^ay^b$, and $\wt x= -l$, $\wt y = w$.
Then $\Sym^2 U = \Spec \C[x_1,x_2,y_1,y_2]^{\mathfrak{S}_2}$, $\Base = \Spec\C[z_1,z_2]^{\mathfrak{S}_2}$, and the map $\theta$ is given by
\[
    \theta^*(z_1+z_2) = h_1\coloneqq x_1^ay_1^b+x_2^ay_2^b,\qquad \theta^*(z_1z_2) = h_2\coloneqq x_1^ax_2^ay_1^by_2^b.
\]
The following isomorphism is a standard exercise:
\begin{gather*}
    \C[s,t,u_0,v_0,w_0]/(v_0^2-u_0w_0) \simto\C[x_1,x_2,y_1,y_2]^{\mathfrak{S}_2};\\
    u_0\mapsto (x_1-x_2)^2,\quad v_0\mapsto (x_1-x_2)(y_1-y_2),\quad w_0\mapsto (y_1-y_2)^2,\\
    s\mapsto x_1+x_2,\quad t\mapsto y_1+y_2.
\end{gather*}
The blow up of this ring at the singular plane is covered by two affine charts:
\begin{align*}
    U_{1}: \C[s,t,u_0,v_0,w_0]/(v_0^2-u_0w_0)\to \C[s,t,w_0,\eta],\quad u_0\mapsto \eta^2w_0,\quad v_0\mapsto \eta w_0;\\
    U_{2}: \C[s,t,u_0,v_0,w_0]/(v_0^2-u_0w_0)\to \C[s,t,u_0,\xi],\quad v_0\mapsto \xi u_0,\quad w_0\mapsto \xi^2u_0.
\end{align*}
The $\T$-weights of these variables are as follows:
\[
\wt s = -l, \wt t = w, \wt u_0 = -2l, \wt v_0 = w-l = k, \wt w_0 = 2w, \wt\eta =-l-w, \wt\xi = w+l.
\]
In terms of $\Hilb^2(S)$, the first chart is the neighbourhood of $I_p^{(1^2)}$, and the second is the neighbourhood of $I_p^{(2)}$.
Taking the images of $h_1,h_2$, we arrive to the situation described in \cref{ssec:strategy}, where the role of $x_i$'s, resp. $y_i$'s is played by the coordinates of negative, resp. positive weight.

Let us assume for now that $p$ is a very stable isolated point; in other words, $a=0$.
Since $I_p^{(1^2)}$ is very stable, we only care about the chart $U_2$.
The images of $h_1,h_2$ in this chart (up to an unimportant multiplicative constant), which we will call $f,g$, are given by the following list:
\begin{itemize}
    \item $b=1$: $f=t$, $g=t^2-2u_0\xi^2$;
    \item $b=2$: $f=u_0\xi^2+t^2$, $g=t^4-2u_0\xi^2t^2+u_0^2\xi^4$;
    \item $b=3$: $f=3u_0\xi^2t-t^3$, $g=t^6-3u_0\xi^2t^4+3u_0^2\xi^4t^2-u_0^3\xi^6$.
\end{itemize}
We consider the associated graded with respect to the variables $(t,\xi)$.
The computation is simplified by the fact that all the polynomials above are homogeneous.

When $\mathscr{E}(x,y) = y$, we have $\gr^0 = \mathbf{R}$, $\gr^1 = \xi\mathbf{R}$, and $\gr^2 = \xi^2\mathbf{R}/u_0^2$ is torsion.
Thus 
\begin{equation}\label{eq:eqm-y-2}
    m_{I^{(2)}}(t) = [2]_{t^{w+l}}\qquad\text{ when }\mathscr{E}(x,y) = y.
\end{equation}

When $\mathscr{E}(x,y) = y^2$, we have $\gr^0 = \mathbf{R}$, $\gr^1 = \xi\mathbf{R}\oplus t\mathbf{R}$.
Recall from \cref{table:points} that this equation only allows weights with $l = w-1$.
Because of the simple shape of $f$, we easily compute
\begin{gather*}
    \gr^2 = \xi t\mathbf{R}\oplus(\xi^2\mathbf{R}\oplus t^2\mathbf{R})/(t^2+2u_0\xi^2)\mathbf{R} = \xi t\mathbf{R}\oplus \xi^2\mathbf{R},\\
    \gr^3 = (\xi^2t\mathbf{R}\oplus t^3\mathbf{R})/(t^3+u_0\xi^2t)\mathbf{R}\oplus(\xi^3\mathbf{R}\oplus \xi t^2\mathbf{R})/(\xi t^2+u_0\xi^3)\mathbf{R} = \xi^2t\mathbf{R}\oplus \xi^3\mathbf{R}.
\end{gather*}
The next summand $\gr^4$ is a direct sum of $\xi^3t\mathbf{R}$ and a more complicated module
\[
    (\xi^4\mathbf{R}\oplus \xi^2 t^2\mathbf{R}\oplus t^4\mathbf{R})/(\xi^2 f\mathbf{R} \oplus t^2 f\mathbf{R} \oplus g\mathbf{R}).
\]
However, observe that the functions spanning the denominator are linearly independent.
This means that the quotient module has zero rank, and so it does not contribute to $(\gr^4)_{\mathrm{tf}}$.
A similar argument shows that $(\gr^5)_{\mathrm{tf}}=0$.
Adding everything up, we obtain:
\begin{equation}\label{eq:eqm-y2-2}
    m_{I^{(2)}}(t) = [2]_{t^w}[4]_{t^{w+l}}\qquad\text{ when }\mathscr{E}(x,y) = y^2.
\end{equation}

Finally, let $\mathscr{E}(x,y) = y^3$.
By the same type of computation as before (essentially using that $f$ consists of two monomial terms, one with invertible coefficient) we find that $\bigoplus_{i\leq 5}\gr^i$ is free, with basis given by monomials $\xi^i t^j$ with $j\leq 2$.
In $\gr^6$, as $g$ begins to contribute, the following summand appears:
\begin{align*}
    (\xi^6\mathbf{R}\oplus & \xi^4 t^2\mathbf{R}\oplus \xi^2 t^4\mathbf{R}\oplus t^6\mathbf{R}) /((t^6-3u_0\xi^2 t^4)\mathbf{R} \oplus (\xi^2 t^4-3u_0\xi^4 t^2)\mathbf{R} \oplus g\mathbf{R})\\
    & = (\xi^6\mathbf{R}\oplus \xi^4 t^2\mathbf{R})/(u_0^3\xi^6 - 3u_0^2\xi^4 t^2)\mathbf{R} = (\xi^6\mathbf{R}\oplus (3\xi^4 t^2-u_0\xi^6)\mathbf{R})/u_0^2(u_0\xi^6 - 3\xi^4 t^2)\mathbf{R}.
\end{align*}
This summand only contributes $\xi^6\mathbf{R}$ to the torsion-free part.
In $\gr^7$, we have one summand of similar shape, contributing $\xi^7\mathbf{R}$, and a summand with generators $\{\xi^6 t, \xi^4 t^3, \xi^2 t^5, t^7\}$, modulo the relations $\{t^4 f, \xi^2t^2 f, \xi^4f,tg\}$.
We can easily check that these relations are linearly independent, so that the quotient is fully torsion.
Starting from $\gr^8$, only summands of this shape will appear, therefore we ran out of contributions to equivariant multiplicity.
Recall from \cref{table:points} that the local equation $\mathscr{E}(x,y) = y^3$ only allows weights $w=2$, $w+l=3$.
Substituting these values and summing up all terms from the discussion above, we obtain a polynomial with positive coefficients:
\begin{equation}\label{eq:eq-mult-Hilb2-punct-y3}
    m_{I^{(2)}}(t) = [22]_t - t - t^{16} - t^{19} - t^{20}\qquad\text{ when }\mathscr{E}(x,y) = y^3.
\end{equation}

\subsection{Punctual wobbly support: case $\mathscr{E}(x,y) = xy^2$}\label{subs:punc-wobbly}
Now assume that $p$ is a wobbly isolated point.
Observe that this automatically restricts us to parabolic surfaces, so that $k=1$ and $l=w-1$.
In this case both ideals $I_p^{(1^2)}$, $I_p^{(2)}$ are wobbly, and so we need to perform computations in both charts $U_1$, $U_2$ of the blow up.
Let us begin by spelling out all details in the case of local equation $\mathscr{E}(x,y) = xy^2$. 

For the ideal $I_p^{(1^2)}$, we work in the chart $U_1$, and consider associated graded with respect to $(t,w_0)$.
The images of $h_1,h_2$ in this chart are
\[
    f = sw_0 + st^2 + 2\eta w_0t,\qquad g = s^2w_0^2 -(2s^2w_0t^2+\eta^2w_0^3) + (s^2t^4 +2\eta^2w_0^2t^2)- \eta^2w_0t^4.
\]
As in \cref{subs:separated}, we can replace $g$
\[
    g' = g - fsw_0 = -(3s^2w_0t^2+2s\eta w_0^2t+\eta^2w_0^3) + (s^2t^4 +2\eta^2w_0^2t^2)- \eta^2w_0t^4.
\]
It is clear that $\bar{f} = sw_0$.
Therefore $\gr^0 = \mathbf{R}$, $\gr^1 = t\mathbf{R}\oplus w_0\mathbf{R}/s$, and $\gr^2 = t^2\mathbf{R}\oplus tw_0\mathbf{R}/s\oplus w_0^2\mathbf{R}/s$.
In degree $3$, besides multiples of $\bar{f}$ we also have the contribution of $\bar{g}'$.
However, the latter does not involve the monomial $t^3$, so that we still obtain $(\gr^3)_{\mathrm{tf}} = t^3\mathbf{R}$.
Finally, in degree $4$ the element
\[
    h\coloneqq (3s^2t^2 + 2s\eta w_0t + \eta^2w_0^2) + sg
\]
has the lowest $(t,w_0)$-degree $4$, and its image $\bar{h}$ contains the monomial $t^4$ with non-zero coefficient.
This implies that $(\gr^4)_{\mathrm{tf}} = 0$, and so in total we get 
\begin{equation}\label{eq:eqm-xy2-11}
    m_{I^{(1^2)}}(t) = 1 + t^w + t^{2w} + t^{3w} = [4]_{t^w}\qquad\text{ when }\mathscr{E}(x,y) = xy^2.
\end{equation}

For the ideal $I_p^{(2)}$, we work in the chart $U_2$.
The images of $h_1,h_2$ in this chart are
\[
    f = su_0\xi^2 + 2u_0\xi t + st^2,\qquad g = (s^2-u_0)(u_0^2\xi^4-2u_0\xi^2t^2 + t^4).
\]
Note that both of them are homogeneous in $(t,\xi)$.
We clearly have $\gr^0\oplus \gr^1 = \mathbf{R}\oplus t\mathbf{R}\oplus \xi\mathbf{R}$.
The modules $\gr^2$ and $\gr^3$, obtained by quotienting out $f\mathbf{R}$ and $\xi f\mathbf{R}\oplus tf\mathbf{R}$ respectively, can be easily checked to be torsion-free.
In degree $4$ we need to quotient out $g$, which immediately produces a torsion element $u_0\xi^4-2u_0\xi^2t^2 + t^4$, whose annihilator is $(s^2-u_0)$. 
Thus we can forget about $g$ and replace $t^4$ with $2u_0\xi^2t^2-u_0\xi^4$ in $(\gr^4)_{\mathrm{tf}}$.
Repeating this for the rest of the relations, we can whittle the number of generators of $(\gr^4)_{\mathrm{tf}}$ down to one.
The easiest way to represent this is with a sequence of matrices, where columns stand for relations, and rows for generators:
\begin{align*}
    &\begin{blockarray}{ccccc}
         & \xi^2 f & \xi tf & t^2f & g \\
        \begin{block}{c(cccc)}
        \xi^4 & su_0 & 0 & 0 & (s^2-u_0)u_0^2\\
        \xi^3t & 2u_0 & su_0 & 0 & 0\\
        \xi^2t^2 & s & 2u_0 & su_0 &2(s^2-u_0)u_0\\
        \xi t^3 & 0 & s & 2u_0 & 0\\
        t^4 & 0 & 0 & s & s^2-u_0\\
        \end{block}
    \end{blockarray}
\rightsquigarrow
    \begin{blockarray}{cccc}
         & \xi^2 f & \xi tf & t^2f \\
        \begin{block}{c(ccc)}
        \xi^4 & su_0 & 0 & -su_0^2\\
        \xi^3t & 2u_0 & su_0 & 0\\
        \xi^2t^2 & s & 2u_0 & 3su_0 \\
        \xi t^3 & 0 & s & 2u_0 \\
        \end{block}
    \end{blockarray}\\
&\qquad\qquad\qquad\qquad\rightsquigarrow
    \begin{blockarray}{ccc}
         & \xi^2 f & 2\xi tf-s\xi^2 f \\
        \begin{block}{c(cc)}
        \xi^4 & su_0 & 0\\
        \xi^3t & 2u_0 & 0\\
        \xi^2t^2 & s & 4(u_0-s^2)\\
        \end{block}
    \end{blockarray}
\rightsquigarrow
    \begin{blockarray}{cc}
         & \xi^2 f \\
        \begin{block}{c(c)}
        \xi^4 & su_0 \\
        \xi^3t & 2u_0 \\
        \end{block}
    \end{blockarray}
\rightsquigarrow
\xi^4\mathbf{R}.
\end{align*}
Thus we obtain $(\gr^4)_{\mathrm{tf}}=\xi^4\mathbf{R}$.
Finally, for $\gr^5$ we get a quotient of a rank $6$ free $\mathbf{R}$-module by $6$ linearly independent relations $\{ \xi^3f,\xi^2tf,\xi t^2f, t^3f, \xi g,\xi t \}$, so that $(\gr^5)_{\mathrm{tf}}=0$.
Summing everything up, we get
\begin{equation}\label{eq:eqm-xy2-2}
    \begin{aligned}
        m_{I^{(2)}}(t) = [5]_{t^{2w-1}} + t^w[3]_{t^{2w-1}} + t^{2w}[2]_{t^{2w-1}} - t^{w+1}[2]_{t^{w}}\qquad\text{ when }\mathscr{E}(x,y) = xy^2. 
    \end{aligned}
\end{equation}
When $w=2$ (that is, $p$ is a fixed point of $S_{\mathbb{Z}/3}$), several terms cancel out, and $m_{I^{(2)}}(t)$ becomes a polynomial with non-negative coefficients.
For $w=3$ and $w=5$ (corresponding to points in $S_{\Z/4}$ and $S_{\Z/6}$), such cancellations do not occur.

\subsection{Punctual wobbly support: other cases}
Analysis of a similar nature (albeit more involved) to \cref{subs:punc-wobbly} can be performed for the remaining isolated wobbly points.
For details, we refer the reader to the Jupyter notebook at~\cite{Jup}.
We summarize the outcomes of computations performed with the help of \texttt{macaulay2} in \cref{table:eqmultH2}.
For completeness, it also contains the formulas~(\ref{eq:eqm-y-2}--\ref{eq:eqm-xy2-2}), as well as relevant specializations of~\eqref{eq:vst-eq-mult-parab-poly}.
In the right column, we chose an expression in a way that shows that $m_{I^{(2)}}(t)-[m]_{t^w}[2m]_{t^{w+l}} = (t^{w+l}-1)P(t)$ for some polynomial $P(t)$ with positive coefficients.

\begin{table}[!htbp]
    \small
    \[\arraycolsep=1.4pt\def\arraystretch{1.2}
    \begin{array}[c]{|c|c|c|}
        \hline
        \textup{Equation} & m_{I^{(1^2)}}(t) & m_{I^{(2)}}(t)\\ \hline
        y & 1 & [2]_{t^{w+l}}\\ \hline
        y^2 & [4]_{t^{w}} & [2]_{t^{w}}[4]_{t^{2w-1}}\\ \hline
        y^3 & [3]_{t^2}[3]_{t^4} & [3]_{t^2}[6]_{t^3} + (t^{18} - t^{16})[2]_{t^3}\\ \hline
        xy^2 & [4]_{t^{w}} & [2]_{t^{w}}[4]_{t^{2w-1}} + (t^{2w-2}-t^{w-1})(1+t+t^{2w-1}+t^{6w-2})\\ \hline
         &  & [3]_{t^w}[6]_{t^{2w-1}} + (t^{2w-2}-1)(t^{w+2}+t^{3w+2}+t^{10w-4})\\
        \smash{\raisebox{.5\normalbaselineskip}{$x^2y^3$}} & \smash{\raisebox{.5\normalbaselineskip}{$[3]_{t^w}[3]_{t^{2w}} + t^{3w}-t^{w+1}$}} & + (t^{2w-2}-1)(t^{w-1}+1)(t^{2w+2}+t^{4w+1}+t^{7w-1}+t^{11w-4})\\\hline
         &  & [4]_{t^2}[8]_{t^3} + (t^{14}-t^6)[2]_{t^3} + (t^{24}-t^{18})[2]_{t^5}\\
        \smash{\raisebox{.5\normalbaselineskip}{$x^2y^4$}} & \smash{\raisebox{.5\normalbaselineskip}{$[4]_{t^2}[4]_{t^4}$}} &  + (t^{16}-t^{12})[2]_{t^{10}} + (t^{20}-t^{18})[2]_{t^{10}} \\\hline
         &  & [4]_{t^3}[8]_{t^5} + (t^{43}-t^{33})[2]_{t^2} + (t^{23}-t^{19})(1+t^4+t^7+t^{25})\\
        \smash{\raisebox{.5\normalbaselineskip}{$x^3y^4$}} & \smash{\raisebox{.5\normalbaselineskip}{$[4]_{t^3}[4]_{t^6}$}} & + (t^{23}-t^{9})[2]_{t^2} + (t^{20}-t^{14})(1+t^2+t^{12}+t^{30}) + (t^{18}-t^{6})[2]_{t^{22}}\\ \hline
         &  & [5]_{t^2}[10]_{t^3} + (t^{10}-1)(t^6+t^8+t^9+t^{12}+t^{26}+t^{29}) \\
        \smash{\raisebox{.5\normalbaselineskip}{$x^4y^5$}} & \smash{\raisebox{.5\normalbaselineskip}{$[5]_{t^2}[5]_{t^4} + (t^{10} - t^9)[3]_{t^2}$}} & + (t^{33}-t^{15})[2]_{t^5} + (t^{23}-t^{21}) + (t^{19}-t^{11}) + (t^{28}-t^{14}) \\ \hline
    \end{array}\qquad
    \]
    \caption{Equivariant multiplicities in $\Hilb^2$, non-separated case}
    \label{table:eqmultH2}
\end{table}

We can check that for $\mathscr{E}(x,y) = x^2y^4$ and $x^3y^4$ (as well as for $xy^2$, see above) all graded pieces $\mathrm{gr}^i$ for the ideal $I^{(1^2)}$ are locally free, which is reflected in palindromicity of $m_{I^{(1^2)}}(t)$.
On the other hand, while the polynomial $m_{I^{(2)}}(t)$ for $\mathscr{E}(x,y) = x^4y^5$ can be checked to be positive, this is a combinatorial coincidence; neither the polynomial is palindromic, nor the sheaves are locally free.

\begin{cor}
    \cref{conj:eq-mult-positivity,conj:mult-inequality} hold for $\Hilb^2 S$.
\end{cor}
\begin{proof}
    Direct verification using \cref{table:eqmultH2}.
\end{proof}

\begin{rmk}
    In the particular case $n=2$, one can canonically associate locally free sheaves to our torsion-free sheaves by taking double duals.
    While their classes in $K$-theory produce positive polynomials (e.g. for $\mathscr{E}(x,y) = x^2y^3$, we get $[7]_{t^w} + t^{2w} + t^{5w-1}$ at the ideal $I^{(1^2)}$), this does not appear to be a meaningful procedure.
\end{rmk}

\section{Mirror symmetry for Hilbert schemes}\label{sec:MS}
In this section we recall several notions of mirror symmetry for Higgs bundles, and put some of our prior computations in this context.

\subsection{Ansätzen}\label{subs:MS-heuristics}
Let $G$ be a reductive group, $C$ a smooth projective curve, and $\CHiggs_G$ the moduli stack of $G$-Higgs bundles on $C$.
Donagi--Pantev proposed~\cite{donagi2012langlands} that there should be an equivalence of the form
\begin{equation*}\tag{DLC}\label{eq:DLC-conj}
    \mathfrak{L}_G : \mathrm{QCoh}(\CHiggs_G) \simto \mathrm{QCoh}(\CHiggs_{\LG}),
\end{equation*}
where $\LG$ is the Langlands dual group of $G$.
This equivalence most definitely does not exist as stated, see e.g.~\cite[Conj~1.3]{padurariu2025dolbeault} for a more precise proposal. 

There are two different ways of looking at~\eqref{eq:DLC-conj}, which provide different expected properties of this equivalence.
On one hand, the Hitchin fibration $h: \CHiggs_G\to \Base_G$ is generically a fibration in abelian varieties.
Furthermore, $\Base_G = \Base_\LG$.
Donagi--Pantev showed that over the generic locus $\Base^\#_G$ these two fibrations are dual, and further postulated that the restriction $\mathfrak{L}_G|_{\Base^\#_G}$ should be the corresponding Fourier--Mukai transform; see~\cite{arinkin2013autoduality,melo2019fourier} for further extensions.
For Hitchin fibration, the role of the identity of an abelian variety is played by the Kostant section $s:\Base_G \to \CHiggs_G$.
Hence, based on properties of Fourier--Mukai transform, one expects the following:

\begin{ansatz}\label{an:FM}
    $\mathfrak{L}_G$ exchanges $\mathcal{O}_{\CHiggs_G}$ with $s_*\mathcal{O}_{\Base_G}$.
\end{ansatz}

On another hand, one can conceptualize~\eqref{eq:DLC-conj} as the Dolbeault shape of geometric Langlands correspondence:
\[
    \begin{tikzcd}
       \mathcal{D}\text{-mod}(\mathcal{B}\mathrm{un}_G)\ar[r,"\mathbb{L}_G"]\ar[d,squiggly,"\substack{\text{semiclassical} \\ \text{limit}}"'] & \mathrm{QCoh}(\mathcal{C}\mathrm{onn}_\LG)\ar[d,squiggly,"\substack{\text{non-abelian} \\ \text{Hodge theory}}"]\\
       \mathrm{QCoh}(\CHiggs_G)\ar[r,"\mathfrak{L}_G"] & \mathrm{QCoh}(\CHiggs_\LG)
    \end{tikzcd}
\]
As such, it should intertwine the action of Hecke operators on both sides.
Denote the set of dominant coweights of $G$ by $\mathbf{X}^+$, and fix $c\in C$, $\beta\in \mathbf{X}^+$.
On the left, Hecke operators $\mathsf{H}_{\beta,c}$ are given by local modifications at $c\in C$, associated to the cell in the affine Grassmannian $\mathrm{Gr}_G$ given by $\beta$.
On the right, they are given by the operators $\mathsf{T}_{\beta,c}$ of tensoring with the tautological vector bundle, associated to the irreducible $\LG$-module of highest weight $\beta$.  
\begin{ansatz}\label{an:Hecke}
    The equivalence $\mathfrak{L}_G$ exchanges $\mathsf{H}_{\beta,c}$ with $\mathsf{T}_{\beta,c}$.
\end{ansatz}

Similarly, for every parabolic subgroup $P\subset G$ with Levi $L$, we should have parabolic induction functors, also known as Eisenstein series $\mathsf{Eis}_L^G : \mathrm{QCoh}(\CHiggs_L)\to \mathrm{QCoh}(\CHiggs_G)$.
Note that neither Hecke operators nor Eisenstein series are not currently defined in full generality in the context of Dolbeault Langlands.
\begin{ansatz}\label{an:Eis}
    For any $L\subset P\subset G$, we have $\mathsf{Eis}_{\widecheck{L}}^\LG \circ \mathfrak{L}_L \simeq \mathfrak{L}_G\circ \mathsf{Eis}_L^G$.
\end{ansatz}

In~\cite{hausel2022very}, the authors considered the restriction of~\eqref{eq:DLC-conj} to the moduli of semistable Higgs bundles for $G = SL_n$.
Combining \cref{an:FM,an:Hecke}, they predicted that mirror duals of certain upward flows $\mathcal{L}_\delta\coloneqq \mathcal{O}_{W^+_\delta}$ are given by tautological vector bundles $\Lambda_\delta$.
They further checked this prediction in two ways:
\begin{itemize}[leftmargin=6mm]
    \item Over the generic locus $\Base^\#_G$, the two sheaves get swapped under Fourier--Mukai transform;
    \item Globally, $\chi_t(\RHom(\Lambda_\delta,\mathcal{L}_{\delta'})) = \chi_t(\RHom(\mathcal{L}_\delta,\Lambda_{\delta'}))$ up to a sign and a power of $t$.
\end{itemize}
Following the same strategy, we will produce some mirror pairs out of \cref{an:FM,an:Eis}.

\subsection{Parabolic Higgs bundles}\label{ssec:par-Higgs}
The same predictions and ansätzen can be transported to the setup of parabolic $GL$-Higgs bundles.
This is relevant for us, as Hilbert schemes of points on surfaces of the parabolic family are isomorphic to certain moduli of stable parabolic Higgs bundles~\cite{groechenig2014hilbert}.
For $S_{\mathbb{Z}/e}$, $e>1$ one obtains moduli of parabolic Higgs bundles of rank $ne$ on $\mathbb{P}^1$ with $3$ or $4$ (for $e=2$) punctures, while for $T^*E$ one gets parabolic Higgs bundles of rank $n$ on $E$ with one very mild puncture.

Instead of stating the precise isomorphisms, which we will not need, let us roughly describe what they do for $n=1$.
Let $\widetilde{C}$ be a smooth projective curve with an action of cyclic group $\Gamma = \mathbb{Z}/e$, $\widehat{C} = [\widetilde{C}/\Gamma]$ the stacky quotient, and $C$ the coarse moduli of $\widehat{C}$.
Denote by $p_1,\ldots,p_r\in C$ the collection of stacky points, and denote by $w_i$ the cardinality of the stabilizer of $p_i$ in $\widehat{C}$.
It is known that the moduli of parabolic Higgs bundles on $C$ with parabolic points $(p_i)$ and a flag of length $w_i$ at each $p_i$ is isomorphic to the moduli of $\Gamma$-equivariant Higgs bundles on $\widetilde{C}$.
In our case, this means that the surface $S_{\mathbb{Z}/e}$ is expressed in terms of $\mathbb{Z}/e$-equivariant Higgs bundles on $E$ of rank $e$.
As the canonical line bundle on $E$ is trivial, a Higgs bundle on $E$ is a pair $(\mathcal{V},\Phi)$, where $\mathcal{V}$ is a vector bundle and $\Phi$ an automorphism of $\mathcal{E}$.

Let $\mathrm{e}\in E$ be the identity under the group law, and denote $L_{p} = \mathcal{O}(p-\mathrm{e})$ for each $p\in E$.
The fiber of $S_{\mathbb{Z}/e}$ at a generic point $a\in \C$ is identified with Higgs bundles $\sum_{\zeta\in \mathbb{Z}/e} (L_{\zeta(p)}, \zeta(a))$.
At the singular fiber $0\in\C$, some of the line bundles $L_{\zeta(p)}$ coincide when $p$ has non-trivial stabilizer.
In this case, we can either self-extend $L_p$ to an indecomposable vector bundle $L_{ip}$, or add a nilpotent Higgs field, or a combination thereof.
The points on the central $\mathbb{P}^1$ correspond to self-extensions with zero Higgs field, and the outermost points are direct sums with nilpotent Higgs field $L\simto L\simto\ldots \simto L$.
This describes very stable points in $S^\T$.

\begin{ex}
    Let us illustrate what happens on the leg of the singular fiber of $S_{\mathbb{Z}/3}$ corresponding to $p = \mathrm{e}$.
    Here $L_p = \mathcal{O}$, which has iterated self-extensions $\mathcal{O}^{(k)}$ for all $k$, together with short exact sequences $\mathcal{O}^{(k_1)},\hookrightarrow \mathcal{O}^{(k_1+k_2)}\twoheadrightarrow \mathcal{O}^{(k_2)}$.
    The $\T$-fixed points look as follows:
    
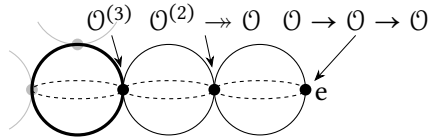
\[
    \tikz[thick,xscale=.6,yscale=.6,font=\normalsize,baseline=(current  bounding  box.center)]{
        \draw[gray!50,thin] (0.86,1.5) arc[start angle=-30, end angle=-150, radius=1];
        \draw[gray!50,thin] (-1.5,0.86) arc[start angle=60, end angle=-60, radius=1];
        \filldraw [gray!50] (-1,0) circle (3pt);
        \filldraw [gray!50] (0,1) circle (3pt);
        \filldraw [color=black,fill=gray!0,very thick, fill opacity=0] (0,0) circle (1);
        \filldraw [color=black,fill=gray!0,thin] (2,0) circle (1);
        \filldraw [color=black,fill=gray!0,thin] (4,0) circle (1);
        \draw[dash pattern=on 1.5pt off 1.5pt,thin] (0,0) ellipse (1 and 0.2);
        \draw[dash pattern=on 1.5pt off 1.5pt,thin] (2,0) ellipse (1 and 0.2);
        \draw[dash pattern=on 1.5pt off 1.5pt,thin] (4,0) ellipse (1 and 0.2);
        \filldraw [black] (1,0) circle (3pt);
        \filldraw [black] (3,0) circle (3pt);
        \filldraw [black] (5,0) circle (3pt);
        \node at (0.7,1.6) {$\mathcal{O}^{(3)}$};
        \node at (2.8,1.6) {$\mathcal{O}^{(2)}\twoheadrightarrow \mathcal{O}$};
        \node at (6.1,1.55) {$\mathcal{O}\to \mathcal{O}\to \mathcal{O}$};
        \draw[-Stealth,very thin] (0.75,1.2) -- (0.95,0.5);
        \draw[-Stealth,very thin] (2.8,1.2) -- (2.95,0.5);
        \draw[-Stealth,very thin] (6.1,1.2) -- (5.2,0.2);
        \node at (5.35,-0.1) {$\mathrm{e}$};  
	}
    \]
    For example, the line between $\mathcal{O}^{(3)}$ and $\mathcal{O}^{(2)}\twoheadrightarrow \mathcal{O}$ parameterizes Higgs bundles $(\mathcal{O}^{(3)}, t\phi)$, where $\phi$ is the composition $\mathcal{O}^{(3)}\twoheadrightarrow \mathcal{O}\hookrightarrow \mathcal{O}^{(3)}$.
\end{ex}

For every $p\in C$, the realization of $S$ via parabolic Higgs bundles gives rise to a tautological vector bundle $V_p$ on $S$ of rank $e$, which is simply the fiber at $p$ of the underlying vector bundle.
Moreover, for a stacky $p$ with stabilizer of cardinality $w$ the vector bundle $V_p$ has a canonical filtration $V_p^1\subset \ldots\subset V_p^w = V_p$, given by the parabolic structure. 
For an isolated very stable point $q\in S^\T$ corresponding to $p$, we denote the first subbundle in this filtration by $V_q$.
Using the explicit description above, we can easily deduce the fibers of these vector bundles at very stable points.
We list them in \cref{table:fibers}, where $q_i$ stands for a very stable point with core multiplicity $i$; note that for $S = T^*E$, we have $V_p = L_p$, so that $\chi_t(L_p|_{p'}) = 1$.

\begin{table}[!htbp]
    \small
    \[
    \begin{array}[c]{|c|c|c|}
        \hline
        S_{\mathbb{Z}/2} & V_{q_2} & V_{q_1}\\ \hline 
        q_2 & 2 & 1\\ \hline 
        q_1 & [2]_t & 1\\ \hline 
    \end{array}\quad 
    \begin{array}[c]{|c|c|c|}
        \hline
        S_{\mathbb{Z}/3} & V_{q_3} & V_{q_1}\\ \hline 
        q_3 & 3 & 1\\ \hline 
        q_1 & [3]_t & 1\\ \hline 
    \end{array}\quad 
    \begin{array}[c]{|c|c|c|c|}
        \hline
        S_{\mathbb{Z}/4} & V_{q_4} & V_{q_2} & V_{q_1}\\ \hline 
        q_4 & 4 & 2 & 1\\ \hline 
        q_2 & 2[2]_t & 2 & 1\\ \hline 
        q_1 & [4]_t & [2]_{t^2} & 1\\ \hline 
    \end{array}\quad 
    \begin{array}[c]{|c|c|c|c|c|}
        \hline
        S_{\mathbb{Z}/6} & V_{q_6} & V_{q_3} & V_{q_2} & V_{q_1}\\ \hline 
        q_6 & 6 & 3 & 2 & 1\\ \hline 
        q_3 & 3[2]_t& 3 & [2]_t & 1\\ \hline 
        q_2 & 2[3]_t & [3]_t & 2 & 1\\ \hline 
        q_1 & [6]_t & [3]_{t^2} & [2]_{t^3} & 1\\ \hline  
    \end{array}
    \]
    \caption{Fibers of tautological bundles}
    \label{table:fibers}
\end{table}

\subsection{Derived McKay}\label{ssec:McKay}
For a smooth quasi-projective surface $S$, the isospectral Hilbert scheme $\IHilb^n S$ is defined as the reduced fiber product
\begin{equation}\label{eq:isospec-diag}
    \begin{tikzcd}
        \IHilb^n S\ar[r,equal] & (S^n\times_{\Sym^n S} \Hilb^n S)_{\mathrm{red}}\ar[r,"\rho"]\ar[d,"\pi"] & \Hilb^n S\ar[d] \\
        & S^n \ar[r] & \Sym^n S
    \end{tikzcd}
\end{equation}
It was shown in~\cite{bridgeland2001mckay,krug2018remarks} that the following functors are derived equivalences:
\begin{align*}
    \Phi & = \pi_*\rho^* : D^b(\Hilb^n S) \to D^b_{\mathfrak{S}_n}(S^n),\\
    \Psi & = (-)^{\mathfrak{S}_n}\circ\rho_*\pi^*: D^b_{\mathfrak{S}_n}(S^n)\to D^b(\Hilb^n S).
\end{align*}
However, these equivalences are not inverse to each other, that is $\Phi\circ \Psi \not\simeq \mathrm{Id}$.
For a composition $\lambda\vDash n$, let us also denote $\Phi_\lambda = \boxtimes_i \Phi_{\lambda_i}$, $\Psi_\lambda = \boxtimes_i \Psi_{\lambda_i}$, and
\[
    \Psi^\lambda\coloneqq (-)^{\mathfrak{S}_\lambda}\circ\rho_*\pi^*: D^b_{\mathfrak{S}_\lambda}(S^n)\to D^b(\Hilb^n S).
\]
For any parabolic $S$ and $\lambda\vDash n$, we have an auto-equivalence $F = F_\lambda: D^b_{\mathfrak{S}_\lambda}(S^n) \to D^b_{\mathfrak{S}_\lambda}(S^n)$, which is essentially induced by the Poincaré bundle on $E\times E^\vee$.

\begin{ansatz}\label{an:Hilb}
    Under the isomorphism in~\cref{ssec:par-Higgs}, we have $\mathfrak{L}_{GL_n} = \Psi\circ F\circ \Phi$, and $\mathsf{Eis}_{GL_\lambda}^{GL_n} = \Psi^\lambda\circ \Phi_\lambda$.
\end{ansatz}

This ansatz requires some explanation.
First, it is known that $\Hitch_n = \Hilb^n S$ is not self-dual in the sense of~\eqref{eq:DLC-conj} beyond $S = T^*E$.
For instance, the mirror dual of $S_{\Z/2}$ should be the DM-stack $(S_{\Z/2})/{(\Z/2)^2}$, which looks like the blow-down of three out of four non-central $\mathbb{P}^1$'s in $S_{\Z/2}$.
The authors do not know the correct definition of the dual space $\Hitch_n^\vee$ in general.
Nevertheless, we presume that there is a ``correction'' equivalence $\mathfrak{I}:D^b(\Hitch_n^\vee)\to D^b(\Hitch_n)$ of Bridgeland--King--Reid type, trivial over the generic locus, such that the composition $\mathfrak{I}\circ \mathfrak{L}$ is the self-equivalence above.
\cref{an:Hilb} is hence formulated for this composition.

Next, the definition of $\mathfrak{L}_{GL_n}$ is chosen to satisfy some desirable properties parallel to Fourier--Mukai transform on abelian varieties, see e.g.~\cite{bai2025hilbert}.
For the Eisenstein series, consider the following commutative diagram:
\[
    \begin{tikzcd}
        & \mathrm{FHilb}^\lambda S\ar[ld]\ar[d]\ar[rd] &  \\
        S^n\sslash\mathfrak{S}_\lambda & \IHilb^n S\sslash\mathfrak{S}_\lambda\ar[l]\ar[r] & \Hilb^n S \\
        \left[S^n/\mathfrak{S}_\lambda\right]\ar[u] & \left[\IHilb^n S/\mathfrak{S}_\lambda\right]\ar[u]\ar[r]\ar[l] & \Hilb^n S\ar[u,equal]
    \end{tikzcd}
\]
Parabolic induction for Hilbert schemes should arise from some correspondence through nested Hilbert schemes $\mathrm{FHilb}^\lambda S$.
However, there is no natural map from $\mathrm{FHilb}^\lambda S$ to $\prod_i \Hilb^{\lambda_i} S$.
Instead, one observes that the correspondence between $S^n\sslash\mathfrak{S}_\lambda$ and $\Hilb^n S$ factors through $\IHilb^n S\sslash\mathfrak{S}_\lambda$, see e.g.~\cite[Sec.~5.4]{chuang2015parabolic}.
Replacing coarse quotients with stacky quotients, one arrives at our definition of Eisenstein series via the bottom row of the diagram.

\subsection{Procesque bundles}\label{ssec:Procesque}
The isospectral Hilbert scheme $\IHilb^n S$ is Cohen--Macaulay by the celebrated result of Haiman~\cite{haiman2001hilbert}.
In particular, miracle flatness implies that the pushforward of any locally free sheaf on $\IHilb^n S$ to $\Hilb^n S$ is locally free.
Given a composition $\lambda\vDash n$, write $S^\lambda \coloneqq \prod_i \Sym^{\lambda_i} S$, $\mathfrak{S}_\lambda \coloneqq \prod_i \mathfrak{S}_{\lambda_i}\subset \mathfrak{S}_n$.
Recall that $\IHilb^n S$ inherits $\mathfrak{S}_n$-action from $S^n$.
For a $\T \times \mathfrak{S}_\lambda$-equivariant vector bundle $\mathcal{V}$ on $S^n$, the pullback $\pi^*(\mathcal{V})$ is also $\T \times \mathfrak{S}_\lambda$-equivariant, and $\mathfrak{S}_\lambda$ acts fiberwise on the pushforward $\rho_*\pi^*(\mathcal{V})$.

\begin{de}
    For $\lambda$, $\mathcal{V}$ as above, define the \textit{modified Procesi} (or \textit{Procesque}) bundle $\mathcal{P}_\lambda(\mathcal{V})$ by
    \[
        \mathcal{P}_\lambda(\mathcal{V}) \coloneqq (\rho_*\pi^*\mathcal{V})^{\mathfrak{S}_\lambda}.
    \]
    Note that $\mathcal{P}_n=\mathcal{P}_{1^n}(\mathcal{O})$ is the usual Procesi bundle on $\Hilb^n(S)$.
\end{de}

\begin{lm}
    $\mathcal{P}_\lambda(\mathcal{V})$ is a vector bundle of rank $\binom{n}{\lambda}\rk(\mathcal{V})$.
\end{lm}
\begin{proof}
    The statement is local in $\Hilb^n S$, so that by \cref{prop:Hilb-etale} we may assume that $S = \C^2$, $\mathcal{V} = \mathcal{O}^{\oplus r}$.
    If $\lambda = (1^n)$, we conclude by miracle flatness.
    By $n!$ theorem~\cite{haiman2001hilbert}, each fiber of $\mathcal{P}_n$ is isomorphic to the regular representation $V_n$ of $\mathfrak{S}_n$.
    Taking $r$ copies of $\mathcal{O}$, the $\mathfrak{S}_\lambda$-invariants are a direct summand of $\rho_*\pi^* \mathcal{O}_{S^n}^{\oplus r}$, hence $\mathcal{P}_\lambda(\mathcal{V})$ is a vector bundle.
    The rank count follows from the fact that $\rk (V_n^{\oplus r})^{\mathfrak{S}_\lambda} = r\binom{n}{\lambda}$.
\end{proof}

We will need explicit formulas for fibers of some Procesque bundles.
Let $\lambda\vDash n$ be a composition of length $\ell(\lambda)=r$, $V_1,\ldots, V_r$ a collection of $\T$-equivariant vector bundles on $S$, and take
\[
    \mathcal{V} \coloneqq V_1^{\boxtimes \lambda_1}\boxtimes \ldots \boxtimes V_r^{\boxtimes \lambda_r}.
\]
This bundle is manifestly $\mathfrak{S}_\lambda$-equivariant.
For $q_1,\ldots,q_s\in S^\T$ and $\underline{\mu}\in \Part^s(n)$, consider the ideal $I^{\underline{\mu}} = \bigcap_i I_{q_i}^{\mu(q_i)}\in \Hilb^n S$ as in~\eqref{eq:fixed-ideals}.
By abuse of notation, we write $\mu_i = |\mu(q_i)|$ and denote the composition $(\mu_1,\ldots, \mu_s)$ by $\mu$.
Let us compute the fiber of Procesque bundle $\mathcal{P}_\lambda(\mathcal{V})$ at $I^{\underline{\mu}}$.
Since the computation is local, we can restrict everything to a neighborhood of the support of $I^{\underline{\mu}}$ in $\Sym^n S$.
We write $\C^2_{q_i} \coloneqq T_{q_i} S$; note that $\T$ acts on $\C^2_{q_i}$ with weights $-l_i, w_i$.
The square of diagram~\eqref{eq:isospec-diag} becomes the following:
\[
    \begin{tikzcd}
        \bigsqcup_{\sigma \in \mathfrak{S}_n/\mathfrak{S}_\mu}\sigma (\prod_{j} \IHilb^{\mu_j}(\C^2_{q_j}))\ar[r,"\rho"]\ar[d,"\pi"] & \prod_j \Hilb^{\mu_j} \C^2_{q_j}\ar[d] \\
        \bigsqcup_{\sigma \in \mathfrak{S}_n/\mathfrak{S}_\mu} \sigma(\prod_{j} (\C^2_{q_j})^{\mu_j}) \ar[r] & \prod_j \Sym^{\mu_j} \C^2_{q_j}
    \end{tikzcd}
\]
As $\mathcal{V}$ is only $\mathfrak{S}_\lambda$-equivariant, it is convenient to rewrite the unions via double cosets.
Recall that we have a bijection between $\mathfrak{S}_\lambda\backslash\mathfrak{S}_n/\mathfrak{S}_\mu$ and the set $\Theta = \Theta(\lambda,\mu)$ of all $r\times s$ matrices $(a_{ij})$ satisfying $\sum_j a_{ij} = \lambda_i$, $\sum_i a_{ij} = \mu_j$ for all $i,j$.
In these notations, for any element $\sigma\in \mathfrak{S}_\lambda\backslash\mathfrak{S}_n/\mathfrak{S}_\mu$ we have $\mathfrak{S}_\lambda \cap \sigma\mathfrak{S}_\mu\sigma^{-1} \simeq \mathfrak{S}_A = \prod_{i,j} \mathfrak{S}_{a_{ij}}$, where $A\in \Theta$ is the matrix corresponding to $\sigma = \sigma_A$.
Then the map $\pi$ in the diagram above rewrites as follows:
\[
    \pi: \bigsqcup_{A \in \Theta} \mathfrak{S}_\lambda \times_{\mathfrak{S}_A} \sigma_A(\prod_{j} \IHilb^{\mu_j}(\C^2_{q_j})) \to \bigsqcup_{A \in \Theta} \mathfrak{S}_\lambda \times_{\mathfrak{S}_A} \sigma_A(\prod_{j} (\C^2_{q_j})^{\mu_j}).
\]
Since we are computing an $\mathfrak{S}_\lambda$-invariant quantity, we can get rid of all $\mathfrak{S}_\lambda$'s by considering $\mathcal{V}$ as an $\mathfrak{S}_A$-equivariant bundle on each component of the union.
Denoting the restriction of $\mathcal{V}$ to the origin of $\sigma_A(\prod_{j} (\C^2_{q_j})^{\mu_j})$ by $V_A$, we get 
\[
    \mathcal{P}_\lambda(\mathcal{V})|_{I^{\underline{\mu}}}
    = \sum_{A\in\Theta} \prod_j \left(V_A\otimes \operatorname{Res}_{\mathfrak{S}_{A_j}}^{\mathfrak{S}_{\mu_j}}\mathcal{P}_{\mu_j}|_{I_{q_j}^{\mu(q_j)}}\right)^{\mathfrak{S}_{A_j}},
\]
where $A_j = (a_{1j},\ldots, a_{rj})$ is a composition of $\mu_j$.

Now, let us further assume that $\mu(q_j) = (1^{\mu_j})$ for all $j$.
It is known (see e.g.~\cite[Prop.~3.5.8]{haiman2002combinatorics}) that the fiber of Procesi bundle at this point is the coinvariant ring with grading parameter $t^{w_j}$.
Let us denote it by $R_{\mu_j}(t^{w_j})$.
In particular, we have 
\[
    \operatorname{Res}_{\mathfrak{S}_{A_j}}^{\mathfrak{S}_{\mu_j}}\mathcal{P}_{\mu_j}|_{I_{q_j}^{(1^{\mu_j})}} = \qnom{\mu_j}{A_j}_{t^{w_j}} \bigotimes_i R_{a_{ij}}(t^{w_j}).
\]
Let us denote $V_{ij} = V_i|_{q_j}$.
Remembering the definition of $\mathcal{V}$, we can further factorize:
\begin{equation}\label{eq:Procesque-fiber-general}
    \chi_t(\mathcal{P}_\lambda(\mathcal{V})|_{I^{\underline{\mu}}})
    = \sum_{A\in\Theta} \prod_j\qnom{\mu_j}{A_j}_{t^{w_j}} \prod_{i,j} \chi_t((V_{ij}^{a_{ij}}\otimes R_{a_{ij}}(t^{w_j}))^{\mathfrak{S}_{a_{ij}}}).
\end{equation}

\begin{lm}\label{lm:renormalization}
    Let $V$ be a $\T$-vector space with $\chi_t(V) = \sum_{i=1}^c t^{a_i}$, $a_i\in\mathbb{N}_{\geq 0}$.
    Then 
    \[
        \chi_t\left((V^{\otimes n}\otimes R_n(t^{b}))^{\mathfrak{S}_n}\right) = \sum_{\lambda \vDash n, \ell(\lambda) = c} t^{\sum_i a_i\lambda_i} \qnom{n}{\lambda}_{t^b}.
    \]
\end{lm}
\begin{proof}
    Replace representations of $\mathfrak{S}_n$ with their characters, and use the notations of Macdonald's monograph~\cite[Ch.~I.7]{macdonald1998symmetric}.
    By Schur--Weyl duality, $V^{\otimes n}$ gets replaced with $\sum_{\lambda \in \Part(n)} \mathbb{S}_\lambda(V)s_\lambda(X)$, where $\mathbb{S}_\lambda$ is the Schur functor, and $s_\lambda(X)$ the Schur polynomial.
    We have\footnote{We are grateful to A. Mellit for providing this argument.}
    \begin{align*}
        \chi_t\left((V^{\otimes n}\otimes R_n(t^{b}))^{\mathfrak{S}_n}\right)
        & = \chi_t(\Hom_{\mathfrak{S}_n}(R_n(t^{b}),V^{\otimes n}))\\
        & = \left\langle  \frac{h_n[X/(1-t^b)]}{h_n[1/(1-t^b)]}, \sum_{\lambda \in \Part(n)} s_\lambda[\chi_t(V)]s_\lambda[X]  \right\rangle \\ 
        & = \left( \sum_{\lambda \in \Part(n)} \langle h_n[X/(1-t^b)],s_\lambda[X]  \rangle s_\lambda[\chi_t(V)] \right) (1-t^b)^{n}[n]!_{t^b}\\
        & = h_n[\chi_t(V)/(1-t^b)](1-t^b)^{n}[n]!_{t^b},
    \end{align*}
    where in the last line we used the equality $\sum_\lambda \langle f[XY], s_\lambda[Y] \rangle s_\lambda[Z] = f[XZ]$, which holds for any symmetric function $f$.
    Using the properties of complete symmetric functions $h_n$ under plethystic substitution, we further obtain
    \begin{align*}
        h_n[\chi_t(V)/(1-t^b)]
        & = h_n\left[\sum\nolimits_{i=1}^c \frac{t^{a_i}}{1-t^b}\right]
        = \sum_{\lambda \vDash n, \ell(\lambda) = c} h_{\lambda_1}\left[\frac{t^{a_1}}{1-t^b}\right]\ldots h_{\lambda_c}\left[\frac{t^{a_c}}{1-t^b}\right]\\ 
        & = \frac{1}{(1-t^b)^n}\sum_{\lambda \vDash n, \ell(\lambda) = c} t^{\sum_i a_i\lambda_i}\frac{1}{\prod_{i=1}^c [\lambda_i]!_{t^b}}.
    \end{align*}
    Substituting this into the previous expression, we arrive to the desired formula.
\end{proof}

\begin{cor}\label{cor:nice-renormalization}
    Let $V$ be a $\T$-vector space with $\chi_t(V) = c[b]_{t^a}$, where $a,b,c\in\mathbb{Z}_{\geq 0}$.
    Then 
    \[
        \chi_t\left((V^{\otimes n}\otimes R_n(t^{ab}))^{\mathfrak{S}_n}\right) = [b]_{t^a}^{(n)} G_{n,c}(t^a), \qquad G_{n,c}(t)\coloneqq \sum_{\lambda \vDash n, \ell(\lambda) = c} \qnom{n}{\lambda}_{t},
    \]
    where $[b]_{t^a}^{(n)}$ is defined by~\eqref{eq:t-powers}.
\end{cor}
\begin{proof}
    Substituting $t'\mapsto t^a$, we can assume $a=1$.
    As $\chi_t(V) = c\frac{1-t^b}{1-t}$, we have $h_n[\chi_t(V)/(1-t^b)] = h_n[c/(1-t)]$.
    Running the same computation as in \cref{lm:renormalization}, we arrive at
    \begin{align*}
        \chi_t\left((V^{\otimes n}\otimes R_n(t^{b}))^{\mathfrak{S}_n}\right) & = h_n[c/(1-t)](1-t^b)^{n}[n]!_{t^b} \\
        & = \left(\frac{1}{(1-t)^n}\sum_{\lambda \vDash n, \ell(\lambda) = c} \frac{1}{\prod_{i=1}^c [\lambda_i]!_{t}}\right) [b]_{t}^{(n)}(1-t)^n[n]!_{t} = [b]_{t}^{(n)}G_{n,c}(t),
    \end{align*}
    so we may conclude.
\end{proof}

\subsection{Duality checks}\label{ssec:dual-checks}
From now on, let us write $\Hitch_n = \Hilb^n S$.
Recall the tautological vector bundles $V_p$ on $S$ from \cref{ssec:par-Higgs}, defined for very stable $p\in S^\T$.

\begin{de}\label{de:nice-Procesque}
    Let $I = \cap_i I_{p_i}^{(1^{\lambda_i})}$, where $\lambda\vDash n$ a composition of length $r$, and $\underline{p} = (p_1,\ldots, p_r)$ a tuple of very stable points in $S^\T$.
    Define
    \[
        \mathcal{P}_I \coloneqq \mathcal{P}_{\lambda}(\boxtimes_i V_{p_i}^{\boxtimes \lambda_i})\in \mathrm{Coh}(\Hitch_n).
    \]
\end{de}

Let us further denote $\mathcal{W}_I\coloneqq \mathcal{O}_{W^+_I}$.
The main claim of this section is:
\begin{equation*}\tag{MS}\label{eq:MS-claim}
    \text{For a very stable $I$, $\mathcal{W}_I$ is exchanged with $\mathcal{P}_I$ under mirror duality in \cref{an:Hilb}.}
\end{equation*}

Let us explain how this prediction is obtained.
For $S = T^*E$, $n=1$ it is essentially equivalent to \cref{an:FM}.
For $S = S_{\Z/e}$, $e>1$, the sheaves $V_p$ and $\mathcal{O}_{W^+_p}$ are matched via derived McKay equivalence $D^b(S_{\Z/e})\simeq D^b_{\Z/e}(T^*E)$.
Finally, for $n>1$ the sheaf $\mathcal{P}_I$, resp. $\mathcal{W}_I$ is obtained by applying Eisenstein series from \cref{an:Hilb} to the product of $V_{p}$'s, resp. $\mathcal{O}_{T_p^+}$'s.
By \cref{an:FM,an:Eis} the two should therefore be exchanged.

As in~\cite{hausel2022very}, we justify the claim~\eqref{eq:MS-claim} by performing two checks.
First, let $\Base^\#\subset \Base$ be the open of $n$-tuples $(x_1,\ldots,x_n)\in \Sym^n\C$, such that $x_i^{ne+1}\neq x_i$ for all $i$, and $x_i\neq x_j$ for all $i<j$.
Write further $\Hitch_n^\#\coloneqq \theta^{-1}(\Base^\#)$.
Note that over each point of $\Base^\#$, the ideals are necessarily supported at $n$ points, therefore $\Hitch_n^\#\simeq \Base^\# \times E^n$.
For $n=1$, we also write $S^\# \simeq \C^*\setminus \{x^e=1\} \times E\subset S$; clearly $\Base^\# = \Sym^n(\C^*\setminus \{x^e=1\})$.
Denote by $\operatorname{FM}_1$ the Fourier--Mukai transform on $S^\#$ along $E$.
It is known that the restriction of $\mathfrak{L}'$ to $\Base^\#$ coincides with $\operatorname{FM}_n \coloneqq \operatorname{FM}_1\boxtimes\ldots \boxtimes\operatorname{FM}_1$.
\begin{prop}
    Let $I$ be as in \cref{de:nice-Procesque}, and let $\mathcal{W}_I^\#$, resp. $\mathcal{P}_I^\#$ be the restriction of $\mathcal{W}_I$, resp. $\mathcal{P}_I$ to $\Base^\#$.
    We have $\operatorname{FM}_n(\mathcal{W}_I^\#) = \mathcal{P}_I^\#$.
\end{prop}
\begin{proof}
    From the definition of $\mathcal{P}_I$, we have $\mathcal{P}_I^\# = \boxtimes_{i} (\mathcal{P}_{I_{p_i}^{(1)}}^\#)^{\boxtimes \lambda_i}$.
    Futhermore, we have $W^+_I = \sqcup_i W^+_{I_{p_i}^{(1^{\lambda_i})}}$.
    Analogously to the proof of \cref{prop:ell-very-stable}, we can show the equality $W^+_{I_{p}^{(1^l)}} = \Hilb^{l}(T^+_p)$, induced by the inclusion $T^+_p\subset S$.
    In particular, $\theta^{-1}(x_1,\ldots,x_l)\cap W^+_{I_{p}^{(1^l)}} = \prod_{j=1}^l \theta^{-1}(x_j)$, and so $\mathcal{W}_I^\# = \boxtimes_{i} (\mathcal{W}_{I_{p_i}^{(1)}}^\#)^{\boxtimes \lambda_i}$.
    Since $\operatorname{FM}_n$ factorizes as well, it suffices to check that $\operatorname{FM}_n(\mathcal{W}_I^\#) = \mathcal{P}_I^\#$ for $n=1$.

    For $n=1$, let us fix $x\in \C^*\setminus \{x^e=1\}$ and consider $E_x\coloneqq \theta^{-1}(x)\subset S$.
    Denote the projection $E_x\simeq E \to C$ by $\tau$.
    Recall the definition of vector bundles $V_{p_i}$ in \cref{ssec:par-Higgs}.
    It is easy to check from that description that $V_p|_{E_x} = \otimes_{q\in \tau^{-1}(p)}L_q$.
    On the other hand, $\mathcal{W}_p|_{E_x} = \mathcal{O}_{\sqcup_{q\in \tau^{-1}(p)} q} = \bigoplus_{q\in \tau^{-1}(p)}\mathcal{O}_{q}$.
    By properties of Fourier--Mukai transform on $E$ we have $\operatorname{FM}_1(\mathcal{O}_{q}) = L_q$, so that we proved the desired isomorphisms fiberwise.
    The global isomorphism follows by seesaw principle.
\end{proof}

For the second check, let us define the equivariant Euler pairing of two coherent sheaves $\mathcal{F}_1,\mathcal{F}_2$ on $\Hitch_n$ by
\[
    \chi_t(\mathcal{F}_1,\mathcal{F}_2) \coloneqq \sum (-1)^k \chi_t(\mathrm{Ext}^k(\mathcal{F}_1,\mathcal{F}_2)).
\]
This pairing must be respected by mirror duality.
Let $I = \cap_i I_{p_i}^{(1^{\lambda_i})}$ as above, and similarly $J = \cap_i J_{q_i}^{(1^{\mu_i})}$, where $\mu\vDash n$ is a composition of length $s$, and $\underline{q} = (q_1,\ldots, q_s)$ a tuple of very stable points in $S^\T$.

\begin{thm}\label{thm:eq-index}
    In the notations above, we have $\chi_t(\mathcal{W}^\vee_I,\mathcal{P}_J) = \chi_t(\mathcal{P}^\vee_I,\mathcal{W}_J)$.
\end{thm}
\begin{rmk}
    One can get rid of duals as explained in~\cite[Rem.~6.12]{hausel2022very}, for the price of multiplying r.h.s by some sign and a power of $t$.
    It means that~\eqref{eq:MS-claim} holds only up to a homological shift and a $\T$-weight, which we do not bother to specify.
\end{rmk}
\begin{proof}[{Proof of \cref{thm:eq-index}}]
    Let us denote the positive tangent weight of $q_i$ by $w_i$.
    By \cref{prop:vst-eq-mult-parab}, the equivariant multiplicity of the core component containing $J$ is given by
    \[
        m_J(t) = \qnom{n}{\mu}_{t^e} \prod_i ([e/w_i]_{t^{w_i}})^{(\mu_i)} = \frac{[n]_{t^e}!}{\prod_i [\mu_i]_{t^w}!}.
    \]
    By the exact same argument as in the proof of~\cite[Th.~6.9]{hausel2022very}, we obtain the following identities:
    \begin{align*}
        \chi_t(\mathcal{W}^\vee_I,\mathcal{P}_J) = \chi_t(\mathcal{P}_J|_I) m_I(t)\chi_t(\Sym \Base), \qquad 
        \chi_t(\mathcal{P}^\vee_I,\mathcal{W}_J) = \chi_t(\mathcal{P}_I|_J) m_J(t)\chi_t(\Sym \Base).
    \end{align*}
    Thus, it is enough to show that the expression $\chi_t(\mathcal{P}_I|_J) m_I(t)$ is symmetric under the exchange of $I$ and $J$.
    Let us compute it using the formula~\eqref{eq:Procesque-fiber-general}.
    Denote $V_{ij}(t) = \chi_t(V_{p_i}|_{q_j})$, and $m_{j}(t) = m_{q_j}(t)$. 
    Observe that $V_{ij}(t) = c_{ij}[b_{ij}]_{w_j/b_{ij}}$ for some $b_{ij},c_{ij}>0$, and furthermore $m_j(t) = [e/w_j]_{t^{w_j}}$.
    We have:
    \begin{align*}
        \chi_t(\mathcal{P}_I|_J) m_J(t)
        & = \frac{[n]_{t^e}!}{\prod_j [\mu_j]_{t^{w_j}}!} \sum_{A\in \Theta(\lambda,\mu)}  \prod_j\qnom{\mu_j}{A_j}_{t^{w_j}} \prod_{i,j} \chi_t((V_{ij}^{a_{ij}}\otimes R_{a_{ij}}(t^{w_j}))^{\mathfrak{S}_{a_{ij}}})\\
        & = \sum_{A\in \Theta(\lambda,\mu)} \frac{[n]_{t^e}!}{\prod_{i,j} [a_{i,j}]_{t^{w_j}}!} \prod_{i,j} [b_{ij}]_{t^{w_j/b_{ij}}}^{(a_{ij})}G_{a_{ij},c_{ij}}(t^{w_{j}/b_{ij}})\\
        & = \sum_{A\in \Theta(\lambda,\mu)} \qnom{n}{A}_{t^e} \prod_{i,j} (m_{j}(t)[b_{ij}]_{t^{w_j/b_{ij}}})^{(a_{ij})}G_{a_{ij},c_{ij}}(t^{w_{j}/b_{ij}})\\
        & = \sum_{A\in \Theta(\lambda,\mu)} \qnom{n}{A}_{t^e} \prod_{i,j} ([eb_{ij}/w_j]_{t^{w_j/b_{ij}}})^{(a_{ij})}G_{a_{ij},c_{ij}}(t^{w_{j}/b_{ij}}).
    \end{align*}
    Here, we used \cref{cor:nice-renormalization} for the second equality, whose conditions one can check directly from \cref{table:fibers}.
    In the same way, one directly checks that the matrix $b_{ij}w_i = b_{ji}w_j$ for all $i,j$, and $c_{ij} = c_{ji}$.
    Hence the expression we obtained is symmetric under swapping $I$ with $J$, so we may conclude.
\end{proof}

\begin{rmk}
    Let $S = T^*E$.
    In this case, the products in the big sum above vanish, so that
    \[
        \chi_t(\mathcal{P}_I|_J) m_J(t) = \sum_{A\in \Theta(\lambda,\mu)} \qnom{n}{A}_t.
    \]
    For $t=1$, r.h.s specializes to $\binom{n}{\lambda}\binom{n}{\mu}$ by a classical combinatorial formula.
    In stark contrast to~\cite{hausel2022very}, this is not true after deformation; for example
    \[
        \qnom{3}{1}_t \qnom{3}{1}_t = [3]_t^2 \neq [3]_t(1 + [2]_t) = \qnom{3}{1}_t + \qnom{3}{1,1,1}_t.
    \]
    Instead, after some manipulation with symmetric functions the r.h.s can be expressed in terms of Kostka polynomials; namely, we get the formula~\cite[(5.20)]{chuang2015parabolic} specialized at $q_1=1$, $q_2=t$.
\end{rmk}

\subsection{Wobbly components}\label{ssec:MS-beyond-vst}
Let $I = \cap_i I_{p_i}^{\lambda(p_i)}\in \Hitch_n^\T$ be a wobbly ideal.
For any wobbly point $q\in S$, we have a tautological bundle $V_q$ as in \cref{ssec:par-Higgs}, where we pick the corresponding step in the parabolic flag instead of the first one.
We can define a Procesque bundle $\mathcal{P}_I$ by the same formula as \cref{de:nice-Procesque}, except that each $V^{\boxtimes \lambda_i}_{p_i}$ gets symmetrized by $\mathfrak{S}_{\lambda(p_i)'}$ instead of $\mathfrak{S}_{\lambda_i}$.
Unfortunately, $\mathcal{O}_{\overline{W^+_I}}$ is not mirror dual to $\mathcal{P}_I$ in general.
Indeed, already for $S_{\Z/e}$, $e>1$ some upward flows are fully contained in the core, while the Procesque bundles are always supported over the whole base.
Following~\cite{hausel2022enhanced}, we define the Lagrangian closure $W^\ggcurly_I$ as the smallest closed subset of $\Hitch_n$ which contains $I$, and such that for any $I'\in (W^\ggcurly_I)^\T$ we have $W^+_{I'}\subset W^\ggcurly_I$.

\begin{con}\label{conj:support}
    For any $I\in \Hitch_n^\T$, the mirror dual $\mathcal{W}_I$ of $\mathcal{P}_I$ is supported on $W^\ggcurly_I$.
\end{con}
It would be interesting to compute the equivariant multiplicities of $\mathcal{W}_I$ along the irreducible components of $W^\ggcurly_I$; we hope to return to this question in the future.

\begin{rmk}\label{rem:dual-Procesi}
    Let $\Hitch_n = \Hilb^n T^*E$.
    A particular case of the conjecture above is that the mirror of the usual Procesi bundle $\mathcal{P}_n\in D^b(\Hilb^n T^*E)$ is supported on $W^\ggcurly_{I_{\mathrm{e}}^{(n)}}$.
    Assuming \cref{an:Hilb}, the bundle $\mathcal{P}_n$ is obtained by applying $\mathsf{Eis}$ to the sheaf $\mathcal{O}_{(T^*E)^n}$.
    By \cref{an:FM,an:Eis}, the dual of $\mathcal{P}_n$ is then the image of $\mathcal{O}_{(T^*_{\mathrm{e}})^n}$ under $\mathsf{Eis}$.
    We deduce a prediction:
    \[
        \mathcal{P}_n^\vee = \rho_*\mathcal{O}_{X_n}, \qquad {X_n}\coloneqq \IHilb^n T^*E\times_{T^*E^n}(T^*_{\mathrm{e}})^n,
    \]
    where $\rho: \IHilb^n T^*E\to \Hilb^n T^*E$ is the projection.
    Note that $\rho(X_n)$ is (set-theoretically) equal to the preimage of $\mathrm{e}^n$ under the projection $\Hilb^n T^*E\to \Sym^n E$. 
    It is straightforward to check that this is precisely the Lagrangian closure $W^\ggcurly_{I_{\mathrm{e}}^{(n)}}$.
    The exact same argument works for $I_{\mathrm{e}}^{(n)}\in \Hilb^n S$, where $S$ is of parabolic type.
    Similar predictions can be made for other wobbly ideals.
\end{rmk}

Assuming the matching of equivariant pairings, we can compute $\chi_t(\theta_*\mathcal{W}_I)$. 
Let $\mathrm{e}\in S$ be the point corresponding to the identity of $E$, and write $J_0 = I_{\mathrm{e}}^{(1^n)}$.
Note that $\mathcal{P}_{J_0} = \mathcal{O}$. Hence
\[
    \chi_t(\theta_*\mathcal{W}_I) = \chi_t(\mathcal{W}_I^\vee,\mathcal{O}) = \chi_t(\mathcal{P}_I^\vee,\mathcal{W}_{J_0}) = \chi_t(\mathcal{P}_I|_{J_0})\chi_t(\Sym\Base).
\]
The fibers of $\mathcal{P}_I|_{J_0}$ can be computed similarly to the proof of \cref{thm:eq-index}.
As the formula~\eqref{eq:Procesque-fiber-general} still applies, we have
\begin{align*}
    \chi_t(\mathcal{P}_I|_{J_0}) &= \chi_t(\mathcal{P}_I|_{J_0})m_{J_0}(t)
    = \qnom{n}{\lambda}_{t^e} \prod_i \chi_t(((V_{p_i}|_{\mathrm{e}})^{\otimes \lambda_i}\otimes R_{\lambda_i}(t^{e}))^{\mathfrak{S}_{\lambda(p_i)'}}) \\
    & = \qnom{n}{\{\lambda(p_i)'\}_i}_{t^e} \prod_{i,l} \chi_t(((V_{p_i}|_{\mathrm{e}})^{\otimes \lambda(p_i)_l}\otimes R_{\lambda(p_i)_l}(t^{e}))^{\mathfrak{S}_{\lambda(p_i)'_i}}).
\end{align*}
Let us denote the r.h.s of the equation in \cref{lm:renormalization} with $b=e$ by $G_{n,V}(t)$.
We finally obtain
\begin{equation}\label{eq:fiber-any-Procesque}
    \chi_t(\mathcal{P}_I|_{J_0}) = \qnom{n}{\{\lambda(p_i)'\}_i}_{t^e} \prod_{i,l} G_{\lambda(p_i)'_l,V_{p_i}|_{\mathrm{e}}}(t).
\end{equation}
Let $n=1,2$, and recall the formulas for equivariant multiplicities $m_I(t)$ computed in \cref{ssec:2d-mult-ex,sec:ex-Hilb2}.
It is easy to check, as with values in \cref{table:fibers}, that for a wobbly point $p\in S$ we have $V_{p}|_{\mathrm{e}} = [m_p]_t$, unless $p\in S_{\Z/6}$ with local equation $x^2y^4$, in which case $V_{p}|_{\mathrm{e}} = 1+t+t^3+t^4$.
The following claim is then proved by a direct comparison.
\begin{prop}\label{prop:chi-separates}
    Denote the expression in~\eqref{eq:fiber-any-Procesque} by $\chi_I(t)$.
    For $n=1,2$, we have
    \[
        \mu_I(t) \leq \chi_I(t) \leq m_I(t) \text{ for all } t>1,
    \]
    and the equalities hold if and only if $I$ belongs to a very stable $F\in \Fix$. \qed
\end{prop}

\subsection{A reckless conjecture}\label{ssec:main-conj}
Let us draw some speculative lessons from \cref{prop:chi-separates}.
Once again, we assume that $I\in F$, $F\in \Fix$ is wobbly.
If we compute $\chi_t(\theta_*\mathcal{W}_I)/\chi_t(\Sym\Base)$ by equivariant localization, then $\mu_I(t)$ is roughly the contribution of the fixed point $I$.
In this light, the inequality $\mu_I(t) < \chi_I(t)$ is reflective of the fact that $W^\ggcurly_I$ has more than one fixed point.
On the other hand, the inequality $\chi_I(t) \leq m_I(t)$ suggests that even though $m_I(t)$ does not have positive coefficients, it might be that $\sp_{W^-_F}(\mathcal{O}_{\mathcal{C}})$ has a locally free subsheaf, whose fiber at $I$ has Hilbert series $\chi_I(t)$.
As $\chi_I(1) = m_I$, these subsheaves should in some sense cover the whole of $\mathcal{O}_{\mathcal{C}}$.

Now let $\theta:\Hitch\to \Base$ be a smooth integrable system of weight $1$; in particular, it satisfies~\eqref{eq:ab-cond}.
Denote the origin $0\in \Base$.
Assume that $\theta$ is equipped with a ``Kostant section'' $\kappa:\Base\to \Hitch$, $\theta\circ \kappa = \mathrm{Id}_\Base$. 
We also assume that $\theta$ admits a  ``mirror dual'' system $\theta^\vee:\Hitch^\vee\to \Base$, in the sense that generically on $\Base$ these are dual abelian fibrations, and there exists a Poincaré complex $\mathcal{P}$ on $\Hitch\times_\Base \Hitch^\vee$ inducing an equivalence $D^b(\Hitch)\to D^b(\Hitch^\vee)$.
Recall the notations of \cref{subs:Flow order}.

\begin{con}\label{conj:reckless}
    There exists a collection $\{\mathcal{W}_p\}_{p\in \Hitch^\T}$ of sheaves $\mathcal{W}_p\in D^b_\T(\Hitch)$, and a collection of non-reduced $\T$-schemes $\{\mathcal{Z}_F\}_{F\in \Fix}$, satisfying the following properties:
    \begin{enumerate}[leftmargin=8mm]
        \item For each $F\in \Fix$, we have $(\mathcal{Z}_F)_\mathrm{red} = W^\llcurly_F$. For any $F_1,F_2\in \Fix$ with $F_1\preceq F_2$ we have a ``blow down'' map $\pi_{F_1}^{F_2}:\mathcal{Z}_{F_2} \to \mathcal{Z}_{F_1}$, and maps $\pi_F: \Core\to \mathcal{Z}_F$ satisfying $\pi_{F_1}^{F_2}\circ\pi_{F_2} = \pi_{F_1}$. The intersection $\bigcap_F \ker(\pi_F)$ is trivial. 
        The specialization $\sp_{W^-_F}(\mathcal{Z}_F)$ is locally free;
        \item For each $p\in \Hitch^\T$, $\mathcal{W}_p$ is supported on the Lagrangian closure $W^\ggcurly_p$, and $\mathcal{W}_p|_{W^+_p}\simeq \mathcal{O}_{W^+_p}$. Its mirror dual is a locally free sheaf $\mathcal{V}_p$ on $\Hitch^\vee$. For any $p_1,p_2\in \Hitch^\T$ and a $\T$-curve from $p_1$ to $p_2$, we have an inclusion $\mathcal{W}_{p_2}\subset \mathcal{W}_{p_1}$;
        \item For $p\in F$, we have $m(W^-_F;\mathcal{O}_{\mathcal{Z}_F}) = \chi_t(\mathcal{V}_p|_{\kappa^\vee(0)}) = \chi_t(\theta_*\mathcal{W}_p^+|_0)$. Furthermore, the value of this polynomial at $t=1$ is the multiplicity of $\Core_F$.
    \end{enumerate}
\end{con}
Note that the schemes $\mathcal{Z}_F$ are not subschemes of $\Hitch$, so that there is no immediate way to ask what their mirror duals are.
This conjecture holds for 2-dimensional integrable systems; we leave the proof as an illuminating exercise for the reader.
For instance, at a wobbly point $p$ with local equation $\mathscr{E}(x,y)=x^{a}y^{a+1}$, the map $\pi_p$ is induced by the inclusion $\C[x,xy]/(xy)^{a+1}\subset \C[x,y]/x^{a}y^{a+1}$.
It is unclear whether one can reasonably expect the conjecture to hold beyond Hilbert schemes, or indeed for Hilbert schemes, as we have not checked it even for $\Hilb^2 S$.

\begingroup
\setstretch{0.88}
\bibliography{refs}
\bibliographystyle{amsalpha}
\endgroup
\end{document}